\documentclass[11pt]{article} 
			 		
\usepackage[utf8]{inputenc} 

\usepackage{geometry} 
\usepackage{graphicx} 

\usepackage{booktabs} 
\usepackage{array} 
\usepackage{paralist} 
\usepackage{amsmath}
\usepackage{amssymb}
\usepackage{amsthm}
\usepackage{xifthen}
\usepackage{xcolor}
\usepackage{ulem}
\usepackage{hyperref}
\hypersetup{
    colorlinks=true,
    linkcolor=red,
    filecolor=magenta,      
    urlcolor=cyan,
    citecolor=red,
    pdftitle={Overleaf Example},
    pdfpagemode=FullScreen,
}

\usepackage{subfigure}

\usepackage{soul}

\usepackage[commentmarkup=footnote]{changes}
\definechangesauthor[name={Andi}, color=red]{and}
\definechangesauthor[name={Kristoffer}, color=blue]{kri}
\definechangesauthor[name={Neo}, color=purple]{neo}

\usepackage{fancyhdr} 
\usepackage{sectsty}
\allsectionsfont{\sffamily\mdseries\upshape} 

\newtheorem{theorem}{Theorem}[section]
\newtheorem*{acknowledgement*}{Acknowledgement}

\newtheorem{assumption}[theorem]{Assumption}

\newtheorem{corollary}[theorem]{Corollary}

\newtheorem{definition}[theorem]{Definition}

\newtheorem{lemma}[theorem]{Lemma}

\newtheorem{proposition}[theorem]{Proposition}
\newtheorem{remark}[theorem]{Remark}

\usepackage[nottoc,notlof,notlot]{tocbibind} 
\usepackage[titles,subfigure]{tocloft} 

\def\F{\mathcal F}

\def\R{\mathbb{R}}

\def\N{\mathbb{N}}

\def\Pr{\mathbb{P}}
\def\Af{\mathbf{A}}
\renewcommand{\phi}{\varphi} 
\newcommand{\tild}{\widetilde} 
\def\ub{\underline{b}}

\DeclareMathOperator{\E}{\mathbb E}

\newcommand{\kom}[1]{}
\renewcommand{\kom}[1]{{\bf [#1]}}

\usepackage{accents}

\newcounter{komcounter}
\numberwithin{komcounter}{section}

\renewcommand{\ul}{\underline}
\newcommand{\ol}{\overline}
 
\renewcommand{\epsilon}{\varepsilon} 

\usepackage{comment}

\title{Time-inconsistent singular control problems: \\ Reflection and Absolutely continuous controls with exploding rates}
\author{
Andi Bodnariu \\Department of Mathematics, Stockholm University\\
\\Kristoffer Lindensjö\\Department of Mathematics, Stockholm University\\
\\Neofytos Rodosthenous\\Department of Mathematics, University College London
}
 
\begin{document}

\maketitle

\begin{abstract}
We study a time-inconsistent singular stochastic control problem for a general one-dimensional diffusion, where time-inconsistency arises from a non-exponential discount function. To address this, we adopt a game-theoretic framework and study the optimality of a novel class of controls that encompasses both traditional singular controls -- responsible for generating multiple jumps and reflective boundaries ({\it strong thresholds}) -- and new {\it mild threshold control strategies}, which allow for the explosion of the control rate in absolutely continuous controls, thereby creating an inaccessible boundary ({\it mild threshold}) for the controlled process. 
We establish a general verification theorem, formulated in terms of a system of variational inequalities, that provides both necessary and sufficient conditions for equilibrium within the proposed class of control strategies and their combinations. To demonstrate the applicability of our theoretical results, we examine case studies in inventory management. We show that for certain parameter values, the problem admits a strong threshold control equilibrium in the form of Skorokhod reflection. 
In contrast, for other parameter values, we prove that no such equilibrium exists, necessitating the use of our extended control class.  
In the latter case, we explicitly construct an equilibrium using a mild threshold control strategy with a discontinuous, increasing, and exploding rate that induces an inaccessible boundary for the optimally controlled process, marking the first example of a singular control problem with such a  solution structure.
\end{abstract}

\section{Introduction} \label{sec:intro}

In this paper, we study a time-inconsistent singular stochastic control problem, where the time-inconsistency arises from the use of a non-exponential discount function.
The uncontrolled state process is a general one-dimensional diffusion, and the controller seeks to minimise an expected cost functional consisting of a non-decreasing running cost and a terminal cost, accumulated until possible absorption at a lower boundary of the state space. Control is exerted by pushing the process downward, incurring a proportional (unit) cost and thereby introducing a natural trade-off between state persistence and intervention.
The precise formulation of the problem is given in Section \ref{sec:model}.
As a motivating application, and the subject of our case studies, we consider a time-inconsistent version of the classical inventory control problem, which fits naturally into this framework (see Sections \ref{sec:applications} and \ref{sec:mixopt} for a detailed analysis).

To address the time-inconsistency inherent in our problem, we adopt a game-theoretic approach.
Informally, a control problem is said to be time-inconsistent if a strategy that appears optimal from the perspective of a given initial time and state may no longer remain optimal when re-evaluated at later times and states -- that is, Bellman’s principle of optimality fails to hold (see Remark \ref{rem:egBM}).
This perspective dates back to the seminal work of Strotz \cite{strotz} in the 1950s, and has since inspired a growing mathematical literature, with early foundational contributions including \cite{tomas-continpubl,bjork2014theory}.
The core idea is to view the control problem as an intrapersonal dynamic game, where the players are a  continuum of the decision-maker's selves -- each corresponding to a future incarnation of the same individual -- who may have differing preferences due to time-inconsistency.
In our context, this translates to the possibility of deviating unilaterally from a proposed strategy at any given 
state $x$, albeit only locally in time, as future selves retain control thereafter.
This framework has been explored extensively in recent works (see e.g.~\cite{tomas-continpubl,christensen2018finding, christensen2020time,lindensjo2019regular}), which offer further discussion of this interpretation and other time-inconsistent examples.

Most of the mathematical literature on time-inconsistent problems has focused either on classical control problems -- where players choose regular controls, absolutely continuous with respect to the Lebesgue measure, with state-dependent control rates 
-- or on time-inconsistent stopping problems (see e.g.~\cite{bjork2021time} for a comprehensive overview of this literature).
Several notions of equilibrium have been proposed in the time-inconsistent setting (for comparisons of different equilibrium concepts, see e.g.~\cite{bayraktar2023equilibria} for stopping and \cite{HuangZhou2020} for control problems); the notion of equilibrium adopted in the present paper corresponds to a so-called weak equilibrium.

Time-inconsistent singular stochastic control problems have also received some attention, though to a more limited extent. 
In \cite{christensen2022moment}, a singular dividend problem with a moment constraint on the number of dividends is studied, leading to an equilibrium of impulse control type. 
Similarly, \cite{doi:10.1137/16M1088983} considers a dividend problem under transaction costs and non-exponential discounting, where again impulse-type equilibria are analysed. 
A related setup with a spectrally positive L\'evy process is investigated in \cite{Levytimeinc}, where the time-inconsistency again stems from non-exponential discounting and the resulting equilibrium is again of impulse control type.
Further variants of time-inconsistent dividend problems arise in \cite{CHEN2014150} and \cite{zhu2020singular}, where the discount rate changes discretely over time.
To our knowledge, the work most closesly related to the present paper is \cite{liang2024equilibria}, which studies a time-inconsistent singular control problem under non-exponential discounting. 
A key difference, however, is that their analysis is restricted to Skorokhod reflection-type equilibria, where the state space is partitioned into a waiting region (no control) and a singular control region (where the state is sent to or reflected at the boundary of the closure of the waiting region). 
Importantly, \cite{liang2024equilibria} does not consider equilibria involving absolutely continuous controls, whether with bounded or unbounded rates. 
This marks a central novelty of the present paper: we allow for a broader class of equilibria, including absolutely continuous controls with potentially discontinuous, non-exploding, or exploding rates that induce inaccessible boundaries. 
These may also be combined with singular controls, which generate multiple jumps and reflective boundaries. 
We further note related studies of time-inconsistent dividend problems with bounded absolutely continuous controls, such as \cite{strini2023time} and \cite{zhao2014dividend}.

The main contributions of the present paper are as follows: 
\begin{enumerate}[(i)]

\vspace{-1.5mm}
\item We develop a general framework for admissible control strategies in singular stochastic control problems, allowing for combinations of jumps, Skorokhod reflection, and absolutely continuous controls with potentially discontinuous and exploding rates. 
A key novelty is that we permit the control rate $u(\cdot)$ to become unbounded near a point $\beta$ in the interior of the state space (see Section~\ref{sec:mild} for details). 
In our framework, this explosive behaviour creates an inaccessible, entrance-not-exit boundary at $\beta$, effectively turning it into a threshold that cannot be reached by the controlled process; 
we refer to such strategies as {\it mild threshold control strategies} (see Definition~\ref{def:mixed0}). 
To accommodate starting the diffusion at such a boundary, we introduce an expected cost criterion defined via a suitable limiting procedure and establish sufficient conditions for its well-posedness (see Proposition~\ref{prop_limit_ex} and Remark~\ref{ghghffsa}).
In contrast, we define {\it strong threshold control strategies} as those involving reflection of the state process at a fixed boundary $b$, in the classical Skorokhod sense (see Definition~\ref{def:strong}). 
Our framework also accommodates strategies that combine both singular and absolutely continuous controls, potentially involving both strong and mild thresholds; we refer to these as {\it generalised threshold control strategies} (see Definition~\ref{def:mixedm}). 
Such strategies can induce multiple reflection boundaries and jumps in the controlled process. 
To the best of our knowledge, this is the first framework in the stochastic control literature to study the optimality of unbounded absolutely continuous controls that generate inaccessible boundaries.

\vspace{-2mm}
\item Building on this framework, we introduce a weak equilibrium concept for time-inconsistent singular stochastic control problems under non-exponential discounting. This significantly extends the range of admissible strategies beyond those currently treated in the literature, which have been limited to bounded absolutely continuous control, impulse control, or Skorokhod reflection (cf.~the literature review above).

\vspace{-2mm}
\item We formulate and prove a general verification theorem that provides both necessary and sufficient conditions -- expressed as a system of variational inequalities -- for equilibrium strategies of both strong and mild threshold control type. 
In the case of strong thresholds, our framework additionally accommodates combinations of absolutely continuous controls with Skorokhod reflection.

\vspace{-2mm}
\item Our theory enables the identification of equilibrium strategies in settings where existing methods fail. 
In particular, we study a time-inconsistent inventory control problem under non-exponential discounting. 
In one parameter regime, the problem admits a strong threshold control equilibrium strategy via Skorokhod reflection. 
In another, we can prove that no such equilibrium exists, and an equilibrium arises only through an unbounded, discontinuous absolutely continuous control whose rate explodes in the interior of the state-space and creates an inaccessible boundary. 
This example demonstrates that our extended class of controls is not merely generalising known cases -- it is essential for solving certain time-inconsistent problems.

\end{enumerate}
\vspace{-1mm}
We believe that our framework and techniques have the potential to inform the analysis of a broader class of time-inconsistent control problems and also other games involving singular controls.

In Section 2, we present the mathematical model and define the optimisation criterion considered in this paper. 
We also clarify the sense in which the problem exhibits time-inconsistency. 
In Section 3, we develop the framework for admissible controls and introduce the equilibrium definition. 
Section 4 presents a verification theorem, formulated in terms of variational inequalities, for equilibrium strategies of both strong and mild threshold control types. 
Finally, in Section 5, we apply the developed theory to an inventory control problem, presenting two case studies that yield strong and mild threshold control equilibrium strategies, respectively.

\section{Problem formulation} 
\label{sec:model}

\subsection{The mathematical model}\label{sec:sub-model}

Denote a filtered probability space by  $(\Omega, \F, \left({\cal F}_t\right)_{t\geq 0}, \mathbb{P})$ satisfying the usual conditions, equipped with a standard one-dimensional $\left({\cal F}_t\right)$-Brownian motion $W=\left(W_t\right)_{t\geq0}$.
We consider a diffusion process $X^0=\left(X_t^0\right)_{t\geq0}$ satisfying, in the absence of intervention, the stochastic differential equation (SDE) 
\begin{align} \label{X0}
d X_t^0 = \mu(X_t^0) \, dt + \sigma(X_t^0) \, dW_t, \quad  X^0_{0}= x \in (l,r), \quad 0 \leq t \leq \tau^0_l,
\end{align}
where $-\infty \leq l <r\leq \infty$, and the potential absorption time $\tau^0_l$ is defined by 
\begin{align} \label{tau0}
\tau^0_y:= \inf\{t\geq 0: X^0_t = y \}, \quad y \in [l,r].
\end{align}

\begin{assumption}[Standing assumption] \label{Ass:bounds} 
The functions $\mu:(l,r) \rightarrow \mathbb{R}$ and $\sigma:(l,r) \rightarrow (0,\infty)$ are Lipschitz continuous, $\sigma(x)>0$ for all $x\in(l,r)$, and the boundary $r$ is inaccessible for $X^0$ satisfying the SDE \eqref{X0} (i.e.~$r$ is either natural or entrance-not-exit; see e.g.~\cite[Section II.1.6]{borodin2012handbook} for more details).  
\end{assumption}

It follows from Assumption \ref{Ass:bounds} that the SDE \eqref{X0} admits a unique strong solution (see e.g.~\cite[Chapter 5.5]{Karatzas2}).
Note that the boundary $r$ cannot be reached in finite time, $\mathbb{P}$--a.s., and that if $l$ is reached in finite time then $X^0$ is killed at $\tau^0_l$, while if $l$ is inaccessible for $X^0$ then $\tau^0_l=+\infty$, $\mathbb{P}$-a.s.  
It is worth noting that Feller’s test for explosion (see e.g.~\cite[Theorem 5.5.29]{borodin2012handbook})
provides necessary and sufficient conditions for boundaries of diffusions to be inaccessible. Furthermore, it follows from Assumption \ref{Ass:bounds} that for any $x\in(l,r)$, there exists an $\varepsilon > 0$ such that 
$$
\int_{x - \varepsilon}^{x + \varepsilon} \frac{1+ |\mu(z)|}{\sigma(z)^2} dz < + \infty. 
$$
Hence, the diffusion $X^0$ is regular in $(l,r)$, 
i.e. for every starting point $X^0_0 = x \in (l,r)$ and $y \in (l,r)$, we have 
$\mathbb{P}(\tau_y^0 < \infty) > 0$, where $\tau_y^0$ is defined by \eqref{tau0}. This means that $(l,r)$ cannot be decomposed into smaller sets from which $X^0$ cannot exit (see e.g.~\cite[Section II]{borodin2012handbook} for more details).

The optimisation problem of the present paper is based on a controlled counterpart $X^D=\left(X_t^D\right)_{t\geq0}$ of $X^0$, which satisfies the SDE 
\begin{align} \label{X}
d X_t^D = \mu(X_t^D) \, dt + \sigma(X_t^D) \, dW_t - d D_t, \quad  X^D_{0-}= x \in (l,r), \quad 0 \leq t \leq \tau^D_l,
\end{align}
where the control process $D= (D_t)_{t\geq0}$ belongs to the set
\begin{align*}
\mathcal{E} := \{ &(D_t(\omega))_{t\geq0} \text{ is non-decreasing, right-continuous, $\left({\cal F}_t\right)$-adapted,}\\
&\; \text{ with $D_{0-}=0$, such that  $D_{t}-D_{t-} \leq X_{t-}^D-l$, $\forall$ $t \geq 0$,  $\mathbb{P}$-a.s.\}}
\end{align*}
and the potential absorption time $\tau^D_l$ is defined by 
\begin{align} \label{tau}
\tau^D_l := \inf\{t\geq 0: X^D_t =l\}.
\end{align}
For each $t\geq 0$, we interpret such $D_t$ as the accumulated amount of control exerted by a decision maker from time $0$ up to time $t$. 
Note that $D\in\mathcal{E}$ implies that the controlled process $X^D$ cannot jump to a point below $l$. 

Finally, we conclude again from Assumption \ref{Ass:bounds}, that for any $x\in(l,r)$ and $D \in \mathcal{E}$, there exists a unique strong solution $X^D$ to the SDE \eqref{X} (see e.g.~\cite[Theorem V.7]{protter2005stochastic}).

\subsection{The time-inconsistent cost criterion}
\label{sec:inconsistent}

Consider a process $X^{D}$ satisfying \eqref{X}, and a decision maker who selects the process $D$ and is faced with a stochastic instantaneous cost $f(X^D_t)$ for all $t \in [0, \tau^D_l]$, and a constant cost $f(l)$ for all $t \geq \tau^D_l$ in case $\tau^D_l$ is finite and the process $X^D$ is absorbed at the point $l$. 
In particular, for a given discount rate $q>0$, the total discounted cost (if controlling $X^D$ were costless) for the decision maker is 
$$
\int_0^{\tau^D_l} e^{-q t} \, f(X^D_t) \, dt + \int_{\tau^D_l}^\infty e^{-qt} f(l) dt 
= \int_0^{\tau^D_l} e^{-q t} \, f(X^D_t) \, dt + e^{-q \tau^D_l} q^{-1} f(l).
$$
The instantaneous cost function $f$ is specified in the following standing assumption.

\begin{assumption}[Standing assumption] \label{Ass:f} 
The function $f:[l,r)\rightarrow [0,\infty)$ is continuous and non-decreasing.
\end{assumption}

We assume that the aim of the decision maker is to minimise the costs generated by the  underlying process $X^{D}$. 
Given that the penalisation is higher when the value of $X^D$ is higher (cf.\ Assumption \ref{Ass:f}), the decision maker is motivated to keep $X^D$ as low as possible, by exerting control $D$ to reduce the level of $X^D$ when it increases `too much'.
However, we also assume that exerting control is  costly with a marginal cost of one per unit of control exerted. This creates a trade-off for the decision maker, whose objective is therefore to find an optimal control process $D$ that minimises the total cost.  
Mathematically, we formulate this trade-off as follows.

For a given discount rate $q>0$, control process $D \in \mathcal{E}$, and starting value 
$X_{0-}=x\in(l,r)$ for the process $X^{D}$ satisfying the SDE \eqref{X}, we define the {\it time-consistent} (expected discounted) cost criterion by
\begin{align} \label{wr} 
w(x;q,D):= \mathbb{E}_x\bigg[ \int_0^{\tau^D_l} e^{-q t} \, f(X^D_t) \, dt + \int_0^{\tau^D_l} e^{-q t} \, dD_t + e^{-q \tau^D_l} q^{-1}f(l) \bigg] ,
\end{align}
where $\E_x[\,\cdot\,] := \E^\Pr[\,\cdot \, | \, X^D_{0-}=x]$, 
and the second integral is interpreted in the Riemann–Stieltjes sense, which includes potential atoms of the random measure $t\mapsto dD_t(\omega)$. 

\begin{remark}[Standard setup: Time-consistent singular control problem] \label{rem:stdprob}
The minimisation of the cost criterion \eqref{wr} corresponds to a standard singular stochastic control problem with value function (see~e.g.~\cite[Chapter VIII]{Fleming-Soner})
\begin{align} \label{Ax}
V_q(x) := \inf_{D \in \mathcal{E}} w(x;q,D), \quad x \in (l,r). 
\end{align}
Using standard verification techniques, it can be shown that the associated optimal control process is of Skorokhod reflection type under suitable assumptions. 
\end{remark}

In the present paper we consider a generalisation of the standard problem in Remark \ref{rem:stdprob}, 
by not considering an exponential discounting function (cf.\ $e^{-qt}$ as in \eqref{wr}). 
We instead consider a weighted discount function (WDF) $h:[0,\infty)\rightarrow (0,1]$ corresponding to
\begin{align}\label{eq:h-function}
h(t)=\int_0^\infty e^{-qt} \, dF(q),
\end{align} 
where $F$ is a cumulative probability distribution function satisfying the following assumption. 

\begin{assumption}[Standing assumption]\label{assum:h}  
$F$ has a finite second moment, is concentrated on $(0,\infty)$, and satisfies $\textstyle{\int_0^\infty} q^{-1} dF(q) < \infty$.
\end{assumption}

In particular, we consider a non-exponentially discounted version of \eqref{wr} that is given in terms of the time-consistent cost criteria $w(x;\cdot,D)$ defined by \eqref{wr}; 
namely, for $D \in \mathcal{E}$ and $X^{D}$ satisfying the SDE \eqref{X}, we define the {\it time-inconsistent} (expected discounted) cost criterion
\begin{align} \label{cost_equation}
\begin{split}
J(x;D) &:= \int_0^\infty w(x;q,D) \, dF(q)\\
&= \mathbb{E}_x\left[ \int_0^{\tau^D_l} h(t) \, f(X^D_t) \, dt + \int_0^{\tau^D_l} h(t) \,  dD_t + \int_{\tau^D_l}^\infty h(t)f(l) dt \right].
\end{split}
\end{align}
It is important to note that the minimisation of \eqref{cost_equation} corresponds to a time-inconsistent problem in the sense that it generally holds that the optimal (in the usual sense) control rule depends explicitly on $X^D_{0-}=x$. 

\begin{remark}[Time-inconsistency and failure of Bellman’s principle of optimality] \label{rem:egBM} 
In view of the expression \eqref{cost_equation} for the cost criterion $J(x; D)$, it is easy to see that a control process $D$ that minimises $J(x; D)$ for a given starting value $X_{0-} = x$ will, in general, \textit{not} minimise the cost criterion $J(y; D)$ for another initial value $X_{0-} = y \neq x$.
For instance, consider the case where $F$ corresponds to $\Pr(q = q_1) = p \in (0,1)$ and $\Pr(q = q_2) = 1 - p$, so that the objective becomes $J(x; D) = p \, w(x; q_1, D) + (1-p) \, w(x; q_2, D)$. Suppose a singular control process $D$ that prescribes Skorokhod reflection of the controlled process $X^D$ at a specific threshold in $(l, r)$ is optimal for initial state $x$; then this same control will, in general, not be optimal for a different starting state $y \neq x$.
In other words, different starting values generally give rise to different optimal decision rules. Consequently, a control strategy that is optimal at the initial time $t = 0-$ for state $x$ will typically not remain optimal at later times $t \geq 0$, as the state $X_t$ will usually differ from $x$. This loss of dynamic optimality is the hallmark of time-inconsistent problems.
\end{remark}

\subsection{Applications}
\label{sec:applications} 

There is a wide range of applications that naturally fit into our time-inconsistent singular stochastic control framework. We highlight two such examples below and analyse specific case studies in Section~\ref{sec:mixopt}.

{\rm (i)}~{\it Optimal inventory control}: 
Here, $X^D$ represents the stochastic inventory level, which is controlled by a decision-maker via the process $(D_t)_{t \geq0}$, denoting the cumulative amount of inventory reduction.
The objective is to minimise the cumulative holding costs associated with excess inventory (modelled by a running cost function $f$), balanced against the cost of reducing inventory, which is proportional to the volume unloaded. 
Time-consistent formulations of such problems have been extensively studied; see, e.g.~\cite{
bather1966continuous,
Brownianinventory1,
Brownianinventory2,
doi:10.1137/21M1442115,
HARRISON1978179,
harrison1983instantaneous}, as well as \cite{zipkin2000foundations} for a general overview of optimal inventory management.

{\rm (ii)}~{\it Irreversible reinsurance problems}: 
In this context, $X^D$ denotes the insurer’s risk exposure, which can be reduced through a reinsurance process $(D_t)_{t \geq0}$, representing the total amount of risk transferred to the reinsurer. 
The objective is to minimise the cumulative expected risk loss over time (again modelled by a function $f$), counterbalanced by the cost of reinsurance. For such formulations and further details, see, e.g.~\cite{liang2024equilibria,yan2022irreversible}.

\section{Admissible controls \& Equilibrium definition}
\label{sec:Admissible}

Due to the presence of time-inconsistency in the singular stochastic control problem introduced in Section \ref{sec:inconsistent},  
we do not define a value function using the traditional minimisation of the cost criterion $J(x; D)$ in \eqref{cost_equation} (cf.\ Remarks \ref{rem:stdprob} and \ref{rem:egBM}). 
Instead, we consider a game-theoretic approach, where the minimisation of $J(x; D)$ is considered in the sense of a properly defined equilibrium. 

In what follows, we first review in Section \ref{sec:strong} the well-known singular control strategy of (Skorokhod) reflection-type at constant thresholds. 
Then, we introduce a novel control strategy in Section \ref{sec:mild} that is absolutely continuous with respect to the Lebesgue measure and is driven by a state-dependent, c\`adl\`ag, and unbounded rate with a possible explosion point. 
In Section \ref{sec:mixed} we construct our general class of admissible control strategies that includes both singular parts with multiple reflection points and jumps, and unbounded non-singular parts with exploding rates. 
Based on these, we finally define our notion of equilibrium in Section \ref{sec:NE}.

\subsection{Strong threshold control strategies}
\label{sec:strong}

We first recall the typical (Skorokhod) reflection control strategy, which appears to be optimal in many singular control problems (see e.g.~the standard time-consistent singular control problem \eqref{Ax} in Remark \ref{rem:stdprob}). 
This is a control process $D$ that prescribes the downward reflection of the controlled process $X^D$ at a constant threshold $b \in (l,r)$, consequently dividing the state-space into a 
 \textit{waiting region} $\mathcal{W}$ where no control is exerted, and a \textit{strong action region} $\mathcal{S}$ where we exert the minimal amount of control to maintain the controlled process $X^D$ in the closure $\overline{\mathcal{W}}$ of the waiting region, $\mathbb{P}$-a.s.~for all $t \geq 0$, which are given by 
$$
\mathcal{W}:= (l,b)
\quad \text{and} \quad 
\mathcal{S}:= [b,r).
$$
Notice that the time spent by the process $X^D$ in the strong action region $\mathcal{S}$ is of Lebesgue measure zero, which motivates the name `strong' action region. 

In the present paper, we refer to such control processes as {\it strong threshold control strategies}. 

\begin{definition} [Strong threshold control strategy] \label{def:strong}
A strong threshold control strategy corresponds to a point $b\in (l,r)$, referred to as the strong threshold, such that the control process $D$ is given by 
$$
D_t := (x-b)^+ + L^{b}_t(X^{D}) , 
\quad 0 \leq t \leq \tau^D_l,
$$
where $L^b(X^D)=(L^b_t(X^D))_{t\geq0}$ is the local time at $b$ of the associated controlled process $X^D$ satisfying the SDE \eqref{X}. 
\end{definition}

For strong threshold control strategies, the resulting controlled SDE in \eqref{X} takes the form (see, e.g.~\cite{alvarez2001singular} for a similar construction)
\begin{align} \label{Xb}
X_t^{D} = x + \int_0^t \mu(X_s^{D}) \, ds + \int_0^t \sigma(X_s^{D}) \, dW_s - (x-b)^+ - L^b_t(X^{D}), 
\quad  0 \leq t \leq \tau^D_l,
\end{align}
where we note that existence of a unique strong solution to this SDE \eqref{Xb} for all $x \in (l,r)$ is a well-known result, corresponding to (one-sided) Skorokhod reflection problems 
(see e.g.~\cite{pilipenko2014introduction}, and \cite{Chitashvili01121981} for two-sided reflection).
This implies that the associated controlled process $X^D$ is reflected downwards at $b$, after a potential initial downward jump to $b$ at time $0$ if it starts from $X^D_{0-} = x > b$.

\subsection{Mild threshold control strategies}
\label{sec:mild}

We now consider the possibility that the control strategy involves a part that is non-singular; in particular, we let the process $X^D$ be controlled by an absolutely continuous control, such that the control process is absolutely continuous with respect to the Lebesgue measure and is given via a state-dependent rate function $u(\cdot)$, namely
$$
D_t=\int_0^tu(X^D_s)ds, \quad t \geq 0.
$$ 
A novelty of our approach is that we study the optimality of such a control process defined by a rate that is given by an unbounded, c\`adl\`ag (right-continuous with left-limits) function $u_\beta:(l,\beta) \rightarrow [0,\infty)$, with a possible explosion point
\begin{align*}
\lim_{x \uparrow \beta} u_\beta(x) = \infty, \quad \text{for some} \quad \beta \in (l,r],
\end{align*}
which is such that $\beta$ becomes an inaccessible boundary for the controlled process $X^D$ (formally introduced in Definition \ref{def:mixed0} below). 
The optimality of such a control process has never been studied in the literature on singular control problems to the best of our knowledge. 

For each such function $u_\beta: (l,\beta) \rightarrow [0,\infty)$, it is easy to see from Assumption \ref{Ass:bounds} that 
\begin{align} \label{cond:weak}
\int_{x_1}^{x_2} \frac{|\mu(z) - u_\beta(z)|}{\sigma^2(z)} dz < + \infty,
\quad \text{for all } l<x_1<x_2<\beta,
\end{align}
and we can therefore define for some arbitrary fixed point $c \in (l,\beta)$ the associated scale function $s_\beta$ by 
\begin{align}\label{scale_def}
s_\beta(x) := \int_{c}^{x} \exp \bigg\{ - 2 \int_{c}^{y} \frac{\mu(z) - u_\beta(z)}{\sigma^2(z)} dz \bigg\} \, dy, \quad \text{for } c \leq x < \beta, 
\end{align}
which clearly has an absolutely continuous first derivative. 

The following definition is considerably different to the one of strong threshold control strategies in Definition \ref{def:strong} (see also their comparison in Section \ref{sec:strong-mild}). 

\begin{definition} [Mild threshold control strategies] \label{def:mixed0}
A pair $(u_\beta,\delta)$ is called a mild threshold control strategy if 
$\beta\in(l,r)$, 
$\delta \in (0,\beta-l)$, 
and 
$u_\beta:(l,\beta)\rightarrow [0,\infty)$ is a c\`adl\`ag function with a finite number of discontinuities on each compact interval, and 
\begin{align}\label{non-exit-entrance-cond}
\lim_{x \uparrow \beta} \bigg\{ \int_c^x s'_\beta(y)\int_c^y \frac{2dz}{s'_\beta(z)\sigma^2(z)} dy \bigg\} = \infty,
\end{align}
for some arbitrary $c\in (l,\beta)$, such that the control process $D$ is given by 
\begin{align*}
D_t := (x-\beta+\delta) \mathbf{1}_{\{x \geq \beta \}} + \int_{0}^{t} u_\beta(X_s^D) ds , 
\quad 0 \leq t \leq \tau^D_l,
\end{align*}
for the associated controlled process $X^D$ satisfying the SDE \eqref{X}. 
\end{definition}

We firstly recall by the Feller’s test for explosion (see e.g.~\cite[Section 5.5]{Karatzas2} or \cite[Section II.1.6]{borodin2012handbook}) that the condition \eqref{non-exit-entrance-cond} implies that the threshold $\beta$ is an inaccessible boundary for the process $X^D$ started from a point $X^D_{0-}=x<\beta$. 
In the following result, which is proved in the Appendix, we show that such a threshold $\beta$ is in fact an entrance-not-exit boundary for the process $X^D$.

\begin{lemma}\label{scale_properties_lemma}
Consider a mild threshold control strategy $(u_\beta,\delta)$ given by Definition \ref{def:mixed0}. 
Then, for $s_\beta$ given by \eqref{scale_def}, we have, for $c\in (l,\beta)$, that  
\begin{align}\label{non-exit-entrance-cond2}
\int_c^{\beta}\frac{2s_\beta(x)}{s_\beta'(x)\sigma^2(x)}dx<\infty, 
\qquad \lim_{x\uparrow\beta}s_{\beta}(x)=\infty
\qquad \text{and} \qquad 
\lim_{x\uparrow \beta} u_\beta(x) = \infty.
\end{align}
Hence, the threshold $\beta$ is an entrance-not-exit boundary for the associated process $X^D$.
\end{lemma}
For mild threshold control strategies, the  controlled SDE in \eqref{X} takes the form
\begin{align} \label{Xubeta}
X_t^{D} = 
x - \mathbf{1}_{\{x \geq \beta\}}(x-\beta+\delta)
+ \int_0^t (\mu - u_\beta)(X_s^{D}) \, ds + \int_0^t \sigma(X_s^{D}) \, dW_s, 
\quad  0 \leq t \leq \tau^D_l,
\end{align}
and it turns out that a unique strong solution to the SDE \eqref{Xubeta} exists for all $x \in (l,r)$; this is established in the next paragraph. 
This implies that the associated controlled process $X^D$ remains in $[l,\beta)$ since $\beta$ is an entrance-not-exit boundary for $X^D$ (thanks to Lemma \ref{scale_properties_lemma}), after a potential initial downward jump to $\beta-\delta$ at time $0$ if it starts from $X^D_{0-} = x \geq \beta$.
Given a mild threshold control strategy, we therefore obtain  
a \textit{waiting region} $\mathcal{W}$, 
a \textit{mild action region} $\mathcal{M}$, 
and 
a \textit{strong action region} $\mathcal{S}$, given by
\begin{align}\label{2412rf}
\mathcal{W}:=\{x\in(l,\beta) :u_\beta(x) = 0\},  \quad
\mathcal{M}:=\{x\in(l,\beta) :u_\beta(x) > 0\}, \quad 
\mathcal{S}:= [\beta,r).
\end{align}
Recall that, contrary to the (Lebesgue measure) zero time spent by the process $X^D$ in the strong action region $\mathcal{S}$, the time spent by the process $X^D$ in the mild action region $\mathcal{M}$ has positive Lebesgue measure, which motivates the name `mild' action region. 

In order to address the existence of a unique strong solution to \eqref{Xubeta} for all $x\in(l,r)$, we equivalently show the existence of a unique strong solution of the SDE 
\begin{align} \label{Xu}
X_t &= x + \int_0^t (\mu - u_\beta)(X_s) \, ds + \int_0^t \sigma(X_s) \, dW_s, 
\quad  0 \leq t \leq \inf\{t\geq 0: X_t =l\}, \quad x\in(l,\beta).
\end{align}
A weak solution to \eqref{Xu} exists by standard arguments up to explosion (cf.~\eqref{cond:weak} and, e.g.~\cite[Chapter 5, Theorem 5.15]{Karatzas2}) or until possible killing at $\inf\{t\geq 0: X_t =l\}$, and it is unique in law. 
In order to establish also the pathwise uniqueness, we define an increasing sequence of stopping times $(\tau_n)_{n\in\N}$ by (cf.~\cite[Chapter 5, proof of Theorem 2.5]{Karatzas2})
$$
\tau_n := \inf \big\{t\geq 0 \;:\; X_t\not\in \big(l+\tfrac{1}{n},\beta-\tfrac{1}{n} \big) \cap (-n,n) \big\}.
$$
For any $n\in\N$ in the above sequence, $\sigma(X_t)$ is positive and bounded away from zero and $(\mu-u_\beta)(X_t)$ is bounded for all $t\in [0,\tau_n)$, which can be combined with the Lipschitz continuity of $\mu$ and $\sigma$ from Assumption \ref{Ass:bounds}, to conclude from  \cite{nakao1972pathwise} that any solution $X$ to \eqref{Xu} is pathwise unique up until $\tau_n$. 
Then, recalling that for any solution $X^D$ (weak or strong) to the SDE \eqref{Xu} the point $\beta$ is an entrance-not-exit boundary for $X^D$, we obtain that 
$$
\lim_{n \to \infty} \tau_n = \inf\{t\geq 0: X_t =l\}, \quad \mathbb{P}-\text{a.s.}
$$
Hence, letting $n \to \infty$, the solution $X^D$ to \eqref{Xu} is also pathwise unique (until possible killing at the hitting time of $l$). 
Finally, we observe that pathwise uniqueness combined with weak existence implies the desired existence of a unique strong solution (see, e.g.~\cite{yamada1971uniqueness} and \cite[Chapter 5]{Karatzas2}) to the SDE \eqref{Xu}. 

\begin{remark}
In the forthcoming analysis of our novel theoretical framework for the time-inconsistent control problem (and the solution of the case study in Section \ref{sec:mixopt}), we rely on the use of a mild threshold control strategy and the fact that $\beta$ is an entrance-not-exit boundary for $X^D$. 
\end{remark}

\subsection{Discussion on strong and mild threshold control strategies}
\label{sec:strong-mild}

Recall that a strong threshold control strategy (Definition \ref{def:strong}) corresponds to a waiting region $\mathcal{W}= (l,b)$ where no control is exerted, and a strong action region $\mathcal{S}=[b,r)$ in which minimal amount of control is exerted in order to maintain the controlled process $X^D$ below $b$. 
This corresponds to the standard solution ansatz for singular control problems, which suggests a split of the state-space into two regions, where we either use no control (i.e.~waiting region $\mathcal{W}$) or use a singular control (i.e.~strong action region $\mathcal{S}$).

On the contrary, for a mild threshold control strategy (Definition \ref{def:mixed0}) we obtain an additional mild action region $\mathcal M$ in which the control is exerted with a positive stochastic rate $u_\beta(X^D)$. 
This is a generalisation of the standard solution ansatz for singular control problems, and the intermediate mild action region $\mathcal{M}$ is a novel feature in singular control theory, which can be viewed as a gradual transition between no control (i.e.~zero rate in the waiting region $\mathcal{W}$) and singular control (i.e.~infinite rate in the strong action region $\mathcal{S}$).    

The strong and mild threshold control strategies enjoy one similarity, that the associated controlled process $X^D$ cannot enter the interior of the strong action region. 
Their difference however lies in the fact that, for a strong threshold control strategy, $X^D$ can reach but cannot cross the strong threshold $b$ from below since it is reflected downwards, while for a mild threshold control strategy, $X^D$ cannot even reach the mild threshold $\beta$ from below since it is inaccessible. 

An additional related fundamental difference is that 
a mild threshold control strategy prescribes an initial jump of the associated controlled process $X^D$ from an initial value $x\geq \beta$ to $\beta-\delta$, i.e.~strictly outside the strong action region $\mathcal{S}=[\beta,r)$ since $\beta$ is inaccessible, 
while a strong threshold control strategy prescribes an initial jump of the associated controlled process $X^D$ from an initial value $x>b$ to $b$, i.e.~the boundary of the strong action region $\mathcal{S}=[b,r)$. 
This technical difference will however be mitigated in the following sections, since the selection of an optimal control process will require taking the limits as $\delta \downarrow 0$ in a certain sense (see \eqref{wrqD}--\eqref{eq:limiting-cost-cr}, below).

\subsection{Generalised threshold control strategies}
\label{sec:general}

We finally consider the possibility that the control strategy may comprise a combination of connected or disconnected strong action $\mathcal{S}$, mild action $\mathcal{M}$, and waiting $\mathcal{W}$ regions. 
These may include, in various parts of the state space, multiple strong thresholds (reflection boundaries for the controlled process, associated with singular controls), a mild threshold (an inaccessible boundary for the controlled process, associated with absolutely continuous controls with exploding rates), multiple jumps, absolutely continuous controls with bounded rates, or no control at all. This general framework encompasses both the strong and mild threshold control strategies defined in Sections \ref{sec:strong}--\ref{sec:mild}, respectively, as special cases (see Remark~\ref{rem:controls} below).

In particular, we allow the strong action region $\mathcal{S}$ to take the form
\begin{align} \label{S}
\begin{split}
\mathcal{S} := 
{\textstyle \bigcup\limits_{i=1}^n} \mathcal{S}_i \subseteq \mathbb{R}, 
\, \text{ for } \; 
\mathcal{S}_i = \begin{cases}
[b_i, a_i], \, i = 1, \ldots, n-1, \\
[b_{n}, r), \, i = n
\end{cases}
\hspace{-3mm}\text{and } \; 
l \leq b_1 \leq a_1 < b_2 \leq a_2 < \ldots < b_{n} \leq a_{n}=r,
\end{split}
\end{align}
for which we denote by $\partial \mathcal{S}$ the set of all its boundary points that can be split into the subsets 
\begin{equation} \label{DefdP}
\ul{\partial \mathcal{S}} := \{b_i \;:\; i=1,\ldots,n\} 
\quad \text{and} \quad 
\ol{\partial \mathcal{S}} := \{ a_i \;:\; a_i\neq b_i, \; i = 1,\ldots, n-1\},
\end{equation}
such that ${\partial \mathcal{S}} = \ul{\partial \mathcal{S}} \cup \ol{\partial \mathcal{S}}$ and $\ul{\partial \mathcal{S}} \cap \ol{\partial \mathcal{S}} = \emptyset$. 

\begin{definition} [Generalised threshold control strategies] \label{def:mixedm}
A triple $(u_\beta,\mathcal{S},\delta)$ is called a generalised threshold control strategy if $\beta \in (l,r]$,   
\begin{enumerate}
\vspace{-1mm}
\item[\rm (i)]
the function $u_\beta$ satisfies the conditions of Definition \ref{def:mixed0} (mild threshold control strategy) in case $\beta<r$, and is a non-negative,  c\`adl\`ag function with a finite number of discontinuities on each compact interval in case $\beta=r$; 

\vspace{-2mm}
\item[\rm (ii)] 
$\mathcal{S}$ satisfies \eqref{S} such that $b_n = \beta$ and $\mathcal{S}_n:=[\beta,r)$ in case $\beta<r$; 

\vspace{-2mm}
\item[\rm (iii)] 
$0<\delta< \beta - l \, \mathbf{1}_{\{n=1\}} - a_{n-1} \,  \mathbf{1}_{\{n\geq 2\}}$ when $\beta<r$, and $\delta=0$ when $\beta=r$,
\end{enumerate}
such that the control process $D$ is given by 
\begin{align}\label{eq:D-forMS}
D_t := 
\sum_{i=1}^{n} (x-b_i) \mathbf{1}_{\{x \in \mathcal{S}_i\}}
+ \delta \, \mathbf{1}_{\{x \geq \beta\}} + \int_{0}^{t} u_\beta(X_s^D) ds + \sum_{i=1}^{n} L^{b_i}_t(X^{D}) \mathbf{1}_{\{b_i < \beta\}} 
+ \sum_{0 < \zeta^D_{a_i} \leq t} \hspace{-1mm} (a_i - b_i) , 
\end{align}
for $0 \leq t \leq \tau^D_l$, where $L^{b_i}(X^D) = (L^{b_i}_t(X^D))_{t\geq0}$ is the local time at $b_i$ of the associated controlled process $X^D$ satisfying the SDE \eqref{X} and 
\begin{align} \label{zeta}
\zeta^D_a := \inf\{t\geq 0 : X^D_{t-} = a\}.
\end{align}
\end{definition}

The condition (i) in Definition \ref{def:mixedm} ensures that $u_\beta$ corresponds to a control process that is absolutely continuous with respect to the Lebesgue measure and has a possible inaccessible explosion point $\beta<r$, or to a (classical) absolutely continuous control strategy $u_r$ when $\beta=r$. 
Then if $\beta<r$, condition (ii) ensures that we have a strong action region above the entrance-not-exit boundary $\beta$, while condition (iii) ensures that the initial jump of $D$ when $x\geq \beta$ is such that the controlled process $X_t^{D}$ is sent strictly below the entrance-not-exit boundary $\beta$, but it is sufficiently small so that $X_t^{D}$ is not sent to another strong action region (notice that $\delta$ plays no other role besides this potential initial jump, cf.~\eqref{eq:D-forMS}).

For generalised threshold control strategies, the resulting controlled SDE in \eqref{X} takes the form 
\begin{align} \label{Xbim}
\begin{split}
X_t^{D} 
&= x + \int_0^t (\mu-u_\beta)(X_s^{D}) \, ds + \int_0^t \sigma(X_s^{D}) \, dW_s \\
&\quad - \delta \mathbf{1}_{\{x \geq \beta\}} - \sum_{i=1}^{n} \Big\{ (x-b_i) \mathbf{1}_{\{x \in \mathcal{S}_i\}} + L^{b_i}_t(X^{D})\mathbf{1}_{\{b_i < \beta\}}  \Big\} - \sum_{0 <\zeta^{D}_{a_i} \leq t} (a_i - b_i),  
\quad  0 \leq t \leq \tau^D_l.
\end{split}
\end{align}
This implies that the associated controlled process $X^D$ (if a solution to the SDE \eqref{Xbim} exists, cf.~Remark \ref{rem:SDEexist} below) is reflected downwards at each $b_i \in \ul{\partial \mathcal{S}} \setminus \{l\}$  and whenever it is in $S \setminus \ul{\partial \mathcal{S}}$ it jumps downwards to the closest point in $\ul{\partial \mathcal{S}}$, i.e.~the minimal amount of (singular) control $D$ is exerted to bring the process $X^D$ outside of $S \setminus \ul{\partial \mathcal{S}}$.
To be more precise, the process $X^D$ can be in the interior of $\mathcal{S}$ only at time $0-$ (when it starts from there), while it may hit at some subsequent times $\zeta^{D}_{a_i} > 0$ the points $a_i \in \ol{\partial \mathcal{S}}$ triggering further jumps in $X^D$.  
Finally, at all other times when $X^D \not\in \mathcal{S}$, the process $X^D$ is controlled via the stochastic rate $u_\beta(X^D)$.  
In this case, we therefore obtain a \textit{waiting region} $\mathcal{W}$, a \textit{mild action region} $\mathcal{M}$, and a \textit{strong action region} $\mathcal{S}$, given by 
\begin{align} \label{SiWiMi}
\begin{split}
\mathcal{W} &:= \{x \not\in \mathcal{S} : u_\beta(x) = 0\}
= \{x\in \cup_{i=2}^{n} (a_{i-1}, b_i) \cup (l, b_1) : u_\beta(x) = 0\} , \\
\mathcal{M} &:= \{x \not\in \mathcal{S} : u_\beta(x) > 0\} 
= \{x\in \cup_{i=2}^{n} (a_{i-1}, b_i) \cup (l, b_1) : u_\beta(x) > 0\}, \\
\mathcal{S} &:= \cup_{i=1}^{n-1} [b_i, a_i] \cup [b_{n}, r) .
\end{split}
\end{align}

\begin{remark}[Special cases of generalised threshold control strategies] \label{rem:controls}
By selecting the parameters in Definition \ref{def:mixedm} appropriately, we can recover well-known control strategies and the newly introduced mild threshold one, as follows: 
\begin{itemize}
\vspace{-2mm}
\item $\beta=r$, $u_\beta(x)=0$ for all $x\in(l,r)$, and $\mathcal{S}=[b,r)$ with $ b\in (l,r)$ corresponds to a strong threshold control strategy $(0,[b,r),0)$ satisfying Definition~\ref{def:strong}. 

\vspace{-2mm}
\item $\beta<r$ and $\mathcal{S}=[\beta,r)$ corresponds to a mild threshold control strategy $(u_\beta,[\beta,r),\delta)$ satisfying Definition~\ref{def:mixed0}.

\vspace{-2mm}
\item $\beta=r$ and $\mathcal{S}=\emptyset$ corresponds to a classical absolutely continuous control strategy $(u_r,\emptyset,0)$ such that  
\begin{align*}
D_t := \int_{0}^{t} u_r(X_s) ds ,\quad 0 \leq t \leq \tau^D_l.
\end{align*}
for the associated controlled process $X^D$ satisfying the SDE \eqref{X}. 

\vspace{-2mm}
\item $\beta=r$ and $u_\beta(x)=0$ for all $x\in(l,r)$ corresponds to a singular control strategy $(0,\mathcal{S},0)$ with potentially multiple reflection boundaries and jumps associated with disconnected (strong) action and waiting regions. 
\end{itemize}
\end{remark}

\begin{remark}[Further generalisation of control strategies] \label{rem:multi-beta}
It is possible to further generalise Definition \ref{def:mixedm} to the case of a function $u_\beta$ that allows for multiple explosion points and inaccessible boundaries (cf.~Definition \ref{def:mixed0}). 
Since this generalisation would not introduce further mathematical difficulties, but would only increase notational complexity in the paper, we consider a single explosion point. 
\end{remark}

\subsection{The general class of admissible control strategies}
\label{sec:mixed}
In this section, we aim at defining the class of our admissible control strategies based on the generalised threshold control strategies in Section \ref{sec:general}. 
The study of the optimality of control strategies in all forms included in this class is novel in the literature on singular control problems; especially, the search of optimal control strategies that create inaccessible boundaries for the controlled processes; see Section \ref{sec:intro} for details. 

In order to introduce our class of admissible controls, we first note that for each generalised threshold control strategy $(u_\beta,\mathcal{S},\delta)$ such that a unique strong solution to the SDE \eqref{Xbim} exists (cf.~Remark \ref{rem:SDEexist} for a discussion in this direction), we denote 
the associated control process $D$ in \eqref{eq:D-forMS} by $D^{u_\beta,\mathcal{S},\delta}$, 
the associated controlled process $X^D$ in \eqref{Xbim} by $X^{u_\beta,\mathcal{S},\delta}$, the absorption time $\tau^D_l$ in \eqref{tau} by $\tau^{u_\beta,\mathcal{S},\delta}_l$, 
the jump times $\zeta^D_a$ in \eqref{zeta} by $\zeta^{u_\beta,\mathcal{S},\delta}_a$, 
and the associated time-consistent cost criterion $w(x;q,D)$ in \eqref{wr} by 
\begin{align} \label{wrd} 
\begin{split}
w(x;q,u_\beta,\mathcal{S},\delta) = \mathbb{E}_x\bigg[ &\int_0^{\tau^{u_\beta,\mathcal{S},\delta}_l} \hspace{-6mm} e^{-q t} \, (f+u_\beta)(X^{u_\beta,\mathcal{S},\delta}_t)\, dt + \delta \mathbf{1}_{\{x \geq \beta\}} + \sum_{i=1}^{n} (x-b_i) \mathbf{1}_{\{x \in \mathcal{S}_i\}} 
+ \hspace{-5mm}\sum_{0 <\zeta^{u_\beta,\mathcal{S},\delta}_{a_i} \leq t} \hspace{-4mm}(a_i - b_i) \\
&+ \sum_{i=1}^{n} \int_0^{\tau^{u_\beta,\mathcal{S},\delta}_l} \hspace{-6mm} e^{-q t} \, d L^{b_i}_t(X^{u_\beta,\mathcal{S},\delta}) \mathbf{1}_{\{b_i < \beta\}} + e^{-q \tau^{u_\beta,\mathcal{S},\delta}_l} q^{-1}f(l) \bigg], \quad \text{for all } x\in(l,r).
\end{split}
\end{align}

\begin{remark}
\label{rem:on-admiss}
If a generalised threshold control strategy $(u_\beta,\mathcal{S},\delta)$ is such that the corresponding controlled SDE has a unique strong  solution for $0 \leq t \leq \tau^D_l$ and $x \in (l,r)$, then this also holds when replacing $\delta$ with any $\tild \delta \in (0,\delta)$. 
This is directly verified by using Definition \ref{def:mixedm}.  
\end{remark}

Recall that the sole purpose of the parameter $\delta$ in Definition \ref{def:mixedm} of generalised threshold control strategies is to ensure that the initial jump of $D$ when $x\geq \beta$ sends $X^{D}$ strictly below the entrance-not-exit boundary point $\beta$. However, we are further interested in allowing the initial jump to be such that the process is sent to the entrance-not-exit boundary point $\beta$, which we obtain by sending $\delta \downarrow 0$ in \eqref{wrd}; to this end we define the {\it limiting} time-consistent cost criterion $w(x;q, u_\beta, \mathcal{S})$ by 
\begin{align} \label{wrqD}
w(x;q, u_\beta, \mathcal{S}) 
&:= \begin{cases} 
w(x;q, u_\beta, \mathcal{S},\delta), & \text{for } x\in(l,\beta), \\
\lim_{\delta\to 0} w(x;q,u_\beta,\mathcal{S},\delta), & \text{for } x \in [\beta,r),
\end{cases} 
\end{align}
(Recall that $\delta$ plays no role when $x<\beta$).
Sufficient conditions for these limits to exist are given in the following general result, which is proved in the Appendix and is applicable beyond the context of this paper. 

\begin{proposition} \label{prop_limit_ex} 
Suppose that $(u_{\beta},\mathcal{S},\delta)$ is a generalised threshold control strategy with $\beta<r$, 
such that the associated SDE \eqref{Xbim} has a unique strong solution $X^D$ for all $X_{0-}^D= x \in (l,r)$.
Suppose also that $w(x;q,u_\beta,\mathcal{S},\delta)<\infty$ for $x\in (l,\beta)$ and that there exists a constant $\epsilon>0$ such that $u_\beta(\cdot)$ is increasing on $(\beta-\epsilon,\beta)$. 
Then the limit in \eqref{wrqD} exists and is finite.
\end{proposition}

We can further define the associated {\it limiting} time-inconsistent cost criterion $J(x; u_\beta, \mathcal{S})$ by
\begin{align} \label{eq:limiting-cost-cr}
\begin{split}
J(x; u_\beta, \mathcal{S}) 
&:= 
\lim_{\delta \to 0} J(x;u_\beta,\mathcal{S},\delta) 
= \lim_{\delta \to 0} \int_0^\infty w(x;q, u_\beta, \mathcal{S}, \delta) \, dF(q) \\
&=\begin{cases} 
\int_0^\infty w(x;q, u_\beta, \mathcal{S}) dF(q), & \text{for } x\in(l,\beta), \\
\lim_{\delta\to 0} \int_0^\infty w(x;q, u_\beta, \mathcal{S},\delta) dF(q), & \text{for } x \in [\beta,r).
\end{cases}
\end{split}
\end{align}

\begin{remark}\label{ghghffsa}
Note that the limit in \eqref{eq:limiting-cost-cr} exists under the conditions of Proposition \ref{prop_limit_ex} together with suitable integrability assumptions on $w(\cdot; \cdot, u_\beta, \mathcal{S})$ and $F$, namely that there exists a locally bounded function $x \mapsto C(x,q)$ such that for all $x\in (l,r)$ and $q\in {\rm supp}(F)$ we have
\begin{align*}
w(x;q, u_\beta, \mathcal{S}) \leq C(x,q) 
\quad \text{and} \quad 
\int_0^\infty C(x,q)dF(q)<\infty,
\end{align*}
which allows application of the dominated convergence theorem; see Lemma~\ref{useful-lemma-2} below for a result in this direction. 
\end{remark}

We are now ready to define our notion of admissible control strategies $D$ in the form of \eqref{eq:D-forMS} which requires that their corresponding SDE in \eqref{Xbim} admits a unique strong solution (see Remark \ref{rem:SDEexist} below regarding existence and uniqueness of solutions to SDEs of this type) and their corresponding limiting cost criteria in \eqref{wrqD}--\eqref{eq:limiting-cost-cr} exist and are finite.   
The set of {\it admissible control strategies} $\mathcal{A}$ is therefore defined by
\begin{align} \label{adm-set}
\begin{split}
\mathcal{A} := \{ &(u_\beta,\mathcal{S},\delta) : (u_\beta,\mathcal{S},\delta) \text{ is a generalised threshold control strategy satisfying Definition \ref{def:mixedm},} \\
&\text{the associated SDE \eqref{Xbim} has a unique strong solution for all $x \in (l,r)$, and the associated} \\
&\text{limiting cost criteria in \eqref{wrqD}--\eqref{eq:limiting-cost-cr} exist and are finite for all $x \in (l,r)$,  and all $q\in {\rm supp}(F)$}\}.
\end{split}
\end{align}
In what follows we aim at defining an appropriate equilibrium definition for the comparison between the cost criterion \eqref{cost_equation} corresponding to an admissible control strategy $(u_\beta,\mathcal{S},\delta) \in \mathcal{A}$, which is given by $J(\cdot; u_\beta, \mathcal{S})$ in \eqref{eq:limiting-cost-cr}, 
and the value obtained when deviating from this strategy locally around each initial value $X^D_{0-}=x$, for all $x \in (l,r)$. 
For each such $x \in (l,r)$, each admissible \textit{deviation} control strategy at $x$, should require the existence of a solution to the associated SDE only locally around $x$. 
Therefore, for each $x \in (l,r)$, we define the associated set of {\it admissible deviation control strategies} $\mathcal{A}_x$ by
\begin{align} \label{adm-dev-set}
\begin{split}
\mathcal{A}_x := \big\{ &(u_\beta,\mathcal{S},\delta) : (u_\beta,\mathcal{S},\delta) \text{ is a generalised threshold control strategy satisfying} \\
&\text{Definition \ref{def:mixedm}, and $\exists\;\epsilon = \epsilon(x)$ in $(0,x-l)$ so that the associated SDE} \\
&\text{in \eqref{Xbim} defined for $0 \leq t \leq \theta^{u_\beta,\mathcal{S},\delta}_{x,\epsilon}$ has a unique strong solution} \big\},
\end{split}
\end{align}
 where 
\begin{align}\label{eq:tau-eps-stop}
\theta^{u_\beta,\mathcal{S},\delta}_{x,\epsilon} := 
\inf \big\{t\geq 0: X_t^{u_\beta,\mathcal{S},\delta} \notin (x-\epsilon,x+\epsilon)\big\} > 0. 
\end{align}
To avoid daunting notation, we denote the controlled process $X^{u_\beta,\mathcal{S},\delta}$ evaluated at $\theta^{u_\beta,\mathcal{S},\delta}_{x,\epsilon}$ by  $X_{\theta_{x,\epsilon}}^{u_\beta,\mathcal{S},\delta}$.

\begin{remark}
\label{rem:SDEexist}
Any strong threshold control strategy $(0,[b,r),0)$, mild threshold control strategy  $(u_\beta,[\beta,r),\delta)$ and absolutely continuous control strategy $(u_r,\emptyset,0)$, leads to the existence and uniqueness of a strong solution to their corresponding SDEs (cf.~Sections \ref{sec:strong}--\ref{sec:mild}). 
Hence, these strategies are admissible deviation strategies (i.e.~belong to $\mathcal{A}_x$), and if their corresponding limiting cost criteria exist (which is trivially true for strong threshold and absolutely continuous control strategies), then these are also admissible (i.e.~belong to $\mathcal{A}$). 
Note, however, that the set of possible deviation control strategies $\mathcal{A}_x$ is very large; indeed it contains all generalised threshold control strategies (Definition \ref{def:mixedm}) such that their corresponding SDEs \eqref{Xbim} admit unique strong solutions locally. The question of existence of unique strong solutions to SDEs of this type that can involve (a subset of) reflecting boundaries, jumps, absolutely continuous controls, and inaccessible boundaries, is well-studied in the theory of SDEs; we refer the interested reader to (a small selection of works)
\cite{
bass2007pathwise,
bass2005one,
blei2013one,
blei2013note,
dupuis1993sdes,
oksendal2019applied,
tanaka1979stochastic,
yang2023strong}.
\end{remark}

\subsection{Equilibrium}
\label{sec:NE}

The decision-maker in our control problem aims at selecting an 
admissible control strategy $(u_\beta,\mathcal{S},\delta) \in \mathcal{A}$ to minimise the corresponding cost criterion in \eqref{cost_equation}, which takes the form of $J(\cdot; u_\beta, \mathcal{S})$ in \eqref{eq:limiting-cost-cr}. 
However, since this problem is in general time-inconsistent (cf.~Remark \ref{rem:egBM}), 
instead of searching for minimising controls in the usual sense (cf.~Remark \ref{rem:stdprob}), the decision-maker searches for a suitably defined equilibrium control strategy.
 
\begin{definition}[Equilibrium]
\label{def:equ_stop_time} 
Recall that strong action regions $\mathcal{S}$ are of the form \eqref{S}--\eqref{DefdP}, 
the set of admissible control strategies is $\mathcal{A}$ from \eqref{adm-set} and admissible deviation control strategies is $\mathcal{A}_{x}$ from \eqref{adm-dev-set}--\eqref{eq:tau-eps-stop}. 

The pair $(u^*,\mathcal{S}^*)$ is called an equilibrium strategy if: 
$(u^*,\mathcal{S}^*,\delta^*) \in \mathcal{A}$ for some $\delta^*\geq 0$ with associated limiting cost criteria 
$w(\cdot;q,u^*,\mathcal{S}^*)$ and $J(\cdot; u^*,\mathcal{S}^*)$ defined by \eqref{wrqD}--\eqref{eq:limiting-cost-cr};
and for each $x\in(l,r)$ and any $(u_\beta, \mathcal{S}, \delta)\in \mathcal{A}_{x}$ with associated deviation control process $D^{u_\beta,\mathcal{S},\delta}$ as in \eqref{eq:D-forMS},  controlled process $X^{u_\beta,\mathcal{S},\delta}$ as in \eqref{Xbim} and $\theta^{u_\beta,\mathcal{S},\delta}_{x,\epsilon}$ as in \eqref{eq:tau-eps-stop}, 
the following two conditions hold: 
\begin{enumerate}[\rm (i)]
\vspace{-2mm}
\item In case 
$x \in (l,\beta) \cap (\mathcal{S}^c \cup \ul{\partial \mathcal{S}})$, we have 
\begin{align}\label{EQ1}
\begin{split}
L(x):=
\limsup_{\epsilon \downarrow 0} 
\bigg\{ \bigg( &J(x; u^*, \mathcal{S}^*) - \int_0^\infty \bigg\{ 
\mathbb{E}_x \bigg[ \int_{0}^{\theta^{u_\beta,\mathcal{S},\delta}_{x,\epsilon}} e^{-qt} f\big( X_t^{u_\beta,\mathcal{S},\delta} \big) dt + \int_{0}^{\theta^{u_\beta,\mathcal{S},\delta}_{x,\epsilon}} e^{-qt} d D^{u_\beta,\mathcal{S},\delta}_t  \\
&+ e^{-q \theta^{u_\beta,\mathcal{S},\delta}_{x,\epsilon}} \, w\Big( X^{u_\beta,\mathcal{S},\delta}_{\theta_{x,\epsilon}};q, u^*, \mathcal{S}^* \Big)\bigg] \bigg\} dF(q) \bigg) \bigg/ \E_x\left[\theta^{u_\beta,\mathcal{S},\delta}_{x,\epsilon} \right] \bigg\}
\leq 0;
\end{split}
\end{align}

\vspace{-2mm}
\item In case $x \in [\beta,r) \cup (\mathcal{S} \setminus \ul{\partial \mathcal{S}})$, we have
\begin{align}\label{EQ2}
\begin{split}
J(x; u^*, \mathcal{S}^*) \leq 
& \sum_{i=1}^{n} \big( J(b_i; u^*, \mathcal{S}^*) + x -  b_i \big) \mathbf{1}_{\{x \in \mathcal{S}_i\}} 
+ \big( J(\beta-\delta; u^*, \mathcal{S}^*) - J(\beta; u^*, \mathcal{S}^*) 
+ \delta \big) \mathbf{1}_{\{x \geq \beta\}}.
\end{split}
\end{align}
\end{enumerate}
\end{definition}
In what follows, we provide a game-theoretic interpretation of the equilibrium strategy $(u^*,\mathcal{S}^*)$ in  Definition \ref{def:equ_stop_time}, which is in line with the usual interpretation of the game-theoretic approach to time-inconsistent control problems (see Section \ref{sec:intro} for more details). 

We first note that, if the equilibrium strategy $(u^*,\mathcal{S}^*)$ is used by the decision maker, then the value of this strategy is given by $J(x; u^*, \mathcal{S}^*)$ as in \eqref{eq:limiting-cost-cr}, when the initial state of the process takes the value $x \in (l,r)$.
At each $x$, the decision-maker may deviate unilaterally from the equilibrium strategy $(u^*,\mathcal{S}^*)$, by  using a deviation control strategy $(u_\beta, \mathcal{S}, \delta)\in \mathcal{A}_{x}$ until the associated state process exits from the interval $(x-\epsilon,x+\epsilon)$ at time $\theta^{u_\beta,\mathcal{S},\delta}_{x,\epsilon}$ from~\eqref{eq:tau-eps-stop}, and then continue according to the equilibrium strategy $(u^*,\mathcal{S}^*)$.
Given that $\epsilon>0$ is eventually sent to zero, we notice that only local deviation is possible. 

Now, if the decision-maker chooses to deviate without causing a jump in the state process, i.e.~when $x$ and $(u_\beta, \mathcal{S}, \delta)$ are such that $x \in (l,\beta) \cap (\mathcal{S}^c \cup \ul{\partial \mathcal{S}})$, then the corresponding value for a sufficiently small $\epsilon>0$ becomes 
$$
\int_0^\infty \bigg\{ 
\mathbb{E}_x \bigg[ \int_{0}^{\theta^{u_\beta,\mathcal{S},\delta}_{x,\epsilon}} e^{-qt} f\big( X_t^{u_\beta,\mathcal{S},\delta} \big) dt + \int_{0}^{\theta^{u_\beta,\mathcal{S},\delta}_{x,\epsilon}} e^{-qt} d D^{u_\beta,\mathcal{S},\delta}_t  + e^{-q \theta^{u_\beta,\mathcal{S},\delta}_{x,\epsilon}} \, w\Big( X^{u_\beta,\mathcal{S},\delta}_{\theta_{x,\epsilon}};q, u^*, \mathcal{S}^* \Big) \bigg] \bigg\} dF(q),
$$
where the integrals inside the expectation represent the value from deviation and the latter term that the decision maker continues according to the equilibrium strategy from the exit time $\theta^{u_\beta,\mathcal{S},\delta}_{x,\epsilon}$ onwards. 
Given that the value of such a strategy results in a higher cost compared with the equilibrium value in the numerator of \eqref{EQ1}, the interpretation of Definition \ref{def:equ_stop_time}.(i) is that the decision-maker has no incentive to deviate in this~way.

Finally, if the decision-maker chooses to deviate by causing an immediate jump in the state process, i.e.~when $x$ and $(u_\beta, \mathcal{S}, \delta)$ are such that $x \in [\beta,r) \cup (\mathcal{S} \setminus \ul{\partial \mathcal{S}})$, then the process jumps to either $\beta-\delta$ (in case $x \in [\beta,r)$) or $b_i$ (in case $x \in  {\cal S}_i\setminus \ul{\partial \mathcal{S}}$ for some $i \in \{1,\ldots,n-1\}$.
The corresponding value thus consists of the size of this initial jump and the value of following $(u^*,\mathcal{S}^*)$ after the jump, which corresponds to the right-hand side in \eqref{EQ2}. 
Given that this results in a higher cost compared with the equilibrium value in the left-hand side of \eqref{EQ2}, the interpretation of Definition \ref{def:equ_stop_time}.(ii) is also that the decision-maker has no incentive to deviate in this way.

\section{A verification theorem for strong and mild threshold control strategies}
\label{sec:main}

In this section, we provide a general verification theorem with necessary and sufficient conditions for obtaining equilibrium strategies that involve a connected strong action region of the form of $\mathcal{S}=[b,r)$ for some $b\in (l,r)$. 
Namely, we  provide the means for obtaining equilibria in the form of generalised threshold control strategies (Definition \ref{def:mixedm}) satisfying $(u_\beta, \mathcal{S}, \delta)=(u_\beta, [b,r), \delta)$, with the following two possible cases:
\begin{itemize}
\vspace{-2mm}
\item If $b<\beta =r$, then the function $u_{\beta}=u_r$, representing the control rate $u_r(X^{u_r, [b,r),\delta})$ of an absolutely continuous control with respect to the Lebesgue measure, does not explode. In this case, $b$ corresponds to a reflecting boundary point for the controlled process $X^{u_r,[b,r),\delta}$, and particularly to a strong threshold (e.g.\ Definition \ref{def:strong}) if $u_r \equiv 0$. Furthermore, recall that this strategy is independent of $\delta$;

\vspace{-2mm}
\item If $b=\beta<r$, then the control strategy involves a part that is absolutely continuous with respect to the Lebesgue measure, given by a non-zero control rate $u_\beta(X^{u_\beta, [\beta,r),\delta})$ such that $\beta$ is an inaccessible boundary for the controlled process $X^{u_\beta, [\beta,r),\delta}$. In this case, $b=\beta$ corresponds to a mild threshold (cf.\ Definition \ref{def:mixed0}). 
\end{itemize}
\vspace{-1mm}
For clarification, we should point out here that our verification theorem does \textit{not}, however, restrict the deviation control strategies; recall that we search for equilibria corresponding to the Definition \ref{def:equ_stop_time}. Furthermore, our verification theorem involves both necessary and sufficient conditions for the optimality of such generalised threshold control strategies, which is one of our main mathematical contributions in this paper. 
Its importance is later illustrated via our case studies in Section \ref{sec:mixopt}, where it provides a pathway for first obtaining an equilibrium in strong threshold control strategies for one case, then for proving that there is no equilibrium strategy in strong threshold control strategies in another case, and finally for finding an equilibrium strategy in mild threshold control strategies instead. 

It is important for the subsequent analysis to define the set $\mathcal{J}_u$ of discontinuity points of each c\`adl\`ag function $u_\beta:(l,\beta) \to [0,\infty)$ that satisfies the conditions of Definition  \ref{def:mixedm}  
by 
\begin{align*}
    \mathcal{J}_u=\{x\in(l,\beta):u_\beta(x-) \neq u_\beta(x)\}
\end{align*}
and recall that this set is finite on each compact interval.  
We also recall the waiting region $\mathcal{W}_u$ and the mild action region $\mathcal{M}_u$ associated to $u_\beta$, defined as in \eqref{SiWiMi}, which take the form  
\begin{align}\label{eq:mild-string-u}
\mathcal{W}_u:=\{x\in(l,\beta) :u_\beta(x) = 0\},
\quad 
\mathcal{M}_u:=\{x\in(l,\beta) :u_\beta(x) > 0\}. 
\end{align}
Finally, we denote by $\Af$ the infinitesimal generator of the uncontrolled process $X^0$ satisfying \eqref{X0}, which is defined at least for functions $G \in {\cal C}^2(l,r)$ by
\begin{align*}
\mathbf{A} G(x) := \mu(x) G'(x) + \frac{1}{2}\sigma^2(x) G''(x), 
\quad \text{for } x \in (l,r).
\end{align*}

Before reporting the verification theorem, we need the following important technical result on the existence and finiteness of the limiting cost criterion $w(\cdot;q,u_\beta,[b,r))$ defined by \eqref{wrqD} for $\mathcal{S}=[b,r)$.  

\begin{lemma}\label{useful-lemma} 
Consider a strong action region $\mathcal{S}=[b,r)$ and a generalised threshold control strategy $(u_\beta,\mathcal{S},\delta)$ satisfying Definition \ref{def:mixedm} with either $l<b<\beta=r$ or $l<b=\beta<r$, such that the associated SDE (cf.~\eqref{Xbim}) has a unique strong solution $X^{u_\beta,\mathcal{S},\delta}$ for all starting values $x \in (l,r)$. 
For a given constant $q>0$, if the boundary value problem (BVP)
\begin{align}
&f(x) + (\mathbf{A}-q)v(x;q) - u_\beta(x) v'(x;q) + u_\beta(x) = 0, \quad x \in (\mathcal{W}_u \cup \mathcal{M}_u) \setminus \mathcal{J}_u, 
\label{BVPaa}\\
&v(x;q) = x - b + v(b;q), \qquad \qquad \qquad \qquad \qquad \qquad \, x \in \mathcal{S}=[b,r), 
\label{BVPab}\\
&v(l;q):=v(l+;q)=q^{-1}f(l), 
\label{BVPac}\\
& \text{$v(\cdot;q)\in \mathcal{C}^2\big( (l,r) \setminus (\mathcal{J}_u\cup \partial\mathcal{S}) \big) \cap \mathcal{C}^1(l,r)$ with $|v''(x\pm;q)|<\infty$ for all $x\in (l,r)$
}\label{BVPad}
\end{align} 
has a solution, then the limiting cost criterion $w(\cdot;q,u_\beta,\mathcal{S})$ defined by \eqref{wrqD} exists and is given by  
\begin{align} \label{w=w} 
w(x;q,u_\beta,\mathcal{S}) = v(x;q), \quad x \in [l,r).
\end{align}
\end{lemma}

\begin{proof} 
Fix a strong action region $\mathcal{S}=[b,r)$ for some $b \in (l,r)$. 
We begin by noticing that the time-consistent cost criterion $w(x;q,D)$ in \eqref{wr} associated with the control strategy $(u_\beta,\mathcal{S},\delta)$, takes the form in \eqref{wrd}, which is involved in the definition \eqref{wrqD} of the limiting cost criterion $w(\cdot;q,u,\beta)$. 
It is clear from the latter definition that $w(x;q, u_\beta, \mathcal{S}) =
w(x;q, u_\beta, \mathcal{S},\delta)$ unless $x \in [b,r) \equiv [\beta,r)$ in the case when $b=\beta<r$, which is the only occasion where $\delta$ becomes relevant. 
We use this observation in all three steps of the proof:

\vspace{1mm}
{\it Step 1: The case $x=l$.} 
In this case $\tau^D_l \equiv \tau^{u_\beta,\mathcal{S},\delta}_l = 0$ in \eqref{tau}, 
and it is straightforwardly seen that $w(l;q,u,\beta) = q^{-1}f(l) = v(l;q)$. 

\vspace{1mm}
{\it Step 2: The case $x\in (l,b)$.} 
Note that this covers both the case when $b=\beta<r$ and the boundary point $b=\beta$ is inaccessible for $X^{u_\beta,\mathcal{S},\delta}$, hence its state-space identifies with the interval $(l,\beta)$, and the case when $b<\beta= r$ and $b$ is a reflecting boundary for $X^{u_\beta,\mathcal{S},\delta}$.   
Consider now the sequence of stopping times 
$$
\eta_k = \eta_k^{u_\beta,\mathcal{S},\delta} 
:= \inf \big\{t\geq 0: X^{u_\beta,\mathcal{S},\delta}_t \not\in (l+1/k, \beta-1/k) \big\} \wedge \inf \big\{ t\geq 0: X^{u_\beta,\mathcal{S},\delta}_t \not\in (-k,k) \big\} \wedge k, \quad k\in\N.
$$ 
In view of \eqref{BVPad}, we can use the It\^{o}-Tanaka formula for continuous semimartingales and the occupation times formula (see e.g.~\cite[Chapter VI, Exercise 1.25, Corollary 1.6]{revuz2013continuous}) to calculate 
\begin{align} \label{FBa0}
v(X^{u_\beta,\mathcal{S},\delta}_{\eta_k}; q) 
&= v(x;q) + \int_0^{\eta_k } v'(X^{u_\beta,\mathcal{S},\delta}_{t};q) dX^{u_\beta,\mathcal{S},\delta}_t + \frac12 \int_0^{\infty} L^y_{\eta_k} v''(dy;q) \notag\\
&= v(x;q) + \int_0^{\eta_k } \textbf{A} v(X^{u_\beta,\mathcal{S},\delta}_t;q) {\bf 1}_{\{X^{u_\beta,\mathcal{S},\delta}_t\not\in \mathcal{J}_u\cup \partial\mathcal{S}\}} dt + \int_0^{\eta_k} v'(X^{u_\beta,\mathcal{S},\delta}_{t};q)\sigma(X^{u_\beta,\mathcal{S},\delta}_t) dW_t \notag\\
&\quad - \int_0^{\eta_k} v'(X^{u_\beta,\mathcal{S},\delta}_{t};q) dD^{u_\beta,\mathcal{S},\delta}_t ,
\end{align}
where $v''(dy;q)$ is the measure that identifies with the second distributional derivative of $v(\cdot;q)$. 
Then, by using the structure of the strong action region $\mathcal{S}=[b,r)$ in Definition \ref{def:mixedm} and the expression \eqref{eq:D-forMS} of $D^{u_\beta,\mathcal{S},\delta}$ we observe that 
$$
D^{u_\beta,\mathcal{S},\delta}_t = \int_{0}^{t} u_\beta(X_s^{u_\beta,\mathcal{S},\delta}) ds + L^{b}_t(X^{u_\beta,\mathcal{S},\delta}) \mathbf{1}_{\{b < \beta = r\}}, \quad 0 \leq t \leq \eta_k. 
$$
Substituting this into \eqref{FBa0} and using the integration by parts formula as well as taking expectations on both sides we obtain 
\begin{align*} 
&\E_x\left[e^{- q \eta_k} v(X^{u_\beta,\mathcal{S},\delta}_{\eta_k};q) \right] - v(x;q) \notag\\
&= \E_x\bigg[\int_0^{\eta_k } e^{-qt} \left( (\textbf{A} - q) v(X^{u_\beta,\mathcal{S},\delta}_t;q) {\bf 1}_{\{X^{u_\beta,\mathcal{S},\delta}_t\not\in \mathcal{J}_u\cup \partial\mathcal{S}\}}  - u_\beta(X^{u_\beta,\mathcal{S},\delta}_t) v'(X^{u_\beta,\mathcal{S},\delta}_{t};q) \right) dt \notag\\
&\qquad \quad - \mathbf{1}_{\{b < \beta = r\}}  \int_0^{\eta_k} e^{-qt} v'(X^{u_\beta,\mathcal{S},\delta}_{t};q) dL^{b}_t(X^{u_\beta,\mathcal{S},\delta})  \bigg],
\end{align*}
due to continuity of $\sigma$ and $v'$ (which imply that the integrand of the stochastic integral is bounded on the time interval $[0,\eta_k]$). 
Combining the above together with \eqref{BVPaa} and using the fact that the set $\{t \geq 0 :  X^{u_\beta,\mathcal{S},\delta}_t \in \mathcal{J}_u\cup \partial\mathcal{S}\}$ has Lebesgue measure zero, yields
\begin{align} \label{FBa}
&\E_x\left[e^{- q \eta_k} v(X^{u_\beta,\mathcal{S},\delta}_{\eta_k};q) \right] - v(x;q) \notag\\
&= \E_x\bigg[- \int_0^{\eta_k } e^{-qt}  \left( f(X^{u_\beta,\mathcal{S},\delta}_t;q) + u_\beta(X^{u_\beta,\mathcal{S},\delta}_t) \right) dt - \mathbf{1}_{\{b < \beta = r\}} \int_0^{\eta_k} e^{-qt} dL^{b}_t(X^{u_\beta,\mathcal{S},\delta}) \bigg]
\end{align}
where we also used the fact that the local time $L^{b}(X^{u_\beta,\mathcal{S},\delta})$ increases only when $X^{u_\beta,\mathcal{S},\delta} = b$ together with $v'(b;q) = 1$ due to \eqref{BVPab} and \eqref{BVPad}. 

Given that $b$ is inaccessible for $X^{u_\beta,\mathcal{S},\delta}$ when $b=\beta<r$ and that $X^{u_\beta,\mathcal{S},\delta}_t \in [l,b]$ for all $t>0$ when $b<\beta=r$ (so that $\beta$ cannot be reached in finite time), we have $\eta_k \rightarrow \tau^{u_\beta,\mathcal{S},\delta}_l \leq +\infty$ as $n\rightarrow\infty$. 
Using the continuity of $v(\cdot;q)$ and the observation that $l\leq X^{u_\beta,\mathcal{S},\delta}_{\eta_k}\leq b<r$, $\mathbb{P}_x$-a.s., together with the property $v(l+;q) \geq 0$ due to \eqref{BVPac} and the non-negativity of $f$, we see that $v(X^{u_\beta,\mathcal{S},\delta}_{\eta_k};q)$ is uniformly bounded in $k$ , so that
$$
\lim_{k \to \infty} \E_x\Big[e^{- q \eta_k} v(X^{u_\beta,\mathcal{S},\delta}_{\eta_k};q) \Big] = \E_x\Big[e^{- q \tau^{u_\beta,\mathcal{S},\delta}_l} q^{-1} f(l) \Big].
$$
Using this limit together with the non-negativity of $f$ and $u$ and the monotone convergence theorem, we conclude by sending $k \rightarrow \infty$ in \eqref{FBa}, that
\begin{align} \label{eq:vxq}
\begin{split}
v(x;q) &= \E_x\bigg[\int_0^{\tau^{u_\beta,\mathcal{S},\delta}_l} e^{-qt} \left( f(X^{u_\beta,\mathcal{S},\delta}_t;q) + u_\beta(X^{u_\beta,\mathcal{S},\delta}_t) \right) dt + \mathbf{1}_{\{b < \beta = r\}} \int_0^{\tau^{u_\beta,\mathcal{S},\delta}_l} \hspace{-3mm}e^{-qt} dL^{b}_t(X^{u_\beta,\mathcal{S},\delta}) \\
&\qquad \quad + e^{- q \tau^{u_\beta,\mathcal{S},\delta}_l} q^{-1} f(l) \bigg]
= w(x;q, u_\beta,\mathcal{S}),
\end{split}
\end{align}
where the last equation follows from \eqref{wrd} with $x \in (l,b)$ and $\mathcal{S}=[b,r)$. 

\vspace{1mm}
{\it Step 3: The case $x\in [b,r)$ when $b<\beta=r$.} 
We firstly notice that in the case when $b<\beta=r$, the expectation derived to be equal to $v(x;q)$ in \eqref{eq:vxq} of Step 2 when $x\in(l,b)$ is well-defined also for $x=b$. 
Then, we realise from \eqref{wrd} with $x = b$ and $\mathcal{S}=[b,r)$ that  
$v(b;q) = w(b;q,u_\beta,\mathcal{S})$.
It thus follows directly from \eqref{BVPab} that
\begin{align*}
v(x;q) 
= v(b;q) + x - b 
= w(b;q,u_\beta,\mathcal{S}) 
+ x - b
= w(x;q,u_\beta,\mathcal{S}) ,
\end{align*}
where the latter equality is due to the expression \eqref{wrd} of $w(x;q,u_\beta,\mathcal{S},\delta)$ with $x \in \mathcal{S}=[b,r)$.

\vspace{1mm}
{\it Step 4: The case $x\in [b,r)$ when $b=\beta < r$.} 
We have by the definition \eqref{wrqD} of $w(x;q,u_\beta,\mathcal{S})$ for $x\in [\beta,r)$, the expression \eqref{wrd} of $w(x;q,u_\beta,\mathcal{S},\delta)$ with $x \in \mathcal{S}=[\beta,r)$,
and the result of Step 2 that
\begin{align*}
w(x;q,u_\beta,\mathcal{S}) 
= \lim_{\delta\downarrow 0} w\big(x;q,u_\beta,\mathcal{S},\delta\big)
&= \lim_{\delta\downarrow 0} \Big\{ w\big(\beta-\delta; q, u_\beta,\mathcal{S},\delta\big) + x - \beta + \delta \Big\}\\
&=\lim_{\delta\downarrow 0} \Big\{ v(\beta-\delta;q) + x - \beta + \delta \Big\} =v(\beta;q) + x - \beta 
= v(x;q),
\end{align*}
where we used the continuity of $v(\cdot;q)$ in the penultimate equality and \eqref{BVPab} in the last one.
\end{proof}

Building on the results in Lemma \ref{useful-lemma}, we also present the following important technical result on the existence and finiteness of the limiting cost criterion $J(\cdot;u_\beta,[b,r))$ defined by \eqref{eq:limiting-cost-cr}, which also leads to the admissibility of the generalised threshold control strategy $(u_\beta,[b,r),\delta)$.  

\begin{lemma} \label{useful-lemma-2} 
Consider a strong action region $\mathcal{S}=[b,r)$ and a generalised threshold control strategy $(u_\beta,\mathcal{S},\delta)$ satisfying Definition \ref{def:mixedm} with either $l<b<\beta=r$ or $l<b=\beta<r$, such that the associated SDE (cf.~\eqref{Xbim}) has a unique strong solution $X^{u_\beta,\mathcal{S},\delta}$ for all starting values $x \in (l,r)$. 
Suppose that the BVP in Lemma \ref{useful-lemma} has a solution $v(x;q)$ for each $q \in {\rm supp}(F)$, and that there exists a locally bounded function $x \mapsto C(x,q)$ such that for all $x\in (l,r)$ and $q\in {\rm supp}(F)$ we have
\begin{align}\label{hlq3hjl} 
\hspace{-1mm}(1+q)|v''(x;q)|  {\bf 1}_{\{x\not\in \mathcal{J}_u\cup \partial\mathcal{S}\}} 
+ (1+q)|v'(x;q)| 
+ (1+q+q^2)|v(x;q)| 
\leq C(x,q) 
, \;\; 
\int_0^\infty \hspace{-2mm}C(x,q)dF(q)<\infty.
\end{align}
Then the limiting cost criterion $J(\cdot;u_\beta,\mathcal{S})$ defined by \eqref{eq:limiting-cost-cr} exists, is given by  
\begin{align}\label{J=J} 
J(x;u_\beta,\mathcal{S}) 
= V(x) :=\int_0^\infty v(x;q)dF(q), 
\quad x \in [l,r),
\end{align} 
and satisfies 
\begin{align}
&f(x) + \mathbf{A} V(x) - \textstyle{\int_{0}^{\infty}} q v(x;q) dF(q) - u_\beta(x) V'(x) + u_\beta(x) = 0, \quad x \in (\mathcal{W}_u \cup \mathcal{M}_u) \setminus \mathcal{J}_u, 
\label{JBVPa}\\ 
&V(x) = x - b + V(b), \qquad \qquad \qquad \qquad \qquad \qquad \qquad \qquad \qquad x \in \mathcal{S}=[b,r), 
\label{JBVPb}\\
&V(l) = f(l) \textstyle{\int_0^\infty} q^{-1} dF(q), 
\label{JBVPc}\\
& \text{$V(\cdot)\in \mathcal{C}^2\big( (l,r) \setminus (\mathcal{J}_u\cup \partial\mathcal{S}) \big) \cap \mathcal{C}^1(l,r)$ with $|V''(x\pm)|<\infty$ for all $x\in (l,r)$.
}\label{JBVPd}
\end{align}
In particular, $(u_\beta,\mathcal{S},\delta) \equiv (u_\beta,[b,r),\delta) \in \mathcal{A}$, i.e.~it is an admissible control strategy.
\end{lemma}

\begin{proof}
By the definition \eqref{J=J} of $V$ and the result \eqref{w=w} in Lemma \ref{useful-lemma} for the solution $v$ of the BVP, we obtain 
\begin{align} \label{Vx1}
\begin{split}
V(x) 
=\int_0^\infty v(x;q)dF(q)
&=\int_0^\infty w(x;q,u_\beta,\mathcal{S}) dF(q) \\
& = \begin{cases} 
\int_0^\infty w(x;q,u_\beta,\mathcal{S}) dF(q), & \text{for } x\in(l,\beta), \\
\int_0^\infty \lim_{\delta\to 0} w(x;q,u_\beta,\mathcal{S},\delta) dF(q), & \text{for } x \geq \beta,
\end{cases}
\end{split}
\end{align} 
where we used the definition \eqref{wrqD} of $w(\cdot;q,u_\beta,\mathcal{S})$ in the last equality. 
Then, we notice that for any $x\geq \beta$, we have  
\begin{equation} \label{lim=lim}
\int_0^\infty \lim_{\delta\to 0} w(x;q,u_\beta,\mathcal{S},\delta) dF(q) = 
\lim_{\delta\to 0} \int_0^\infty w(x;q,u_\beta,\mathcal{S},\delta) dF(q)
\end{equation}
thanks to the structure of the generalised threshold control strategy $(u_\beta,\mathcal{S},\delta)$ for any $x\geq \beta$ yielding  
$$
w(x;q,u_\beta,\mathcal{S},\delta) 
= w\big(\beta-\delta; q, u_\beta,\mathcal{S},\delta\big) + x - \beta + \delta 
= v(\beta-\delta;q) + x - \beta + \delta,
$$ 
which is then combined with the conditions $|v(x;q)| \leq C(x,q)$ and $\int_0^\infty C(x,q)dF(q)<\infty$ (cf.~\eqref{hlq3hjl}),
and the dominated convergence theorem. 
Using the equality \eqref{lim=lim} in \eqref{Vx1} we obtain  
\begin{align*}
V(x) 
&= \left. \begin{cases} 
\int_0^\infty w(x;q,u_\beta,\mathcal{S}) dF(q), & \text{for } x\in(l,\beta), \\
\lim_{\delta\to 0} \int_0^\infty w(x;q,u_\beta,\mathcal{S},\delta) dF(q), & \text{for } x \geq \beta
\end{cases} \right\}
= J(x;u_\beta, \mathcal{S}),
\end{align*} 
where the latter equality is a direct consequence of the definition \eqref{eq:limiting-cost-cr} of $J(\cdot; u_\beta, \mathcal{S})$.

Finally, given that $v(\cdot,q)$ satisfies \eqref{BVPaa}--\eqref{BVPad} for each $q\in  {\rm supp}(F)$, it follows from the conditions (cf.~\eqref{hlq3hjl})
\begin{align*}
|v''(x;q)| {\bf 1}_{\{x\not\in \mathcal{J}_u\cup \partial\mathcal{S}\}} +|v'(x;q)|+ (1+q)|v(x;q)| \leq C(x,q) 
\quad \text{and} \quad 
\int_0^\infty C(x,q)dF(q)<\infty,
\end{align*}
and the dominated convergence theorem (see also Remark \ref{asdqwrr}) that $V$ satisfies the system \eqref{JBVPa}--\eqref{JBVPd}.
\end{proof}

\begin{remark}\label{asdqwrr}
A close inspection of the proof of Lemma \ref{useful-lemma-2} would reveal that the condition \eqref{hlq3hjl} is stronger than it needs to be for the validity of this lemma. 
The reason we use this integrability condition is because of its usefulness later on in the proof of the verification theorem. 
The condition \eqref{hlq3hjl} is particularly useful for changing the order of differentiation and integration for $V(x)=\int_0^\infty v(x;q)dF(q)$, in the sense that $V'(x)=\int_0^\infty v'(x,q)dF(q)$ and $V''(x)=\int_0^\infty v''(x,q)dF(q)$. 
\end{remark}

We are now ready to present the main result of this section, which is the following verification theorem, providing necessary and sufficient conditions for the optimality of equilibria given by generalised threshold control strategies $(u_\beta,[b,r),\delta)$ satisfying Definition \ref{def:mixedm}.

\begin{theorem}[Verification] \label{main_thm}
Consider a strong action region $\mathcal{S}^*=[b^*,r)$ and a generalised threshold control strategy $(u_\beta^*,\mathcal{S}^*,\delta^*)$ satisfying Definition \ref{def:mixedm} with either $l<b^*<\beta^*=r$ (reflecting boundary $b^*$), or $l<b^*=\beta^*<r$ (inaccessible boundary $b^*=\beta^*$), such that the associated SDE (cf.~\eqref{Xbim}) has a unique strong solution $X^{u_\beta^*,\mathcal{S}^*,\delta^*}$ for all starting values $x \in (l,r)$. 
Suppose that 
the associated BVP in Lemma \ref{useful-lemma} has a solution $v(\cdot;q)$ for each $q \in {\rm supp}(F)$ and the conditions of the associated Lemma \ref{useful-lemma-2} hold true, 
so that the corresponding time-inconsistent limiting cost criterion \eqref{eq:limiting-cost-cr} is given by $J(\cdot;u_\beta^*,\mathcal{S}^*) = V(\cdot)$ defined by \eqref{J=J}.  Then the pair $(u_\beta^*, [b^*,r))$ is an equilibrium strategy if and only if 
\begin{align}
&V'(x) \leq 1, \qquad \qquad \qquad \qquad \qquad \qquad  x \in {\rm int} (\mathcal{W}_{u^*}), \label{mainthmcond1} \tag{I}\\
&V'(x) = 1, \qquad \qquad \qquad \qquad \qquad \qquad  x \in \ol{\mathcal{M}_{u^*}}\setminus \{l,b^*\}, \label{mainthmcond2} \tag{II}\\
&f(x) + \mu(x) - \textstyle{\int_{0}^{\infty}} q v(x;q) dF(q) \geq 0, 
\quad x \in \mathcal{S}=[b^*,r). \label{mainthmcond3} \tag{III}
\end{align}
\end{theorem}
The proof of Theorem \ref{main_thm} is presented in a series of results in Section \ref{main_thm:proof-section}. 

\begin{remark} 
Notice that if we have an equilibrium candidate $(u_\beta, [b,r))=(u_\beta, [\beta,r))$, then $l<b=\beta<r$ which corresponds to a (limiting) mild threshold control strategy, while if we have $(u_\beta, [b,r))=(u_r, [b,r))$, then $l<b<\beta=r$, $u_r(\cdot)$ does not explode, and $b$ is a reflection boundary for $X^{u_r, [b,r), 0}$. The latter case corresponds to a strong threshold control strategy when $u_r \equiv 0$.
\end{remark}

\begin{remark} 
We would like to stress the importance of the equivalence statement in the (Verification) Theorem \ref{main_thm} for applications. 
Essentially, given a proposed control strategy solving the BVP in Lemma \ref{useful-lemma} and satisfying the conditions of Lemma \ref{useful-lemma-2}, if all conditions \eqref{mainthmcond1}--\eqref{mainthmcond3} hold true, then it is an equilibrium strategy. 
If, on the contrary, we show that one of the conditions \eqref{mainthmcond1}--\eqref{mainthmcond3} does not hold true, then we can immediately conclude from Theorem \ref{main_thm} that the proposed control strategy is not an equilibrium strategy. 
This is particularly useful in proving that for some time-inconsistent problems the traditional strong threshold control strategy (Definition~\ref{def:strong}) employed in singular stochastic control problems is not optimal any more, which necessitates the search for equilibria of a different type. 
For further details in this direction refer to our case studies in Section \ref{sec:mixopt}, in particular Theorem \ref{thm:no-pure-NE}. 
\end{remark}

In the following result we show that an equilibrium strategy $(u_\beta, [b,r))$ yields a limiting time-inconsistent equilibrium cost $J(\cdot;u_\beta,[b,r))$ which is globally ${\cal C}^2$.
This regularity is not only useful in terms of applications of our theoretical results (see, e.g.~the case studies in Section \ref{sec:mixopt}), but also implies that the equilibrium cost $J(\cdot;u_\beta,[b,r))$ satisfies the so-called {\it smooth-fit principle} in stochastic singular control problems (in this case, $J''(b;u_\beta,[b,r))=0$) irrespective to whether the threshold $b$ is a strong one (reflecting) or a mild one (entrance-not-exit, thus inaccessible). 

\begin{corollary}\label{coq_beta}
Suppose that $(u_\beta,[b,r))$ is an equilibrium strategy according to Theorem \ref{main_thm}. Then, $V(\cdot) = J(\cdot;u_\beta,[b,r)) \in {\cal C}^2(l,r)$.
\end{corollary}
\begin{proof}
We firstly note that given $(u_\beta,[b,r))$ is an equilibrium strategy as in Theorem \ref{main_thm}, we know that \eqref{mainthmcond1}--\eqref{mainthmcond3} hold true for $V$. 
By using the definition \eqref{eq:mild-string-u} of $\mathcal{W}_u$ in the system \eqref{JBVPa}--\eqref{JBVPd}, and the condition \eqref{mainthmcond2} which implies that $V''(x)=0$ for $x \in {\rm int}(\mathcal{M}_{u})$, we obtain that 
$V(\cdot)\in \mathcal{C}^2\big( (l,r)\setminus(\mathcal{J}_u\cup \partial\mathcal{S}) \big) \cap \mathcal{C}^1(l,r)$ and  
\begin{align} 
&f(x) + \mathbf{A} V(x) - \textstyle{\int_{0}^{\infty}} q v(x;q) dF(q) = 0, \quad x \in \mathcal{W}_u \setminus \mathcal{J}_u, \quad \text{where} \quad {\rm int}(\mathcal{W}_u) \subseteq \mathcal{W}_u \setminus \mathcal{J}_u, \label{JODE1}\\
&f(x) + \mu(x) - \textstyle{\int_{0}^{\infty}} q v(x;q) dF(q) = 0, \qquad x \in 
\mathcal{M}_u \setminus \mathcal{J}_u, \label{JODE2}\\  
&V(x) = x - b + V(b), \qquad \qquad \qquad \qquad \;\; x \in \mathcal{S}=[b,r). \label{JODE3} 
\end{align}
Given that $V(\cdot) \in  \mathcal{C}^1(l,r)$, we have again from the condition \eqref{mainthmcond2} that $V'(x)=1$ for $x\in\partial \mathcal{W}_u \cap \partial \mathcal{M}_u$, so that subtracting \eqref{JODE2} from \eqref{JODE1} yields 
\begin{align*}
\lim_{y \in {\rm int}(\mathcal{W}_u),\, y \uparrow x} \tfrac12 \sigma^2(y) V''(y) =0, \quad x \in \partial \mathcal{W}_{u} \cap \partial \mathcal{M}_{u},
\end{align*}
which implies together with $V''(x)=0$ for $x \in {\rm int}(\mathcal{M}_{u})$ thanks to condition \eqref{mainthmcond2}, as well as $\mathcal{J}_u \subseteq {\rm int}(\mathcal{M}_{u})\cup(\partial \mathcal{W}_{u} \cap \partial \mathcal{M}_{u})$, that 
$V(\cdot)\in \mathcal{C}^2\big( (l,r)\setminus \partial\mathcal{S} \big) \cap \mathcal{C}^1(l,r)$.
Then, observe that $\partial\mathcal{S} \setminus\{r\} = \{b\}$, so that we have:

{\it Case} (a). If $b \in \partial\mathcal{M}_u$, then we clearly have from condition \eqref{mainthmcond2} and \eqref{JODE3} that $V'(b)=1$ and $V''(b)=0$; 

{\it Case} (b). If $b \in \partial\mathcal{W}_u$, then thanks to $V(\cdot) \in  \mathcal{C}^1(l,r)$ and \eqref{JODE3}, we have $V'(b)=1$ and from \eqref{JODE1} that 
\begin{align*}
0 &= \lim_{x \uparrow b} \Big\{f(x) + \frac12 \sigma^2(x) V''(x) + \mu(x) V'(x) - \textstyle{\int_{0}^{\infty}} q v(x;q) dF(q) \Big\} \\
&= f(b) + \frac12 \sigma^2(b) \lim_{x \uparrow b} V''(x) + \mu(b) - \textstyle{\int_{0}^{\infty}} q v(b;q) dF(q) 
\geq \frac12 \sigma^2(b) \lim_{x \uparrow b} V''(x),
\end{align*}
where the last inequality follows from condition \eqref{mainthmcond3}.
Assume (aiming for a contradiction) that $\lim_{x \uparrow b} V''(x) < 0$, so that $V'(x)$ is strictly decreasing as $x \uparrow b$ with $V'(b)=1$, i.e.~$V'(x) > 1$ for $x\in(b-\varepsilon, b)$ and a sufficiently small $\varepsilon>0$. 
This contradicts the fact that $V'(x) \leq 1$ thanks to conditions \eqref{mainthmcond1}--\eqref{mainthmcond2}, hence 
$\lim_{x \uparrow b} V''(x) = 0$. 

In both cases, we obtain $V(\cdot)\in \mathcal{C}^2(l,r)$ and thus complete the proof. 
\end{proof}

\subsection{Proof of Theorem \ref{main_thm}}\label{main_thm:proof-section}

Given the assumptions in the statement of Theorem \ref{main_thm} it is clear (cf.~\eqref{adm-set}, Lemmata \ref{useful-lemma} and \ref{useful-lemma-2}) that the control strategy $(u_\beta^*,\mathcal{S}^*,\delta^*)$ is admissible.
We thus proceed by considering the candidate equilibrium pair $(u^*,\mathcal{S}^*) = (u_\beta^*, [b^*,r))$ with either $l<b^*<\beta^*=r$ or $l<b^*=\beta^*<r$ defined in Theorem \ref{main_thm}, so that
$$
v(\cdot;q) = w(\cdot;q,u_\beta^*,[b^*,r)) 
\quad \text{and} \quad 
V(\cdot) = J(\cdot;u_\beta^*,[b^*,r)).
$$
Then consider a fixed but arbitrary $x \in (l,r)$ and a fixed but arbitrary admissible deviation strategy  $({u_\beta,\mathcal{S},\delta}) \in \mathcal{A}_x$ with associated deviation control process $D^{u_\beta,\mathcal{S},\delta}$ as in \eqref{eq:D-forMS}, 
such that the associated SDE \eqref{Xbim} for the controlled process $X^{u_\beta,\mathcal{S},\delta}$ has a unique strong solution for all $0 \leq t \leq \theta^{u_\beta,\mathcal{S},\delta}_{x,\epsilon}$, where the latter is defined by \eqref{eq:tau-eps-stop}. These are the only processes appearing in this section and to simplify notation, we will refer to them (only in this section) as  
\begin{align}\label{q2rq13ra}
u = u_\beta, 
\quad 
D_t = D_t^{u_\beta,\mathcal{S},\delta}, 
\quad 
X_t = X_t^{u_\beta,\mathcal{S},\delta}
\quad \text{and} \quad 
\theta(\epsilon) = \theta^{u_\beta,\mathcal{S},\delta}_{x,\epsilon}.
\end{align}
In order to prove Theorem \ref{main_thm}, we must show that the equilibrium conditions \eqref{EQ1}--\eqref{EQ2} are equivalent to the conditions \eqref{mainthmcond1}--\eqref{mainthmcond3}.

\subsubsection{The equilibrium condition \eqref{EQ1}}
\label{sec:EQ1}
We firstly analyse the case of $x\in (l,\beta) \cap (\mathcal{S}^c \cup \ul{\partial \mathcal{S}})$, for which the associated deviation control $D$ does not jump immediately (consequently neither $X$ jumps immediately). 
That is, the deviation control strategy corresponds locally to either a (classical) absolutely continuous control strategy with rate function $u$, or to the reflection of the state-process at $x\in\ul{\partial \mathcal{S}}$. Furthermore, recalling Definition \ref{def:mixedm} and that the points in $\ul{\partial \mathcal{S}}$ are finitely many, we have for sufficiently small $\epsilon>0$ that $D$ admits the form
\begin{align} \label{D_form_proof}
D_{\theta(\epsilon)}=\int_0^{\theta(\epsilon)}u(X_t)dt+L^x_{\theta(\epsilon)}(X) {\bf 1}_{\{x\in \ul{\partial \mathcal{S}}\}}.
\end{align}
\begin{proposition} \label{splitL}
For any $x\in (l,\beta) \cap (\mathcal{S}^c \cup \ul{\partial \mathcal{S}})$, the function $L$ defined by \eqref{EQ1} is given by 
\begin{align*} 
L(x) &= \lim_{\epsilon \downarrow 0} \frac{\int_0^\infty \mathbb{E}_x\big[ \int_0^{\theta(\epsilon)} \big\{ \big(v'(X_{t};q) - 1 \big) dD_t - \big((\textbf{A}-q) v(X_t;q) {\bf 1}_{\{X_t\not\in \mathcal{J}_{u^*} \cup \partial\mathcal{S}^*\}} + f(X_t) \big) dt \big\} \big] dF(q) }{\E_{x}\left[\theta(\epsilon)\right]}.
\end{align*}
\end{proposition} 
\begin{proof}
For $x\in (l,\beta) \cap (\mathcal{S}^c \cup \ul{\partial \mathcal{S}})$ we obtain from an application of the It\^{o}-Tanaka formula for continuous semimartingales and the occupation times formula (see e.g.~\cite[Chapter VI, Exercise 1.25, Corollary 1.6]{revuz2013continuous})
thanks to the regularity of $v(\cdot;q)$ in \eqref{BVPad}, that 
\begin{align*}
e^{- q \theta(\epsilon)} v(X_{\theta(\epsilon)}; q) 
&= v(x;q) 
+ \int_0^{\theta(\epsilon)} e^{-qt} (\textbf{A}-q) v(X_t;q) {\bf 1}_{\{X_t\not\in \mathcal{J}_{u^*} \cup \partial\mathcal{S}^*\}} dt \notag\\
&\quad+ \int_0^{\theta(\epsilon)} e^{-qt} v'(X_{t};q)\sigma(X_t) dW_t 
- \int_0^{\theta(\epsilon)} e^{-qt} v'(X_{t};q) dD_t , 
\quad \text{for a sufficiently small $\epsilon>0$.}
\end{align*}
This implies thanks to the continuity of $\sigma$ and $v'$ that 
\begin{align*}
\mathbb{E}_x\bigg[ \int_0^{\theta(\epsilon)} e^{-qt}  v'(X_{t};q) dD_t + e^{- q \theta(\epsilon)} v(X_{\theta(\epsilon)}; q) \bigg] - v(x;q) 
= \mathbb{E}_x\bigg[ \int_0^{\theta(\epsilon)} e^{-qt} (\textbf{A}-q) v(X_t;q) {\bf 1}_{\{X_t\not\in \mathcal{J}_{u^*} \cup \partial\mathcal{S}^*\}} dt \bigg].
\end{align*}
Recalling the definition \eqref{J=J} of $V$ we obtain, using straightforward manipulations, that the above gives 
\begin{align*}
&V(x) - \int_0^\infty \mathbb{E}_x\bigg[ \int_0^{\theta(\epsilon)} e^{-qt} \big(dD_t + f(X_t) dt) + e^{- q \theta(\epsilon)} v(X_{\theta(\epsilon)}; q) \bigg] dF(q) \\
&= \int_0^\infty \mathbb{E}_x\bigg[ \int_0^{\theta(\epsilon)} e^{-qt} \big(v'(X_{t};q) - 1 \big) dD_t - \int_0^{\theta(\epsilon)} e^{-qt} \Big((\textbf{A}-q) v(X_t;q) {\bf 1}_{\{X_t\not\in \mathcal{J}_{u^*} \cup \partial\mathcal{S}^*\}} + f(X_t) \Big) dt \bigg] dF(q),\notag
\end{align*}
which corresponds to the numerator of $L$; to see this, use the left-hand side of the above equality and Lemmata \ref{useful-lemma} and \ref{useful-lemma-2}.
 
Using this expression in the definition \eqref{EQ1} of $L$, we thus obtain that 
\begin{align*}
L(x) &= \lim_{\epsilon \downarrow 0} \frac{\int_0^\infty \mathbb{E}_x\big[ \int_0^{\theta(\epsilon)} \big\{ \big(v'(X_{t};q) - 1 \big) dD_t - \big((\textbf{A}-q) v(X_t;q) {\bf 1}_{\{X_t\not\in \mathcal{J}_{u^*} \cup \partial\mathcal{S}^*\}} + f(X_t) \big) dt \big\} \big] dF(q) }{\E_{x}\left[\theta(\epsilon)\right]} 
\end{align*}
holds true, if we show that the additional term 
\begin{align} \label{L1}
\begin{split} 
\tild L(x) &:= 
\lim_{\epsilon \downarrow 0} \frac{\int_0^\infty \mathbb{E}_x\big[ \int_0^{\theta(\epsilon)} (e^{-qt} - 1) \big(-(\textbf{A}-q) v(X_t;q) {\bf 1}_{\{X_t\not\in \mathcal{J}_{u^*} \cup \partial\mathcal{S}^*\}} - f(X_t) \big) dt \big] dF(q)}{\E_{x}\left[\theta(\epsilon)\right]}\\ 
&\quad + \lim_{\epsilon \downarrow 0} \frac{\int_0^\infty \mathbb{E}_x\big[ \int_0^{\theta(\epsilon)} (e^{-qt} - 1) \big(v'(X_{t};q) - 1 \big) dD_t \big] dF(q)}{\E_{x}\left[\theta(\epsilon)\right]}
= 0.
\end{split}
\end{align}

The remainder of the proof is thus devoted to proving \eqref{L1}. 
To that end, notice first that by using $|e^{-qt}-1| \leq qt$, we can get
\begin{align} \label{limA}
|\tild L(x)| 
&\leq \lim_{\epsilon \downarrow 0} \frac{\int_0^\infty \mathbb{E}_x\big[ \int_0^{\theta(\epsilon)} q t \big|(\textbf{A}-q) v(X_t;q) {\bf 1}_{\{X_t\not\in \mathcal{J}_{u^*} \cup \partial\mathcal{S}^*\}} + f(X_t) \big| dt + \int_0^{\theta(\epsilon)} q t \big| v'(X_{t};q) - 1 \big| dD_t \big] dF(q) }{\E_{x}\left[\theta(\epsilon)\right]}
\notag\\
&\leq \lim_{\epsilon \downarrow 0} \bigg\{ \sup_{y\in[x-\epsilon,x+\epsilon]} \int_0^\infty q \big|(\textbf{A}-q) v(y;q) {\bf 1}_{\{y\not\in \mathcal{J}_{u^*} \cup \partial\mathcal{S}^*\}} + f(y) \big| dF(q) \, \frac{\mathbb{E}_x [\theta(\epsilon)^2]}{\E_{x}\left[\theta(\epsilon)\right]}\bigg\} \notag\\
&\quad + \lim_{\epsilon \downarrow 0} \bigg\{ \sup_{y\in[x-\epsilon,x+\epsilon]} \int_0^\infty q |v'(y;q)-1| dF(q) \, \frac{\E_x\left[ \theta(\epsilon) D_{\theta(\epsilon)} \right]}{\E_x\left[\theta(\epsilon)\right]} \bigg\}.
\end{align}
The first limit on the right-hand side of \eqref{limA} is zero thanks to \eqref{hlq3hjl}, Assumption \ref{Ass:f} of $f$, and \cite[Lemma A.5]{christensen2020time}. 
Hence, it only remains to show that the second limit on the right-hand side of \eqref{limA} is also zero, as this will then imply \eqref{L1}. This is obtained as follows. 

Recalling \eqref{D_form_proof} and using again \eqref{hlq3hjl} and that $F$ has a finite first moment thanks to Assumption \ref{assum:h}, we can obtain by \cite[Lemma 28]{bodnariu2022local} and \cite[Lemma A.5]{christensen2020time} that 
\begin{align*}
\lim_{\epsilon\downarrow 0} \frac{\E_x\left[\theta(\epsilon) D_{\theta(\epsilon)}\right]}{\E_x\left[\theta(\epsilon)\right]} \leq \lim_{\epsilon\downarrow 0} \left\{ \sup_{y\in [x-\epsilon,x+\epsilon]} |u(x)| \, \frac{\E_x\left[\theta(\epsilon)^2\right]}{\E_x\left[\theta(\epsilon)\right]} + {\bf 1}_{\{x \in \ul{\partial \mathcal{S}}\}} \frac{\E_x\big[\theta(\epsilon) L_{\theta(\epsilon)}^x(X) \big]}{\E_x\left[\theta(\epsilon)\right]} \right\} 
= 0 ,
\end{align*} 
which concludes the validity of \eqref{L1}, and hence completes the proof.
\end{proof}

In light of Proposition \ref{splitL}, we next obtain the expression of $L$ in the special case when the candidate equilibrium control strategy exhibits an immediate jump. 

\begin{corollary} \label{lemma_payoutmix} 
For any $x\in (b^*,r) \cap (l,\beta) \cap (\mathcal{S}^c \cup \ul{\partial \mathcal{S}})$, the function $L$ defined by \eqref{EQ1} satisfies 
\begin{align*}
L(x) = - \mu(x)-f(x)+\int_0^\infty q v(x;q) dF(q).
\end{align*}
\end{corollary} 
\begin{proof}
For $x \in (b^*,r) \cap (l,\beta) \cap (\mathcal{S}^c \cup \ul{\partial \mathcal{S}})$, we obtain from \eqref{BVPab} satisfied by $v(\cdot;q)$ for each $q \in {\rm supp}(F)$, that 
$$
(\textbf{A}-q)v(x;q) = \mu(x) - q v(x;q), \quad \text{for all } x\in(b^*,r)
$$ 
and thanks to \eqref{BVPad} that $v'(x;q)=1$ for all $x\in (b^*,r)$ as well, which can be used in the expression of $L$ in Proposition \ref{splitL} to obtain (recall that $F$ is a probability distribution) 
\begin{align*} 
L(x) 
&= \lim_{\epsilon\downarrow 0} \frac{\int_0^\infty \mathbb{E}_x\big[ - \int_0^{\theta(\epsilon)} 
\big(\mu(X_t) - q v(X_t;q) + f(X_t) \big) dt \big] dF(q)}{\mathbb{E}_x[\theta(\epsilon)]} \notag\\
&\; = -\mu(x)-f(x)+\int_0^\infty q v(x;q) dF(q).
\end{align*}
The last equality is obtained using \eqref{hlq3hjl}, the continuity of $\mu$, $f$ and $v$ from Assumptions \ref{Ass:bounds}, \ref{Ass:f} and \eqref{BVPad}, respectively, and Remark \ref{asdqwrr}.
\end{proof}

Finally, we can also obtain the expression of $L$ in the special case when both the candidate equilibrium and deviation control strategies correspond locally to (classical) absolutely continuous control strategies.

\begin{corollary} \label{lemma_bothmix} 
For any $x\in (l,b^*) \cap (l,\beta) \cap \mathcal{S}^c$, the function $L$ defined by \eqref{EQ1} is given by 
\begin{align*}
L(x) = \frac{1}{2}\left(u(x-)+u(x)-u_\beta^*(x-)-u_\beta^*(x)\right)(V'(x)-1).
\end{align*}
\end{corollary}
\begin{proof} 
For $x \in (l,b^*) \cap (l,\beta) \cap \mathcal{S}^c$, we recall that $v(\cdot;q)$ satisfies \eqref{BVPaa} with $u_\beta=u_\beta^*$, for each $q \in {\rm supp}(F)$, and the deviation control strategy is given by $D_t = \int_0^t u(X_t) dt$ for all $0\leq t < \theta(\epsilon)$ and sufficiently small $\epsilon>0$. 
Using this in the expression of $L$ in Proposition \ref{splitL} we obtain for $x \in (l,b^*) \cap (l,\beta) \cap \mathcal{S}^c$ that 
\begin{align*} 
L(x) 
&= \lim_{\epsilon \downarrow 0} \frac{\int_0^\infty \mathbb{E}_x \big[\int_0^{\theta(\epsilon)} \big(v'(X_{t};q) - 1 \big) \big(u(X_t) - u_\beta^*(X_t)\big) dt \big] dF(q)}{\mathbb{E}_x[\theta(\epsilon)]}.   
\end{align*}
Using \eqref{hlq3hjl}, $V \in {\cal C}^1$ thanks to \eqref{JBVPd}, Remark \ref{asdqwrr}, $u$ and $u_\beta^*$ being c\`adl\`ag, and basic properties of diffusions (cf.~\cite[Lemma A.4]{christensen2025time}), we can further conclude that  
\begin{align*}
L(x) 
&= \lim_{\epsilon\downarrow 0}\frac{\E_x\big[ \int_0^{\theta(\epsilon)}(V'(X_t)-1)(u(X_t) - u_\beta^*(X_t)) dt\big]}{\E_x[\theta(\epsilon)]} 
= \frac{1}{2}(V'(x)-1)\left(u(x-)+u(x)-u_\beta^*(x-)-u_\beta^*(x)\right),
\end{align*}
which completes the proof.
\end{proof}

The general result in Proposition \ref{splitL} and both special cases in Corollaries \ref{lemma_payoutmix}--\ref{lemma_bothmix} will be useful in the forthcoming proofs.

\subsubsection{Proof: \eqref{mainthmcond1}--\eqref{mainthmcond3} $\Rightarrow$ \eqref{EQ1} for any admissible deviation control strategy $(u_\beta, \mathcal{S}, \delta)\in \mathcal{A}_{x}$}
\label{lemma_mixingregion}

Recall that for proving \eqref{EQ1}, we consider deviation control strategies corresponding locally to either an absolutely continuous control
or to the reflection of the state-process at $x\in\ul{\partial \mathcal{S}}$.
In what follows, we assume that the conditions \eqref{mainthmcond1}--\eqref{mainthmcond3} hold true and we split the proof into the following two steps. 

\vspace{1mm}
{\it Step 1. Suppose $x\in (l,b^*] \cap (l,\beta) \cap (\mathcal{S}^c \cup  \ul{\partial \mathcal{S}})$}.
It follows from the expression of $L$ in Proposition \ref{splitL}, conditions \eqref{hlq3hjl}, Remark \ref{asdqwrr}, and the regularity $V \in {\cal C}^2$ thanks to conditions \eqref{mainthmcond1}--\eqref{mainthmcond3} and Corollary \ref{coq_beta}, that for any $x\in (l,\beta) \cap (\mathcal{S}^c \cup \ul{\partial \mathcal{S}})$ we have 
\begin{align*}
L(x) 
&= \lim_{\epsilon \downarrow 0} \frac{\int_0^\infty \mathbb{E}_x\big[ \int_0^{\theta(\epsilon)} \big\{ \big(v'(X_{t};q) - 1 \big) dD_t - \big((\textbf{A}-q) v(X_t;q) {\bf 1}_{\{X_t\not\in \mathcal{J}_{u^*} \cup \partial\mathcal{S}^*\}} + f(X_t) \big) dt \big\} \big] dF(q) }{\E_{x}\left[\theta(\epsilon)\right]} \\
&= \lim_{\epsilon \downarrow 0}
\frac{\mathbb{E}_x\big[ 
\int_0^{\theta(\epsilon)} \big\{ \big(V'(X_{t}) - 1 \big) dD_t + \big(-\textbf{A}V(X_t) + \int_0^\infty q v(X_t;q) dF(q) - f(X_t) \big) dt \big\} \big]}{\E_{x}\left[\theta(\epsilon)\right]}.
\end{align*}
Then note that
\begin{equation} \label{lb}
(l,b^*) 
= \mathcal{W}_{u^*} \cup \mathcal{M}_{u^*} 
= {\rm int}({\mathcal{W}}_{u^*}) \cup (\ol{\mathcal{M}_{u^*}} \setminus \{l, b^*\}),
\end{equation}
and recall from \eqref{mainthmcond1}--\eqref{mainthmcond2} that $V'(x)\leq 1$ for all $x\in (l,b^*)$.
Combining the above with \eqref{JBVPb} and \eqref{JBVPd}, we can also conclude that $V'(b^*) = 1$, hence $V'(x)\leq 1$ for all $x\in (l,r)$.
This implies that 
\begin{align} \label{limB}
L(x)
&\leq \lim_{\epsilon\downarrow 0} \frac{\mathbb{E}_x\big[ \int_0^{\theta(\epsilon)} \big(-\textbf{A}V(X_t) + \int_0^\infty q v(X_t;q) dF(q) - f(X_t) \big) dt \big]}{\E_{x}\left[\theta(\epsilon)\right]} \notag\\
&=- \textbf{A}V(x) + \int_0^\infty q v(x;q) dF(q) - f(x), 
\quad \text{for all } x\in (l,b^*],
\end{align} 
where the last equality is obtained using the continuity of $f$ and $v$ from Assumption \ref{Ass:f} and \eqref{BVPad}, respectively, 
the regularity $V \in {\cal C}^2$ thanks to conditions \eqref{mainthmcond1}--\eqref{mainthmcond3} and Corollary \ref{coq_beta}. 
What remains to prove \eqref{EQ1} is to show that the upper bound of $L$ in \eqref{limB} is zero.

To that end, notice from \eqref{lb} and condition \eqref{mainthmcond2} that 
$$
u_\beta^*(x)(1 - V'(x))=0, \quad 
\text{thanks to } \begin{cases}
u_\beta^*(x)=0, &\text{for } x \in {\rm int}({\mathcal{W}}_{u^*}), \\
V'(x)=1, &\text{for } x \in \overline{\mathcal{M}_{u^*}} \setminus \{l,b^*\} .
\end{cases}
$$
Using this together with \eqref{JBVPa}, we obtain that the upper bound of $L$ in \eqref{limB} satisfies for each $x\in (l,b^*)$ that    
\begin{align} \label{limB=0}
\mathbf{A} V(x) - \textstyle{\int_{0}^{\infty}} q v(x;q) dF(q) + f(x) 
= \mathbf{A} V(x) - \textstyle{\int_{0}^{\infty}} q v(x;q) dF(q) + f(x) - u^*_\beta(x) V'(x) + u^*_\beta(x) 
= 0.
\end{align}
Given that the limit of $\mathbf{A} V(x) - \textstyle{\int_{0}^{\infty}} q v(x;q) dF(q) + f(x)$ as $x \to b^*$ exists, again thanks to $V\in\mathcal{C}^2(l,r)$, we can conclude that the upper bound of $L$ in \eqref{limB} is zero also for $x=b^*$, 
which thus completes the proof of this step.

\vspace{1mm}
{\it Step 2. Suppose $x\in (b^*,r) \cap (l,\beta) \cap (\mathcal{S}^c \cup  \ul{\partial \mathcal{S}})$}.
For such values of $x$ the candidate equilibrium control strategy corresponds to an immediate jump. 
The validity of \eqref{EQ1} in this case is an immediate consequence of Corollary \ref{lemma_payoutmix} and condition \eqref{mainthmcond3}. 
This concludes the proof of the desired equivalence.

\subsubsection{Proof: \eqref{mainthmcond1}--\eqref{mainthmcond3} $\Rightarrow$ \eqref{EQ2} for any admissible deviation control strategy $(u_\beta, \mathcal{S}, \delta)\in \mathcal{A}_{x}$.}

Assume that conditions \eqref{mainthmcond1}--\eqref{mainthmcond3} hold true. 
Recall from \eqref{mainthmcond1}--\eqref{mainthmcond2} and \eqref{eq:mild-string-u} that $V'(x)\leq 1$ for all $x\in (l,b^*)$, which gives together with \eqref{JBVPb} and $V \in {\cal C}^1$ from \eqref{JBVPd} (thanks to the validity of Lemma \ref{useful-lemma-2}) that $V'(x)\leq 1$ for all $x\in (l,r)$.
This further implies, for any $l \leq y \leq x < r$ with $y\in \mathbb{R}$, that
$$
V(x) - V(y) = \int_y^x V'(s) ds \leq x - y .
$$
Hence, it particularly holds for $y=b_i \in \ul{\partial \mathcal{S}}$ in case $x \in \mathcal{S}_i$, and for $y=\beta-\delta$ in case $x \in [\beta,r)$.
This proves that \eqref{EQ2} holds true.

\subsubsection{Proof: \eqref{EQ1}--\eqref{EQ2} for any admissible deviation control strategy $(u_\beta, \mathcal{S}, \delta)\in \mathcal{A}_{x}$ $\Rightarrow$  \eqref{mainthmcond1}--\eqref{mainthmcond3}.}

Since \eqref{EQ1}--\eqref{EQ2} hold true for all admissible deviation controls, they particularly hold true for all classical control strategies $(u_\beta,\mathcal{S},\delta) = (u_r, \emptyset, 0)$ (see Remark \ref{rem:controls}), such that the control rate $u_r(x) = u$ for a constant $u>0$ and all $x\in(l,r)$.
Firstly, notice that thanks to $\beta=r$ and $\mathcal{S} = \emptyset$, the condition \eqref{EQ2} is void for these deviation control strategies.
In what follows, we show that considering this set of controls is sufficient to prove that the condition \eqref{EQ1} implies that  \eqref{mainthmcond1}--\eqref{mainthmcond3} hold true. 

To that end, notice that for any deviation control strategy $(u_r, \emptyset, 0)$ such that $u_r \equiv u > 0$, the state-space can be written as 
\begin{equation} \label{lr}
(l,r) = (l,b^*)
\cup [b^*,r), 
\quad \text{where} \quad 
(l,b^*) = {\rm int}({\mathcal{W}}_{u^*})
\cup (\overline{\mathcal{M}_{u^*}} \setminus \{l,b^*\}).
\end{equation}
We also recall from Corollary \ref{lemma_bothmix} and the condition \eqref{EQ1} for the deviation control strategy $(u, \emptyset, 0)$, that for any $x\in (l,b^*)$ we have  
\begin{align} \label{u-L<0}
L(x) = \frac{1}{2}\left(2 u - u_\beta^*(x-) - u_\beta^*(x) \right)(V'(x)-1) \leq 0. 
\end{align}
In what follows, we split the proof into three steps, each proving the validity of a condition from \eqref{mainthmcond1}--\eqref{mainthmcond3}. 

\vspace{1mm}
{\it Step 1. Suppose $x\in {\rm int}({\mathcal{W}_{u^*}})$}. 
Recall from the definition \eqref{eq:mild-string-u} that $u_\beta^*(x-) + u_\beta^*(x) = 0$, hence we conclude from \eqref{lr}, $\eqref{u-L<0}$ and the positivity of $u$ that $V'(x) \leq 1$, which implies that \eqref{mainthmcond1} holds true.

\vspace{1mm}
{\it Step 2. Suppose $x\in \overline{\mathcal{M}_{u^*}} \setminus \{l,b^*\}$}.  
Recall from the definition \eqref{eq:mild-string-u} that $u_\beta^*(x-) + u_\beta^*(x) > 0$ for all $x\in \mathcal{M}_{u^*}$. Given that for any such $x\in \mathcal{M}_{u^*}$, we can always find $u(x) \equiv u >0$ such that 
$2 u - u_\beta^*(x-) - u_\beta^*(x) > 0,$
and since $\eqref{u-L<0}$ must hold true for all $u>0$, we conclude that $V'(x) = 1$ must hold true for all $x \in \mathcal{M}_{u^*}$. 
Combining this with the property $V \in {\mathcal C}^1(l,r)$ from \eqref{JBVPd}, we can further conclude thanks to \eqref{lr} that $V'(x) = 1$ for all $x \in 
\overline{\mathcal{M}_{u^*}} \setminus \{l,b^*\}$, which implies that \eqref{mainthmcond2} holds true.

\vspace{1mm}
{\it Step 3. Suppose $x \in [b^*,r)$}. 
Recall from Corollary \ref{lemma_payoutmix} and the condition \eqref{EQ1} that 
\begin{align*}
L(x) 
= - \mu(x) - f(x) + \int_0^\infty q v(x;q) dF(q) \leq 0, \quad x \in (b^*,r).
\end{align*}
Thanks to the continuity of $\mu$, $f$ and $v$ from Assumptions \ref{Ass:bounds}, \ref{Ass:f} and \eqref{BVPad}, respectively, we obtain that $L(b^*) \leq 0$ as well, which implies that \eqref{mainthmcond3} holds true. 
\hfill \qedsymbol{}

\section{Case studies with strong and mild threshold control strategies}
\label{sec:mixopt}

In this section, we apply the general theory developed in the preceding sections to a certain time-inconsistent stochastic singular control problem. We examine two case studies in detail, one 
with an equilibrium control given by a strong threshold control strategy (as in the traditional literature on singular control problems), and another one with a novel equilibrium control given by a mild threshold control strategy. 

In mathematical terms, we consider $X$ to be a controlled geometric Brownian motion satisfying  
\begin{align}\label{example_GBM}
dX^D_t=\sigma X^D_tdW_t-dD_t, \quad X^D_{0-}= x \in (l,r) := (0,\infty),
\end{align}
for a constant $\sigma>0$.
We therefore have $\mu(x)=0$ and $\sigma(x)=\sigma x$, for all $x>0$ in \eqref{X} and the infinitesimal generator takes the form $\textbf{A} = \frac{1}{2}\sigma^2 x^2 \frac{d^2}{d x^2}$ in this section. 
We further consider in \eqref{eq:h-function} the pseudo-exponential discount function 
\begin{align*}
h(t) = \tfrac12 e^{-q_1 t} + \tfrac12 e^{-q_2 t}, \quad t \geq 0, \quad 0<q_1<q_2, 
\end{align*} 
so that either a low- or high-discount-rate regime occurs with an equal probability.
We consider a quadratic running cost function $f:[0,\infty) \mapsto [0,\infty)$ defined by $f(x):=\frac12 x^2$.  Note that Assumptions \ref{Ass:bounds}, \ref{Ass:f} and \ref{assum:h} are easily verified for this setup. 

Hence, given an admissible control strategy  $(u_\beta,\mathcal{S},\delta) \in \mathcal{A}$ from \eqref{adm-set} with associated control process  $D^{u_\beta,\mathcal{S},\delta}$ given by \eqref{eq:D-forMS}, the cost criterion (cf.~\eqref{wr} and \eqref{wrd}) associated to each $q_i$, $i=1,2$, is given (due to $l=0$ and $f(l)=0$) by 
\begin{align} \label{ex-wrd} 
\begin{split}
w(x;q_i,u_\beta,\mathcal{S},\delta) = \mathbb{E}_x\bigg[ &\int_0^{\tau^{u_\beta,\mathcal{S},\delta}_0} e^{-q_i t} \, \tfrac12 \, (X^{u_\beta,\mathcal{S},\delta}_t)^2 \, dt + \int_0^{\tau^{u_\beta,\mathcal{S},\delta}_0} e^{-q_i t} \, d D^{u_\beta,\mathcal{S},\delta}_t \bigg], 
\quad x \in (0,\infty).
\end{split}
\end{align}
The associated time-inconsistent cost criterion (cf.~\eqref{cost_equation}) thus becomes 
\begin{align}\label{ex:cost-crit-adm} 
J(x;u_\beta,\mathcal{S},\delta) = \tfrac12 \, w(x; q_1,u_\beta,\mathcal{S},\delta) + \tfrac12 \, w(x; q_2,u_\beta,\mathcal{S},\delta), 
\quad x \in (0,\infty). 
\end{align}
Therefore, the time-inconsistent singular stochastic control problem takes the form: 
\begin{align} \label{TISC}
\begin{split}
&\qquad \quad \;\; \text{Find an equilibrium strategy $(u_\beta^*, \mathcal{S}^*)$ 
satisfying Definition \ref{def:equ_stop_time} for} \\ &\text{$w(\cdot;q,u^*,\mathcal{S}^*)$ and $J(\cdot; u^*,\mathcal{S}^*)$ defined by \eqref{wrqD} and \eqref{eq:limiting-cost-cr} in accordance with \eqref{ex-wrd}--\eqref{ex:cost-crit-adm}}.
\end{split}
\end{align}

Throughout Section \ref{sec:mixopt} we consider the following standing assumption.
\begin{assumption}\label{assum:appli}
The model parameters satisfy $q_2>q_1>\sigma^2$.
\end{assumption}

In what follows, we present two case studies. In Section \ref{sec:app-pure} we present a case study in a parameter constellation that leads to an equilibrium given by a strong threshold control strategy; recall that this corresponds to a reflecting boundary (strong threshold) for the controlled process. 
In Section \ref{sec:mixed-case}, we present a case study in a different parameter constellation that instead leads to an equilibrium control strategy given by a mild threshold control strategy; recall that this corresponds to a control process absolutely continuous with respect to the Lebesgue measure, with a rate which explodes at a point in the interior of the state-space and creates an entrance-not-exit boundary (mild threshold) for the controlled process.

\subsection{Case study with strong threshold control equilibrium strategy} 
\label{sec:app-pure}

Recall from Definition \ref{def:strong} that a strong threshold control strategy results in a controlled process $X^D$ given by the unique strong solution to the SDE \eqref{Xb} (cf.~Section \ref{sec:strong}), which under the dynamics of \eqref{example_GBM} becomes 
$$
X_t^{D} = x + \sigma \int_0^t X_s^{D} \, dW_s - (x-b)^+ - L^b_t(X^{D}), 
\quad  0 \leq t \leq \tau^D_0.
$$
The decision maker should thus optimally choose $b \in (0,\infty)$, such that the strong threshold control strategy $(0,[b,\infty), 0)$ yields an equilibrium strategy $(0,[b,\infty))$ as in Definition \ref{def:equ_stop_time}. This strategy would then split the state-space into a waiting region $\mathcal{W}=(0,b)$ where no control is exerted, and a strong action region $\mathcal{S}=[b,\infty)$ where the minimal amount of control is exerted to maintain the controlled process $X^D$ in the closure $\overline{\mathcal{W}}=(0,b]$ of the waiting region, $\mathbb{P}$-a.s.~for all $t \geq 0$. 
The latter prescribes the downward reflection of $X^D$ at $b$ and a potential initial downward jump to $b$ at time $0$ if it starts from $X^D_{0-} = x > b$. 

In the forthcoming analysis, we firstly construct our candidate strong threshold control equilibrium strategy in Section \ref{sec:inform-strong}. 
Subsequently, we prove in Section \ref{sec:Equib}, by relying on our general verification theorem, that the constructed candidate leads to an equilibrium in the case study when the difference between the two possible discount rates $q_2$ and $q_1$ is sufficiently small; see Theorem \ref{prop_pure_ex} and an illustration in Figure \ref{fig:purework1}--\ref{fig:purederwork1}. 
As a by-product of our general theory, by using the equivalence relationship in our general verification theorem, we can also prove that there is no equilibrium in the form of strong threshold control strategies in the case study when the difference between $q_2$ and $q_1$ is sufficiently large; see Theorem \ref{thm:no-pure-NE} and an illustration in Figure \ref{fig:purenotwork1}--\ref{fig:purdernotworke1}. This is the fundamental step that motivates and necessitates our search for an equilibrium in the form of mild threshold control strategies and paves the way for the analysis of this latter case study in Section \ref{sec:mixed-case}.

\subsubsection{Construction of a candidate equilibrium via a strong threshold control strategy}
\label{sec:inform-strong}

In light of Theorem \ref{main_thm}, we derive our candidate strong threshold control equilibrium strategy $(0,[b^*,\infty))$ from the values $v(\cdot;q_i)$, $i=1,2$, and strong threshold $b^*$, that solve the associated BVP from Lemma \ref{useful-lemma}, which takes the form 
\begin{align}
&\tfrac12 x^2 + \tfrac12 \sigma^2 x^2 v''(x;q_i) - q_i v(x;q_i) = 0, \quad x \in (0,b^*), 
\label{xBVPaa}\\
&v(x;q_i) = x - b^* + v(b^*;q_i), \qquad \qquad \quad \; x \in [b^*,\infty), 
\label{xBVPab}\\
&v(0;q_i) = 0, 
\label{xBVPac}\\
& \text{$v(\cdot;q_i)\in \mathcal{C}^2\big( (0,\infty) \setminus \{b^*\} \big) \cap \mathcal{C}^1(0,\infty)$ with $|v''(b^*-;q_i)|<\infty$.}
\label{xBVPad}
\end{align}
Then, we will know from \eqref{w=w} that the associated limiting cost criterion $w(\cdot;q_i,0,[b^*,\infty))$ defined by \eqref{wrqD} exists and is given by $w(x;q_i,0,[b^*,\infty)) = v(x;q_i)$, for all $x \in (0,\infty)$ and $i=1,2$.

Combining the general solution to the ODE in \eqref{xBVPaa} with the condition \eqref{xBVPac} implies that 
\begin{align}\label{w_pure_def}
v(x;q_i)=A_ix^{\gamma_i}+\frac{x^2}{2(q_i-\sigma^2)}, \quad x\in (0,b^*), 
\quad \text{where} \quad 
\gamma_i:=\frac{1}{2}\bigg(1+\sqrt{1+\frac{8q_i}{\sigma^2}} \bigg),\quad i=1,2,
\end{align}
for some constant $A_i$ to be determined, and $\gamma_i>2$, thanks to Assumption \ref{assum:appli}, for both $i=1,2$. 
Then, using the continuous differentiability of $v$ from \eqref{xBVPad} and $v'(b^*+;q_i) = 1$ from \eqref{xBVPab} we obtain that $v'(b^*-;q_i) = 1$ as well, which yields via \eqref{w_pure_def} that 
\begin{align} \label{Ai}
A_i = \Big(1-\frac{b^*}{q_i-\sigma^2} \Big) \gamma_i^{-1} (b^*)^{1-\gamma_i} , 
\quad i=1,2.
\end{align}
By noticing that \eqref{hlq3hjl} is straightforwardly satisfied, we can also conclude from Lemma \ref{useful-lemma-2} that the associated limiting cost criterion $J(\cdot;0,[b^*,\infty))$ defined by \eqref{eq:limiting-cost-cr} exists and is given by (cf.~\eqref{J=J})
\begin{equation} \label{J_pure_def}
J(x;0,[b^*,\infty)) = V(x) := \tfrac12 \, v(x; q_1) + \tfrac12 \, v(x; q_2), \quad x \in (0,\infty).
\end{equation}

Building further in the spirit of Theorem \ref{main_thm}, if our candidate $(0,[b^*,\infty))$ ends up being an equilibrium, then we will know from Corollary \ref{coq_beta} that  $V \in \mathcal{C}^2(0,\infty)$ and thus $V''(b^*)=0$. 
This implies via \eqref{J_pure_def} that a necessary condition for equilibrium is
\begin{align}\label{sfsdfdsas}
0 
= 2 V''(b^*-) = v''(b^*-;q_1) + v''(b^*-;q_2) = \sum_{i=1}^2 \Big( A_i \gamma_i (\gamma_i-1) (b^*)^{\gamma_i-2} + \frac{1}{q_i-\sigma^2} \Big),
\end{align}
which therefore gives that our candidate strong threshold $b^*$ must be given by 
\begin{align} \label{xhat_def}
b^*=\frac{(\gamma_1 + \gamma_2 - 2)(q_1 - \sigma^2)(q_2 - \sigma^2)}{(\gamma_1 - 2)(q_2 - \sigma^2) + (\gamma_2 - 2)(q_1 - \sigma^2)}
>0,
\end{align}
whose positivity follows straightforwardly from Assumption \ref{assum:appli} and its implication that $\gamma_i>2$, $i=1,2$.

\subsubsection{Small discount rate difference: Strong threshold control equilibrium strategy} 
\label{sec:Equib}

In this section we investigate the conditions under which our candidate $(0,[b^*,\infty))$ corresponding to $b^*$ defined by \eqref{xhat_def} is indeed an equilibrium. 
To that end, we firstly present an important result regarding the properties of the candidate value function $V$.

\begin{proposition} \label{non_ex_equilibrum}
Recall $V$ and $b^*$ defined by \eqref{J_pure_def} and \eqref{xhat_def}, respectively. 
\begin{enumerate}[\rm (i)]
\item For each $q_1$, there exits a constant $h>0$, such that if 
$q_2-q_1 < h$, then $V'(x)<1$ for all $x\in(0,b^*)$; 
\item For each $q_1$, there exits a constant $h>0$, such that if 
$q_2-q_1 \geq h$, then $V'(x)>1$ for $x\in (b^*-\epsilon, b^*)$ and a sufficiently small $\epsilon>0$.
\end{enumerate}
\end{proposition}
\begin{proof}
It is easy to see from \eqref{w_pure_def} and \eqref{J_pure_def}, as well as by construction, that 
\begin{equation} \label{V''0}
V''(0+) = \frac{1}{2(q_1-\sigma^2)} + \frac{1}{2(q_2-\sigma^2)} > 0 
\quad \text{and} \quad V''(b^*) = 0.
\end{equation}
Then, we also observe directly from \eqref{w_pure_def} that $V \in \mathcal{C}^3(0,b^*)$, and 
\begin{align*}
V'''(x) = 
\tfrac{1}{2} x^{\gamma_1-3} \left( A_1\gamma_1(\gamma_1-1)(\gamma_1-2) + A_2\gamma_2(\gamma_2-1)(\gamma_2-2) x^{\gamma_2-\gamma_1}\right),  \quad x \in (0,b^*),
\end{align*}
yielding that $V'''(\cdot)$ has at most one zero on $(0,b^*)$. 
Combining this with \eqref{V''0}, it is possible to verify that
\begin{align} \label{V'''zeros}
\text{either} \quad 
V''(x)>0, \quad x \in (0,b^*), 
\quad \text{or} \quad 
V''(x) \begin{cases} > 0 , & x \in (0, x_m), \\
< 0 , & x \in (x_m,b^*),
\end{cases}
\quad \text{for some } x_m \in (0,b^*).
\end{align}
By differentiating the ODEs \eqref{xBVPaa} with respect to $x$ for each $i=1,2$ and adding them, we obtain from \eqref{J_pure_def} that 
\begin{align*}
&2x + \sigma^2 x^2 V'''(x) + 2 \sigma^2 x V''(x) - q_1 v'(x;q_1) - q_2 v'(x;q_2) = 0, \quad x \in (0,b^*).
\end{align*}
Hence, given that $v'(b^*;q_i) = 1$ thanks to \eqref{xBVPad} and $V''(b^*)=0$ thanks to \eqref{V''0}, we obtain that  
\begin{align} \label{V'''b}
&\lim_{x \uparrow b^*} V'''(x) 
= \frac{H(q_2;q_1)}{\sigma^2 b^*(q_2;q_1)^2},
\quad \text{where \;$H(q_2;q_1):=q_1 + q_2 - 2b^*(q_2;q_1)\;$ and $\;b^*(q_2;q_1) \equiv b^*$ from \eqref{xhat_def}.}
\end{align}
It thus follows from the above that the sign of $\lim_{x \uparrow b^*} V'''(x)$ is given for each fixed $q_1 > \sigma^2$ by the sign of $H(\cdot;q_1)$ on $(q_1,\infty)$.

We observe from $\gamma_1 = \gamma(q_1)$ in \eqref{gammaq}, Lemma \ref{nec_inequality} and Assumption \ref{assum:appli}, that   
\begin{align} \label{Hq1}
H(q_1;q_1) 
= 2 \big(q_1 - b^*(q_1;q_1)\big)
= \frac{2 ( \sigma^2 \gamma(q_1) - \sigma^2 - q_1)}{\gamma(q_1) - 2} < 0 
\end{align}
as well as, by using $q_2=q_1+h$ for $h>0$ (without loss of generality) so that $\gamma_2 = \gamma(q_1 + h)$ in \eqref{gammaq}, that  
\begin{align} \label{Hinf}
\begin{split}
\lim_{h \to \infty} H(q_1+h;q_1) 
&= \lim_{h \to \infty} \Big\{ 2 q_1 + h - \frac{2\big(\gamma(q_1) + \gamma(q_1+h) - 2\big)(q_1 - \sigma^2)(q_1+h - \sigma^2)}{\big(\gamma(q_1) - 2\big)(q_1+h - \sigma^2) + \big(\gamma(q_1+h) - 2\big)(q_1 - \sigma^2)} \Big\} \\
&= \lim_{h \to \infty} \big\{ h \big(1 + 2 \psi(h)\big) \big\} = \infty,
\end{split}
\end{align}
where we define $\psi:(0,\infty)\mapsto \R$ by  
\begin{align*}
\psi(h) 
&:= \frac{q_1}{h} - \Big(\frac{\gamma(q_1) + \gamma(q_1+h) - 2}{h} \Big) \Big(\frac{\gamma(q_1) - 2}{q_1 - \sigma^2} + \frac{\gamma(q_1+h) - 2}{q_1+h - \sigma^2} \Big)^{-1} , 
\quad \text{such that } \lim_{h \to \infty} \psi(h) = 0. 
\end{align*}
It therefore follows from \eqref{Hq1}--\eqref{Hinf} and the continuity of $H(\cdot;q_1)$ on $(q_1,\infty)$ that: 

(i) There exists $h>0$ such that $H(q_2;q_1)<0$ for all $q_2 \in (q_1, q_1 + h)$, and consequently from \eqref{V'''b} that $\lim_{x \uparrow b^*} V'''(x) < 0$. 
Hence, for $q_2 - q_1 < h$, we have  that $V''(\cdot)$ is decreasing in $(b^*-\epsilon, b^*)$ for a sufficiently small $\epsilon>0$.
Combining this with \eqref{V''0}, we obtain that $V''(x)>0$ for $x \in (b^*-\epsilon, b^*)$, which contradicts the second possibility in \eqref{V'''zeros}, thus $V''(x)>0$, for all $x \in (0,b^*)$. 
This in turn yields that $V'(\cdot)$ is strictly increasing on $(0,b^*)$ and given that $V'(b^*)=1$, we conclude that $V'(x) < 1$ for all $x \in (0,b^*)$.

(ii) There exists $h>0$ such that $H(q_2;q_1)>0$ for all $q_2 \in (q_1 + h,\infty)$, and consequently from \eqref{V'''b} that $\lim_{x \uparrow b^*} V'''(x) > 0$. 
Hence, for $q_2 - q_1 > h$, we have that $V''(\cdot)$ is increasing in $(b^*-\epsilon, b^*)$ for a sufficiently small $\epsilon>0$.
Combining this with \eqref{V''0}, we obtain that $V''(x)<0$ for $x \in (b^*-\epsilon, b^*)$, which contradicts the first possibility in \eqref{V'''zeros}.
This in turn yields that $V'(\cdot)$ is strictly decreasing on $(b^*-\epsilon, b^*)$ and given that $V'(b^*)=1$, we conclude that $V'(x) > 1$ for $x \in (b^*-\epsilon, b^*)$.
\end{proof}

We are now ready to present our main result, according to which we have an equilibrium strategy $(0,[b^*,\infty))$ given in terms of the strong threshold control strategy $(0,[b^*,\infty),0)$, when the difference $q_2-q_1$ of the possible discount rates is sufficiently small. 
Its proof is based on an application of our general verification Theorem \ref{main_thm}. 

\begin{theorem}\label{prop_pure_ex} 
For each $q_1$, there exits a constant $h>0$ such that if 
$q_2-q_1\leq h$, then the strong threshold control strategy $(0,[b^*,\infty),0)$ corresponding to $b^*$ given by \eqref{xhat_def} leads to an equilibrium strategy $(0,[b^*,\infty))$ for the time-inconsistent singular control problem defined by \eqref{TISC}.
\end{theorem}

\begin{proof} 
Note that the SDE corresponding to the strong threshold control strategy $(0,[b^*,\infty),0)$ has a unique strong solution (cf.~Section \ref{sec:strong}).
Moreover, note that the solution $v(\cdot;q_i)$ to the associated BVP from Lemma \ref{useful-lemma}, which takes the form \eqref{xBVPaa}--\eqref{xBVPad}, is constructed by \eqref{w_pure_def}, \eqref{Ai} and \eqref{xhat_def}, and satisfies $w(\cdot;q_i,0,[b^*,\infty)) = v(\cdot;q_i)$ for each $i=1,2$. 
We also note that condition \eqref{hlq3hjl} is straightforwardly verified, and we further obtain that $J(\cdot;0,[b^*,\infty))$ satisfies \eqref{J_pure_def}, by Lemma \ref{useful-lemma-2}. 
Hence, all that remains is to verify that conditions \eqref{mainthmcond1}--\eqref{mainthmcond3} of the (verification) Theorem~\ref{main_thm} with $0<b^*<\beta^*=\infty$  hold true (recalling that $l=0$, $r=\infty$ and \eqref{xhat_def}).

\vspace{1mm}
{\it Verification of condition \eqref{mainthmcond1}}. This follows from the result that there exits, for each $q_1$, a constant $h>0$ such that if $q_2-q_1 < h$, then $V'(x)<1$ for all $x\in(0,b^*)={\rm int}(\mathcal{W}_0)$ (recall that $u_\beta\equiv0$) thanks to Proposition \ref{non_ex_equilibrum}.(i), combined with the fact that $V'(b^*) = 1$ due to its definition \eqref{J_pure_def} and $v'(b^*;q_i) = 1$, $i=1,2$, which follows from \eqref{xBVPab} and \eqref{xBVPad}.

\vspace{1mm}
{\it Verification of condition \eqref{mainthmcond2}}. 
This is trivial since the mild action region $\mathcal{M}_0 \equiv \emptyset$ in \eqref{eq:mild-string-u}, given that $u_\beta \equiv 0$ for the candidate equilibrium strategy $(0,[b^*,\infty), 0)$. 

\vspace{1mm}
{\it Verification of condition \eqref{mainthmcond3}}.
Denote the left-hand side of \eqref{mainthmcond3} by (recall that $\mu(\cdot) \equiv 0$ and $f(x)=\frac{x^2}{2}$)
\begin{align*}
\phi(x):
& =\tfrac{1}{2} x^2 - \tfrac{1}{2} q_1 v(x;q_1) - \tfrac{1}{2} q_2 v(x;q_2) , 
\quad x\in[b^*,\infty).
\end{align*}
Differentiating $\phi(x)$ and using the fact that $v'(x;q_i) = 1$ for $x\in[b^*,\infty)$, thanks to \eqref{xBVPab}, we obtain 
\begin{align*}
\phi'(x) = x - \tfrac12 (q_1+q_2) > b^* - \tfrac12 (q_1+q_2), 
\quad x\in (b^*,\infty).
\end{align*}
Recalling its definition \eqref{xhat_def} we write $b^* = b^*(q_2;q_1)$, and using the continuity of $b^*(\cdot;q_1)$ and the fact that $b^*(q_1;q_1) - q_1 =\frac{q_1 - \sigma^2 \gamma_1 + \sigma^2}{\gamma_1-2} > 0$ (thanks to Lemma \ref{nec_inequality} and Assumption \ref{assum:appli}), we can see that there exists $h>0$ such that $b^*(q_2;q_1) - \tfrac12(q_1+q_2)>0$ for all $q_2 \in (q_1, q_1 + h)$, so that $\phi'(\cdot)>0$ on $(b^*,\infty)$ for such $q_2$. 
Combining this with 
\begin{align*}
\phi(b^*) 
= \tfrac{1}{2} (b^*)^2 - \tfrac{1}{2} q_1 v(b^*;q_1) - \tfrac{1}{2} q_2 v(b^*;q_2) 
= -\tfrac{1}{2} (b^*)^2 \sigma^2  V''(b^*) 
= 0,
\end{align*}
where the penultimate equality follows from adding together the ODEs \eqref{xBVPaa} for $i=1,2$ and the latter by \eqref{sfsdfdsas}. We conclude that $\phi(x) \geq 0$ for all $x \in [b^*,\infty)$.
\end{proof} 

\begin{figure}[tb] \label{fig1}
 \subfigure[]{\includegraphics[width=0.5\textwidth]{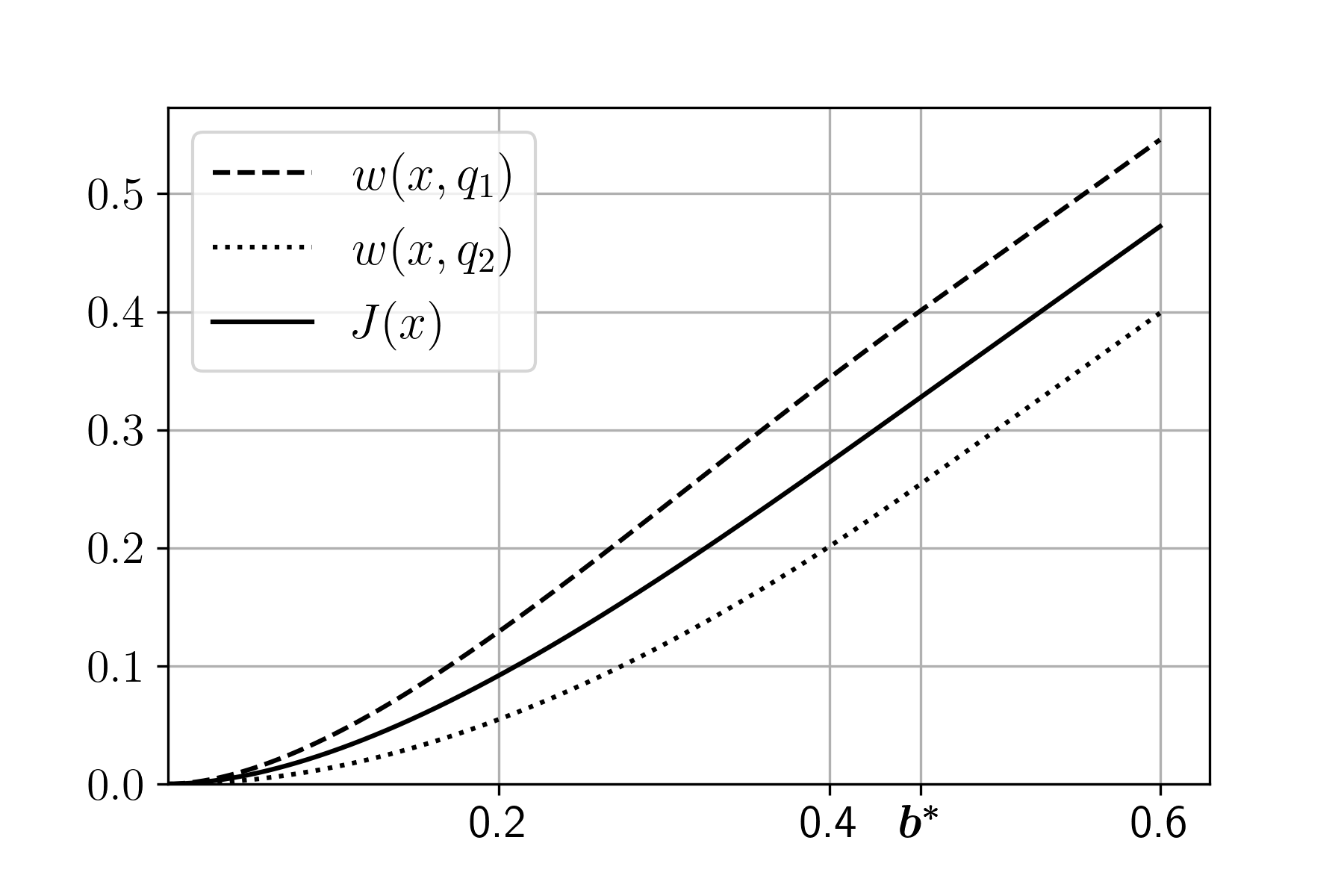}\label{fig:purework1}}
\subfigure[]{\includegraphics[width=0.5\textwidth]{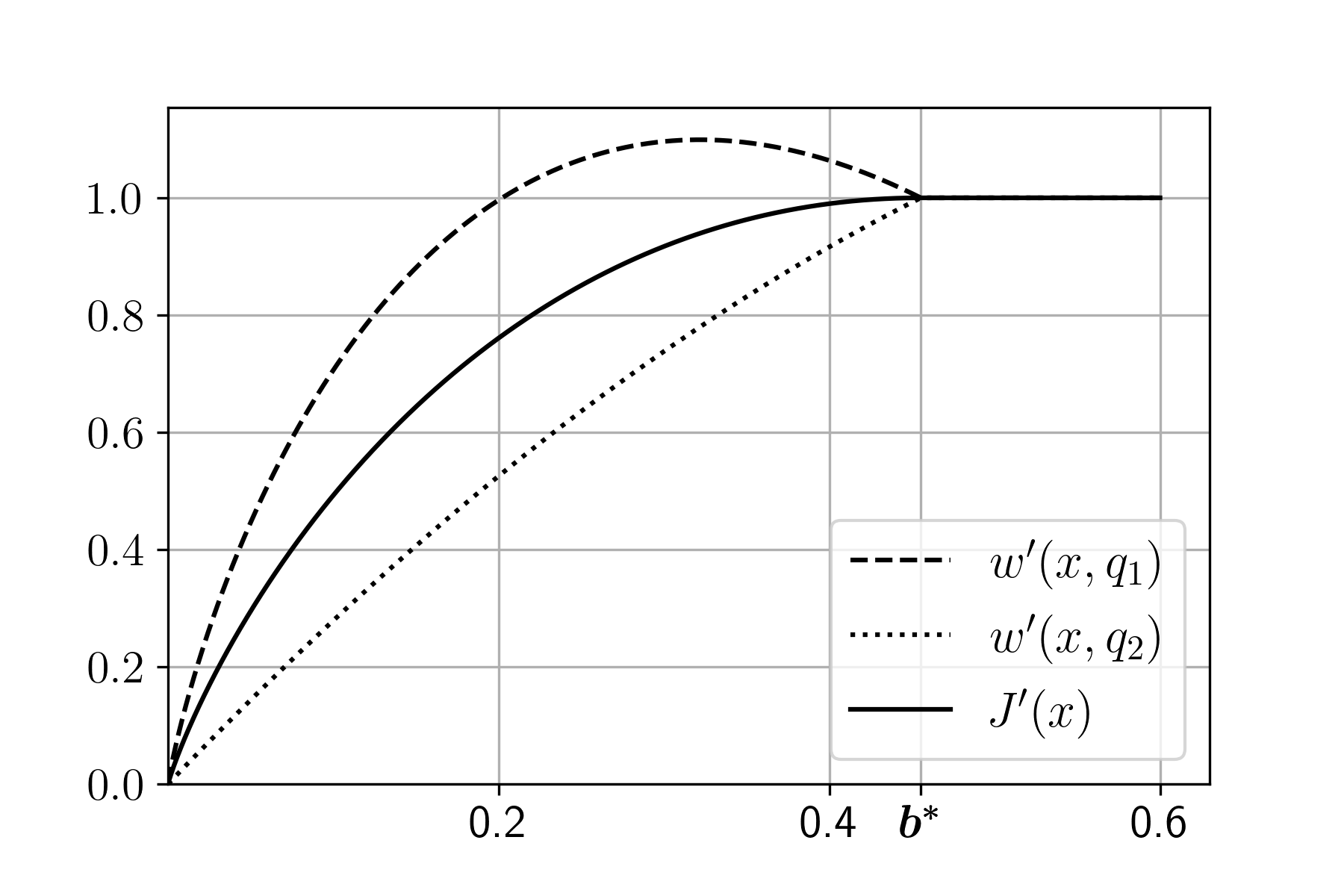}\label{fig:purederwork1}}

\vspace{-7.5mm}
\subfigure[]{\includegraphics[width=0.5\textwidth]{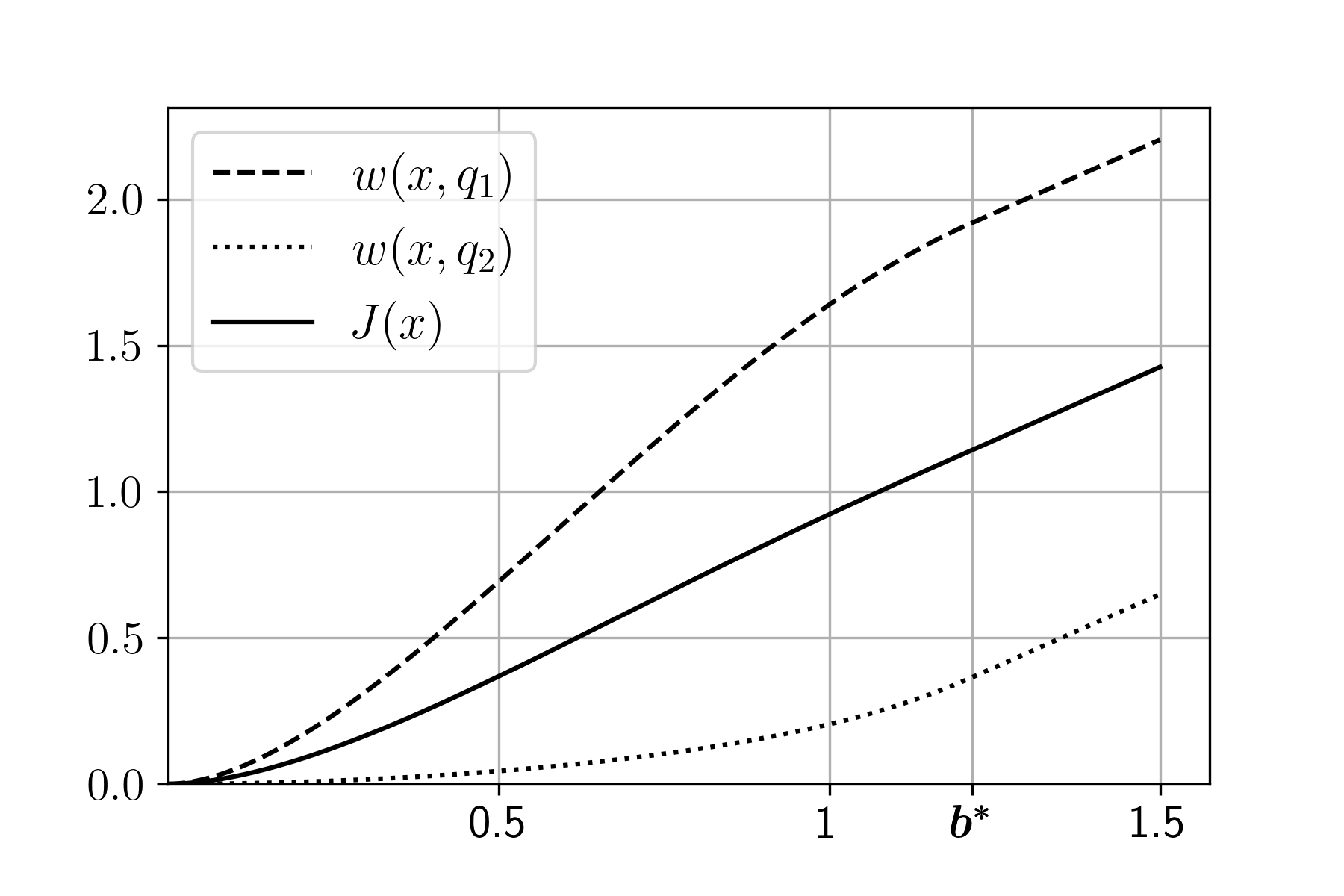}\label{fig:purenotwork1}}
\subfigure[]{\includegraphics[width=0.5\textwidth]{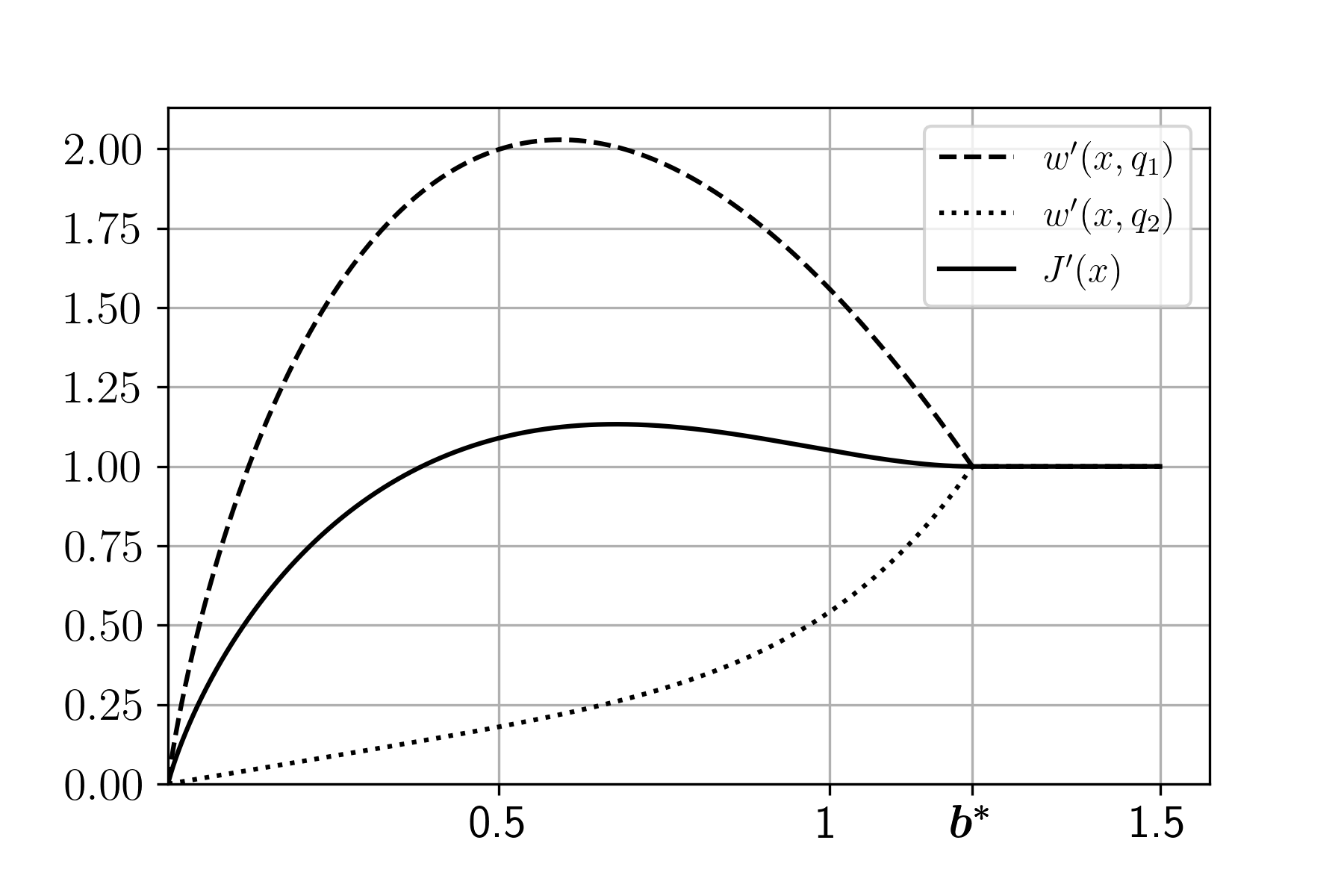}\label{fig:purdernotworke1}}
\vspace{-5mm}
\caption{Candidate strong threshold control equilibrium cost criteria $w$ and $J$. 
In panels~(a)-(b) the cost criterion $J$ and its derivative are illustrated for parameter values $\sigma^2=0.16$, $q_1=0.2$, $q_2=0.4$, and satisfy $J'\leq 1$ in the waiting region. 
In panels~(c)-(d) they are illustrated for parameter values $\sigma^2=0.16$, $q_1=0.2$, $q_2=3$, and we observe that this cannot lead to an equilibrium since $J'(x)>1$ for some $x$ in the waiting region.} 
\end{figure}

\subsubsection{Large discount rate difference: No strong threshold control equilibrium strategy} 
\label{sec:Equib2}

In this section we expand our analysis by showing that there is no equilibrium in the form of strong threshold control strategies when the difference $q_2-q_1$ in the possible discount rates is sufficiently large. This follows again by an application of our general Theorem \ref{main_thm}.

\begin{theorem}\label{thm:no-pure-NE} 
For each $q_1$, there exits a constant $h>0$ such that if 
$q_2-q_1\geq h$, then there is no equilibrium strategy $(0,[b,\infty))$ for the time-inconsistent singular control problem defined by \eqref{TISC} resulting from a strong threshold control strategy $(0,[b,\infty),0)$, $b \in (0,\infty)$. 
\end{theorem} 
\begin{proof}
We firstly notice that the limiting cost criterion $w(\cdot;q_i,0,[b,\infty))$ associated with a strong threshold control strategy $(0,[b,\infty),0)$, for an arbitrary $b \in (0,\infty)$, exists and is given by the solution $v(\cdot;q_i)$ of the BVP \eqref{xBVPaa}--\eqref{xBVPad} (cf.~Lemma \ref{useful-lemma}) with $b$ in place of $b^*$. 
This solution $v$ is constructed by \eqref{w_pure_def}--\eqref{Ai} with $b$ in place of $b^*$ and satisfies the conditions of Lemma \ref{useful-lemma-2}.

Now, the (verification) Theorem \ref{main_thm} implies that $(0,[b,\infty))$ is an equilibrium strategy if and only if the corresponding $V$ (constructed as in \eqref{J_pure_def}) satisfies conditions \eqref{mainthmcond1}--\eqref{mainthmcond3}, and consequently $V\in \mathcal{C}^2(0,\infty)$ thanks to Corollary \ref{coq_beta}. 
We know from Section \ref{sec:inform-strong} that there exists a unique $b$ that satisfies this regularity condition (cf.~\eqref{sfsdfdsas}) and is precisely given by $b^*$ in \eqref{xhat_def}. 
Hence, $(0,[b^*,\infty))$ is the only possible equilibrium in the form of strong threshold control strategies.  

However, we can conclude from the result in Proposition \ref{non_ex_equilibrum}.(ii) that there exits, for each $q_1$,  a constant $h>0$, such that if 
$q_2-q_1 \geq h$, then $V'(x)>1$ for $x\in (b^*-\epsilon, b^*)$ and a sufficiently small $\epsilon>0$. This violates condition \eqref{mainthmcond1}, hence $(0,[b^*,\infty))$ cannot be an equilibrium strategy, which completes the proof.
\end{proof}

Motivated by the result of Theorem \ref{thm:no-pure-NE}  that no strong threshold control strategy can yield an equilibrium when the difference $q_2-q_1$ is sufficiently large, we proceed in the following section with the search for an equilibrium strategy in the form of mild threshold control strategies.

\subsection{Case study with mild threshold control equilibrium strategy} 
\label{sec:mixed-case}
Recall from Definition \ref{def:mixed0} that a mild threshold control strategy results in a controlled process $X^D$ given by the unique strong solution to the SDE \eqref{Xubeta} (cf.~Section \ref{sec:mild}), which under the dynamics of \eqref{example_GBM} becomes
$$
X_t^{D} = 
x - \int_0^t u_\beta(X_s^{D}) \, ds + \int_0^t \sigma X_s^{D} \, dW_s - \mathbf{1}_{\{x \geq \beta\}}(x-\beta + \delta), 
\quad  0 \leq t \leq \tau^D_0.
$$
We consider the following specification with constants $\ub \in (0,\beta)$ and $\delta \in (0,\beta)$, and a c\`adl\`ag function 
$$
u_{\ub,\beta}(x) = u_{\ub,\beta}(x) {\bf 1}_{\{x \in [\ub,\beta)\}}, 
\quad \text{such that} \quad 
u_{\ub,\beta}(x) > 0, \; x \in [\ub,\beta)
\quad \text{and} \quad 
\lim_{x \uparrow \beta} u_{\ub,\beta}(x) = \infty,
$$
such that the state-space splits into:
\begin{itemize}
\vspace{-1mm}
\item Waiting region $\mathcal{W}=(0,\ub)$, where $X^D$ is not controlled, i.e.~$u_{\ub,\beta}(x)=0$, $x\in(0,\ub)$; 

\vspace{-1mm}
\item Mild action region $\mathcal{M}=[\ub,\beta)$, where $X^D$ is controlled by an absolutely continuous control with a rate $u_{\ub,\beta}(x)>0$, $x\in [\ub, \beta)$, such that $\beta$ is an entrance-not-exit boundary for~$X^D$; 

\vspace{-1mm}
\item Strong action region $\mathcal{S}=[\beta,\infty)$, from where $X^D$ jumps instantly downwards to $\beta - \delta$. 
\end{itemize}
\vspace{-1mm}
Our ansatz in this section is to search for an equilibrium strategy (Definition \ref{def:equ_stop_time}) of the type $(u_{\ub,\beta}, [\beta,r))$ -- for which we recall that $\delta$ is sent to zero in the sense of the limiting cost criterion \eqref{eq:limiting-cost-cr}.

In the forthcoming analysis, we firstly construct our candidate mild threshold control equilibrium strategy in Section \ref{sec:inform-mild}. 
Subsequently, we prove in Section \ref{sec:Equiu}, by relying on our general verification theorem, that the constructed candidate is indeed an equilibrium strategy in the case study when the difference between the two possible discount rates $q_2$ and $q_1$ is sufficiently large; see Theorem \ref{ex_mix_theorem}.

\subsubsection{Construction of a candidate equilibrium via a mild threshold control strategy}
\label{sec:inform-mild}

In light of Theorem \ref{main_thm}, we derive our candidate equilibrium strategy $(u_{\ub^*,\beta^*},[\beta^*,\infty))$ with associated values $v(\cdot;q_i)$, $i=1,2$, so that it solves, cf.~Lemma \ref{useful-lemma}, the BVP
\begin{align}
&\tfrac12 x^2 + \tfrac12 \sigma^2 x^2 v''(x;q_i) - q_i v(x;q_i) = 0, \qquad \qquad \qquad \qquad \qquad \qquad \qquad \;\, x \in (0,\ub^*), 
\label{XBVPaa}\\
&\tfrac12 x^2 + \tfrac12 \sigma^2 x^2 v''(x;q_i) - u_{\ub^*,\beta^*}(x) v'(x;q_i) - q_i v(x;q_i) + u_{\ub^*,\beta^*}(x) = 0, \quad x \in (\ub^*,\beta^*), 
\label{XBVPaaa}\\
&v(x;q_i) = x - \beta^* + v(\beta^*;q_i), \qquad \qquad \qquad \qquad \qquad \qquad \qquad \qquad \qquad \; x \in [\beta^*,\infty), 
\label{XBVPab}\\
&v(0;q_i) = 0, 
\label{XBVPac}\\
& \text{$v(\cdot;q_i)\in \mathcal{C}^2\big( (0,\infty) \setminus \{\ub^*, \beta^*\} \big) \cap \mathcal{C}^1(0,\infty)$ with $|v''(x\pm;q_i)|<\infty$ for $x\in\{\ub^*,\beta^*\}$}.
\label{XBVPad}
\end{align} 
and by imposing that (cf.~\eqref{J=J}) 
\begin{equation} \label{J_mild_def}
V(x) := \tfrac12 \, v(x; q_1) + \tfrac12 \, v(x; q_2), \quad x \in [0,\infty),
\end{equation} 
satisfies condition \eqref{mainthmcond2} of Theorem \ref{main_thm}, namely 
\begin{align} \label{eq_ex_1} 
V'(x) = \tfrac{1}{2} v'(x;q_1) + \tfrac{1}{2} v'(x;q_2) &= 1, \quad x \in (\ub^*, \beta^*), 
\end{align}
where we used the fact that $(\ub^*,\beta^*) \subset [\ub^*,\beta^*) = \overline{\mathcal{M}}_{u} \setminus \{\beta^*\}$.
To that end, we obtain the following result. 

\begin{proposition} \label{thm:mildvalue} 
Define the constants $A_i,a_i,b_i,c_i$, for $i=1,2$, $\ub^*$ and $\beta^*$ by 
\begin{align}
A_i &:= \Big(a_i \ub^* -\frac{\ub^*}{q_i-\sigma^2} + b_i \Big) \gamma_i^{-1} (\ub^*)^{1-\gamma_i}, \quad a_i := \frac{2 (-1)^i}{q_2-q_1}, \quad 
b_i := \frac{2 q_{3-i} (-1)^{3-i}}{q_2-q_1}, \label{Aiaibi}\\
c_i &:= \frac{\tfrac12 (\ub^*)^2 \big(1 + (\sigma^2 - q_i) a_i -(-1)^i \sigma^2 (\gamma_1 - 2) \big(a_1 -\frac{1}{q_1-\sigma^2}\big) \big) - \ub^* \big((-1)^i \tfrac12 \sigma^2 (\gamma_1 - 1) b_1 + q_i b_i\big)}{q_i}, \label{ci}\\
\ub^* &:=\frac{(\gamma_1-1)b_1+ (\gamma_2-1)b_2}{(2-\gamma_1) \big(a_1-\frac{1}{q_1-\sigma^2}\big) + (2-\gamma_2) \big(a_2-\frac{1}{q_2-\sigma^2}\big)} 
\quad \text{and} \quad 
\beta^* := \tfrac{1}{2} (q_1 + q_2).\label{ubbeta}
\end{align}
and the function
\begin{equation} \label{u*}
u_{\ub^*,\beta^*}(x) 
:= \frac{\tfrac12 \big(1 + (\sigma^2 - q_1) a_1 \big) x^2 - q_1 b_1 x - q_1 c_1}{a_1 x + b_1 - 1} {\bf 1}_{\{x \in [\ub^*,\beta^*)\}}, \quad x \in (0,\beta^*).
\end{equation}
If $0 < \ub^* < \beta^*$ holds true in \eqref{ubbeta}, then $u_{\ub^*,\beta^*}(x)$ given by \eqref{u*} together with
\begin{align}\label{candv}
v(x;q_i)= \begin{cases}
A_i x^{\gamma_i}+\frac{1}{2(q_i-\sigma^2)} x^2, & x\in (0,\ub^*), \\
\frac{1}{2} a_i x^2 + b_i x + c_i, 
& x \in [\ub^*,\beta^*), \\
x - \beta^* + v(\beta^*;q_i), & x \in [\beta^*,\infty),
\end{cases}
\end{align} 
solve the BVP \eqref{XBVPaa}--\eqref{XBVPad} (where $\gamma_i$ is defined by \eqref{w_pure_def} for $i=1,2$). 
In this case, also \eqref{J_mild_def}--\eqref{eq_ex_1} hold true. 
\end{proposition}
\begin{proof} 
We prove this proposition as follows:
Assuming that a solution to the BVP \eqref{XBVPaa}--\eqref{XBVPad} satisfying \eqref{J_mild_def}--\eqref{eq_ex_1} exists, we show that it has to be of the form specified by \eqref{Aiaibi}--\eqref{candv}; see Steps 1--6. 
In Step 7, we conclude the proof by verifying that if  $0 < \ub^* < \beta^*$ holds true, then our proposed solution given by \eqref{Aiaibi}--\eqref{candv} does indeed satisfy the BVP \eqref{XBVPaa}--\eqref{XBVPad} as well as \eqref{J_mild_def}--\eqref{eq_ex_1}.

\vspace{1mm}
{\it Step 1. Case $x \in (\ub^*,\beta^*)$}.
Recall the function $V$ defined by \eqref{J_mild_def} and satisfying \eqref{eq_ex_1}, where the latter further implies that $V''(x)=0$. 
Using these in the sum of the ODEs in \eqref{XBVPaaa} for $i=1,2$, we obtain
\begin{align} \label{eq_ex_2b}
&x^2 + \sigma^2 x^2 V''(x) - 2 u_{\ub^*,\beta^*}(x) (V'(x) - 1) - q_1 v(x;q_1) - q_2 v(x;q_2) 
= x^2 - q_1 v(x;q_1) - q_2 v(x;q_2) = 0, 
\end{align}
which in turn implies that 
\begin{align} \label{eq_ex_2}
q_1 v'(x;q_1) + q_2 v'(x;q_2) = 2x , \quad x \in (\ub^*, \beta^*). 
\end{align}
Viewing \eqref{eq_ex_1} and \eqref{eq_ex_2} as an equation system for $v'(x;q_1)$ and $v'(x;q_2)$, yields that 
\begin{align} \label{eq_ex_3}
v'(x;q_i) = a_i \, x + b_i, \quad x \in (\ub^*, \beta^*), \quad i=1,2, 
\end{align}
where $a_i$ and $b_i$ are defined by \eqref{Aiaibi}. 
It thus follows that 
\begin{align}\label{hq2hkl12r}
v(x;q_i) = \tfrac{1}{2} a_i x^2 + b_i x + c_i, 
\quad x \in [\ub^*,\beta^*), \quad i=1,2, 
\end{align}
for some constants $c_i$, $i=1,2$, to be determined (later on). 

\vspace{1mm}
{\it Step 2. Case $x \in (0,\ub^*)$}.
Now, the combination of the general solution to the ODE in \eqref{XBVPaa} with the condition \eqref{XBVPac} implies that 
\begin{align}\label{w_mild_def}
v(x;q_i)=A_i x^{\gamma_i}+\frac{1}{2(q_i-\sigma^2)} x^2, \quad x\in (0,\ub^*), 
\quad \text{where } \gamma_i \text{ is defined by \eqref{w_pure_def} for $i=1,2$},
\end{align}
and the constants $A_i$, $i=1,2$, are given by \eqref{Aiaibi}. 
To see the latter, we use the continuous differentiability of $v$ from \eqref{XBVPad} and \eqref{eq_ex_3}, which yield that $v'(\ub^*-;q_i) = v'(\ub^*+;q_i) = a_i \, \ub^* + b_i$, and combine it with the derivative of \eqref{w_mild_def}. 

\vspace{1mm}
{\it Step 3. Deriving $\beta^*$.}  
Using the continuous differentiability of $v$ from \eqref{XBVPad}, but now with \eqref{XBVPab}, yields that $v'(\beta^*-;q_i) = v'(\beta^*+;q_i) = 1$, hence \eqref{eq_ex_2} implies the expression of $\beta^*$ in \eqref{ubbeta}.

\vspace{1mm}
{\it Step 4. The control rate $u_{\ub^*,\beta^*}(x)$ for all $x \in [\ub^*,\beta^*)$}.
It follows from a rearrangement of the ODE \eqref{XBVPaaa} for $i=1,2$, that 
\begin{equation} \label{u*-prel}
u_{\ub^*,\beta^*}(x) 
= \frac{\tfrac12 x^2 + \tfrac12 \sigma^2 x^2 v''(x;q_1) - q_1 v(x;q_1)}{v'(x;q_1) - 1} 
= \frac{\tfrac12 x^2 + \tfrac12 \sigma^2 x^2 v''(x;q_2) - q_2 v(x;q_2)}{v'(x;q_2) - 1} 
, \quad x \in (\ub^*,\beta^*), 
\end{equation}
where the second equality follows from \eqref{Aiaibi} and the expression \eqref{hq2hkl12r} of $v(\cdot;q_i)$ on $(\ub^*,\beta^*)$ for $i=1,2$, which imply that 
\begin{align*}
&v'(x;q_1) - 1 
= a_1 \, x + b_1 - 1 
= 1 - a_2 \, x - b_2
= 1 - v'(x;q_2), \\
&\tfrac12 x^2 - q_1 v(x;q_1) 
= \tfrac12 (1 - q_1 a_1) x^2 - q_1 (b_1 x + c_1) 
= \tfrac12 (q_2 a_2 - 1) x^2 + q_2 (b_2 x + c_2) 
= q_2 v(x;q_2) - \tfrac12 x^2, \\
&v''(x;q_1) = a_1 = - a_2 = - v''(x;q_2).
\end{align*}
By substituting in \eqref{u*-prel} the expression \eqref{hq2hkl12r} of $v$, we obtain \eqref{u*}.

\vspace{1mm}
{\it Step 5. The constant $\ub^*$}.
By evaluating the ODEs \eqref{XBVPaa}--\eqref{XBVPaaa} at $\ub^*-$ and $\ub^*+$ respectively, and subtracting them for each $i=1,2$, we obtain thanks to the regularity of $v$ from \eqref{XBVPad} that
\begin{align} \label{eq:premub}
\tfrac12 \sigma^2 (\ub^*)^2 \big(v''(\ub^*+;q_i) - v''(\ub^*-;q_i) \big) + u_{\ub^*,\beta^*}(\ub^*) \big(1 - v'(\ub^*;q_i)\big) = 0 
, \quad i=1,2,
\end{align}
which thus implies that 
\begin{align}
\begin{split}
u_{\ub^*,\beta^*}(\ub^*) 
&= \frac{\tfrac12 \sigma^2 (\ub^*)^2 \big(v''(\ub^*+;q_1) - v''(\ub^*-;q_1) \big)}{v'(\ub^*;q_1) - 1} 
= \frac{\tfrac12 \sigma^2 (\ub^*)^2 \big(v''(\ub^*+;q_2) - v''(\ub^*-;q_2) \big)}{v'(\ub^*;q_2) - 1} \\
&\Leftrightarrow \quad 
v''(\ub^*+;q_1) - v''(\ub^*-;q_1) = v''(\ub^*-;q_2) - v''(\ub^*+;q_2) \\
&\Leftrightarrow \quad 
-(\gamma_1 - 2) \Big(a_1 -\frac{1}{q_1-\sigma^2}\Big) - (\gamma_1 - 1) \frac{b_1}{\ub^*} = (\gamma_2 - 2) \Big(a_2 -\frac{1}{q_2-\sigma^2}\Big) + (\gamma_2 - 1) \frac{b_2}{\ub^*} 
\end{split}
\label{loweq_x_eq}
\end{align}
where the first equivalence follows from \eqref{eq_ex_1} which implies that $v'(\ub^*;q_1) - 1 = 1 - v'(\ub^*;q_2)$, and the second by the expressions \eqref{Aiaibi}, \eqref{hq2hkl12r} and \eqref{w_mild_def}. Simple calculations then yield the expression of $\ub^*$ in \eqref{ubbeta}.

\vspace{1mm}
{\it Step 6. The constants $c_i$, $i=1,2$}. 
Using the ODEs in \eqref{XBVPaaa} for $i=1,2$ at $x=\ub^*+$, we obtain 
\begin{align*}
\tfrac12 (\ub^*)^2 + \tfrac12 \sigma^2 (\ub^*)^2 v''(\ub^*+;q_i) - u_{\ub^*,\beta^*}(\ub^*) (v'(\ub^*+;q_i) - 1) - q_i v(\ub^*+;q_i) = 0, \quad i=1,2,  
\end{align*}
which gives thanks to the expressions \eqref{hq2hkl12r} of $v(\cdot;q_i)$, for each $i=1,2$, that 
\begin{align*}
c_i 
&= \frac{\tfrac12 (\ub^*)^2 \big(1 + (\sigma^2 - q_i) a_i \big) - u_{\ub^*,\beta^*}(\ub^*) (a_i \ub^* + b_i - 1) - q_i b_i \ub^*}{q_i} 
, \quad i=1,2,
\end{align*}
where by substituting the expression of $u_{\ub^*,\beta^*}(\ub^*)$ from \eqref{loweq_x_eq}, we obtain \eqref{ci}, using also \eqref{hq2hkl12r}. 

\vspace{1mm}
{\it Step 7.} We are now ready to show that $v(\cdot;q_i)$ in 
\eqref{candv} and $u_{\ub^*,\beta^*}(\cdot)$ in \eqref{u*} satisfy the BVP \eqref{XBVPaa}--\eqref{XBVPad} as well as \eqref{J_mild_def}--\eqref{eq_ex_1}, when $0 < \ub^* < \beta^*$. 
Using the analysis in Steps 1--6, we observe that what remains to be proved is that $v(\cdot;q_i)$ satisfies the smoothness condition \eqref{XBVPad}.

It follows by construction (again, see Steps 1--6) that $v(\cdot;q_i)\in \mathcal{C}^2\big( (0,\infty) \setminus \{\ub^*, \beta^*\} \big) \cap \mathcal{C}^1\big( (0,\infty) \setminus \{\ub^*\} \big)$ and $v'(\ub^*-;q_i) = v'(\ub^*+;q_i)$, hence we only need to verify continuity at $\ub^*$. We have that 
\begin{align*}
q_i v(\ub^*+;q_i) 
&= \tfrac1{2} (\ub^*)^2 + \tfrac1{2} \sigma^2 (\ub^*)^2 v''(\ub^*+;q_i) + u_{\ub^*,\beta^*}(\ub^*) \big(1 - v'(\ub^*;q_i)\big) \\
&= \tfrac1{2} (\ub^*)^2 + \tfrac1{2} \sigma^2 (\ub^*)^2 v''(\ub^*-;q_i)
= q_i v(\ub^*-;q_i) , \quad i=1,2,
\end{align*}
where the equalities follow from \eqref{XBVPaaa}, \eqref{loweq_x_eq} and \eqref{XBVPaa}, respectively. This completes the proof.
\end{proof}

\begin{figure}[tb] \label{fig2}
\subfigure[]{\includegraphics[width=0.5\textwidth]{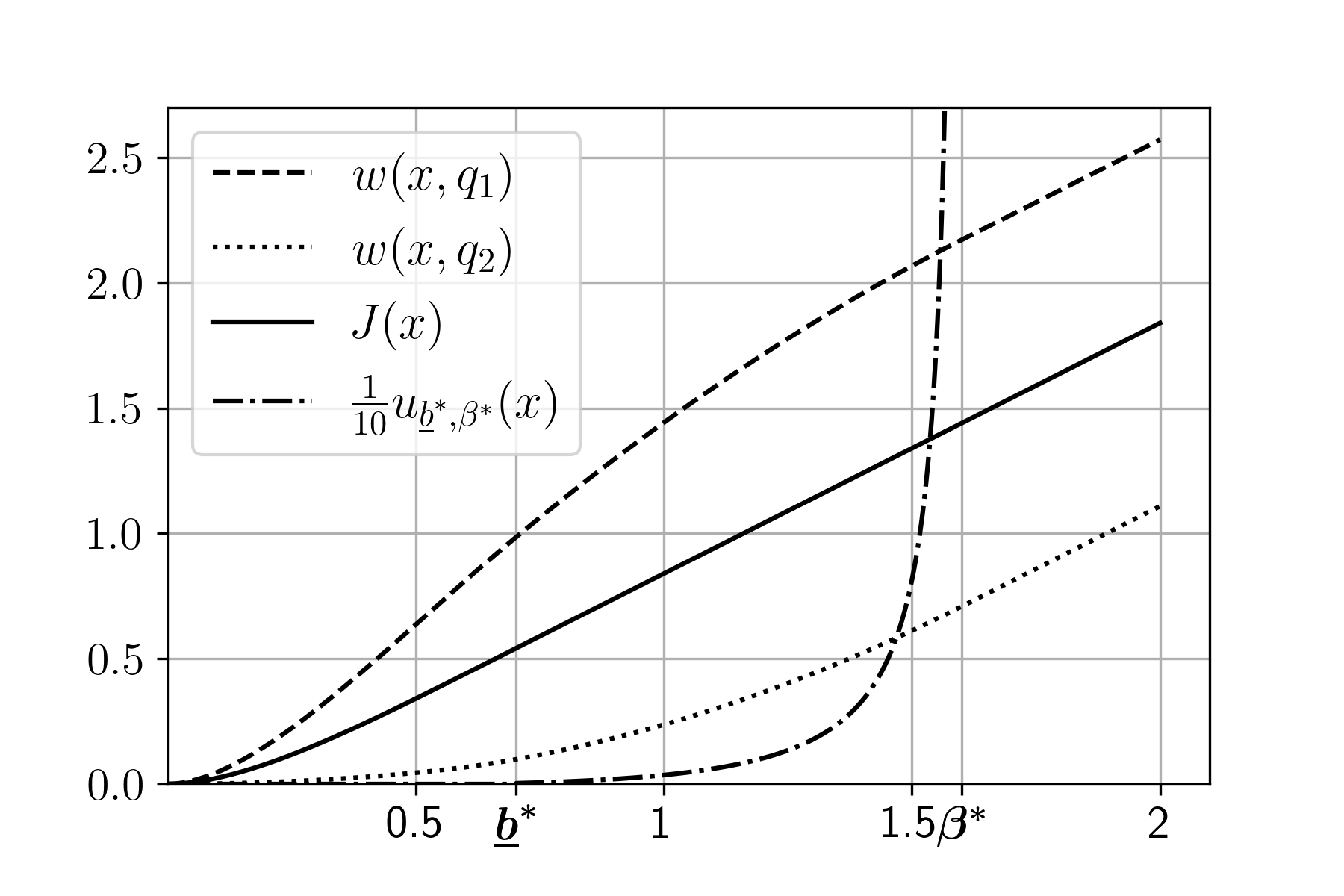}\label{fig:mixedwithu1}}
\subfigure[]{\includegraphics[width=0.5\textwidth]{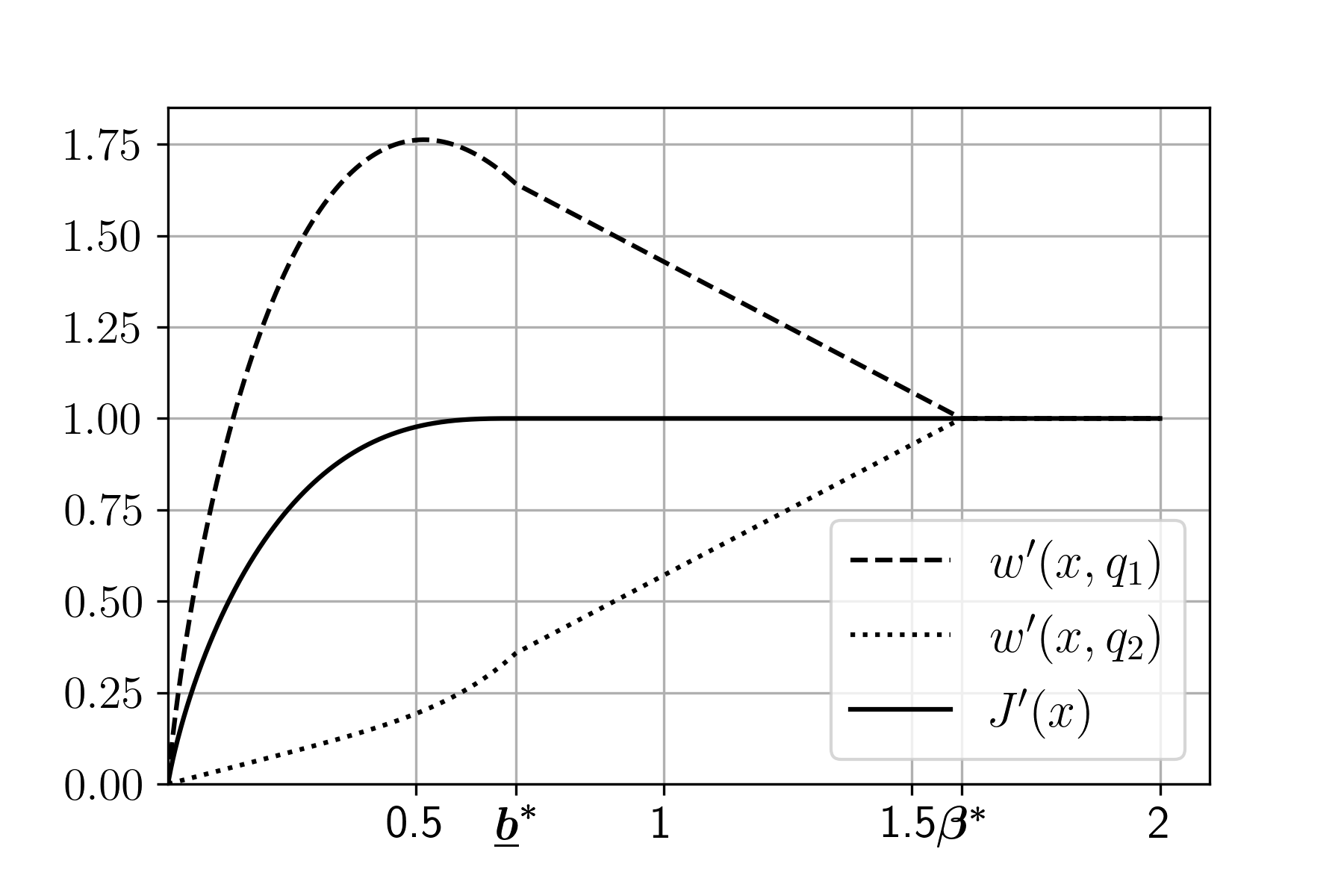}\label{fig:mixedder1}}
\vspace{-7.5mm}
\begin{center}
\subfigure[]{\includegraphics[width=0.5\textwidth]{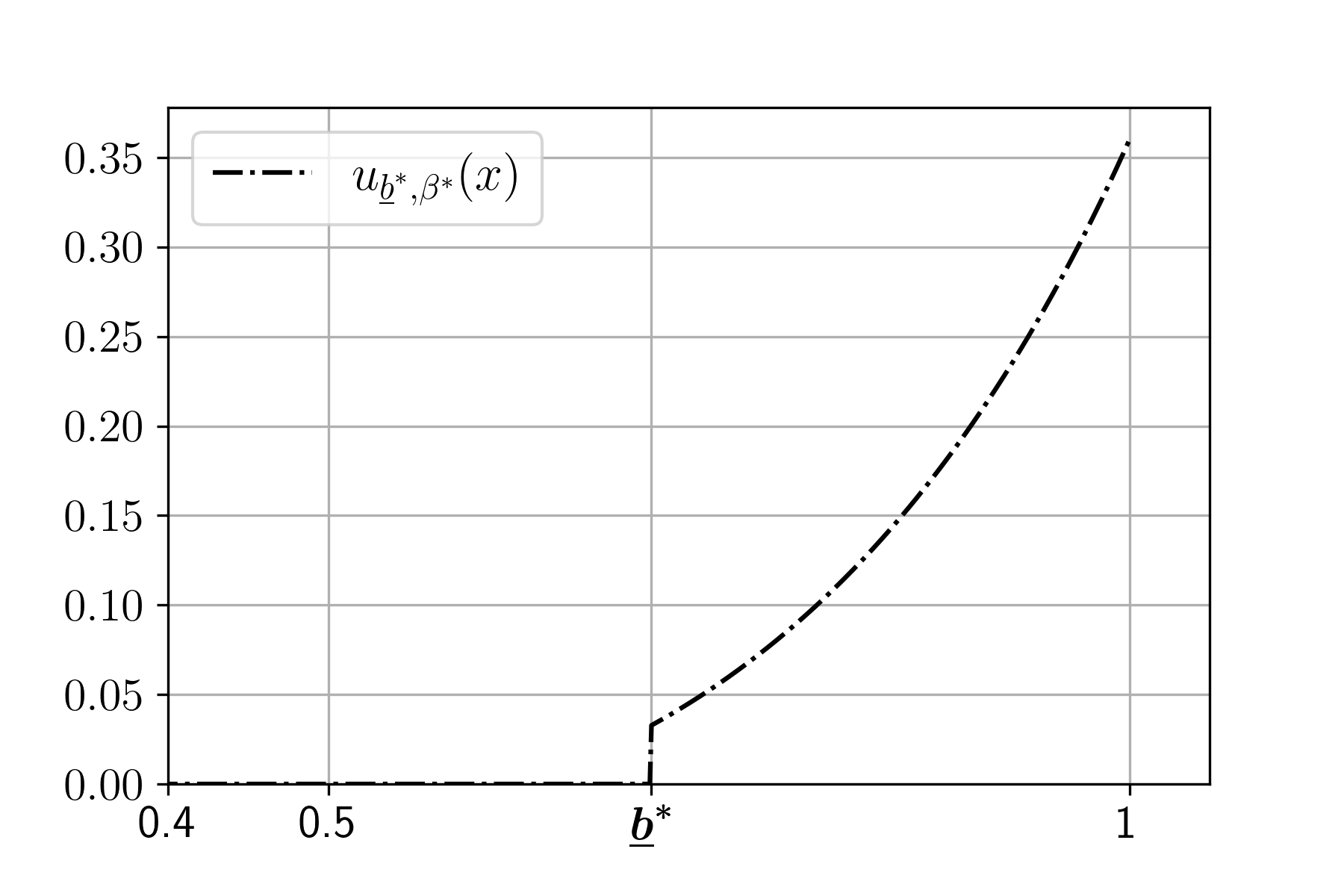}\label{fig:udetail}}
\end{center}
\vspace{-5mm}
\caption{Candidate mild threshold control equilibrium cost criteria $w$ and $J$. 
In panels~(a)-(b) the cost criterion $J$ and its derivative are illustrated for parameter values $\sigma^2=0.16$, $q_1=0.2$, $q_2=3$ (same as in Fig.~\ref{fig:purenotwork1}--\ref{fig:purdernotworke1}), and satisfy $J'\leq 1$ in the waiting region. 
Panel (a) also highlights the explosion point $\beta^*$ of $u_{\ub^*,\beta^*}$, while panel (c) provides a zoomed-in view of the discontinuity of $u_{\ub^*,\beta^*}$ at $\ub^*$.}
\end{figure}

\subsubsection{Mild threshold control equilibrium strategy}
\label{sec:Equiu}

We first prove in Proposition \ref{2141gsd} that our candidate equilibrium strategy $(u_{\ub^*,\beta^*},[\beta^*,\infty))$ results from a mild threshold control strategy $(u_{\ub^*,\beta^*},[\beta^*,\infty), \delta)$ (cf.~Proposition \ref{thm:mildvalue}) and that both limiting cost criteria \eqref{wrqD} and \eqref{eq:limiting-cost-cr} exist, under the assumption that the difference between the two discount rates $q_2$ and $q_1$ is sufficiently large.
We then prove in Theorem \ref{ex_mix_theorem} that this candidate is indeed an equilibrium strategy under the same assumption; see Figure \ref{fig:mixedwithu1}--\ref{fig:udetail} for an illustration.

\begin{proposition}\label{2141gsd} For each $q_1$, there exists a constant $h>0$ such that if $q_2-q_1\geq h$, then:
\begin{enumerate}[\rm (i)]
\vspace{-2mm}
\item $(u_{\ub^*,\beta^*},[\beta^*,\infty), \delta)$  corresponding to \eqref{ubbeta}--\eqref{u*} satisfies $0 < \ub^* < \beta^*$ and is a  mild threshold control strategy (cf.~Definition \ref{def:mixedm}), for each $\delta \in (0, \beta^*)$; 
    
\vspace{-2mm}
\item The associated limiting cost criterion $w(x;q_i,u_{\ub^*,\beta^*},[\beta^*,\infty))$ (cf.~\eqref{wrqD}) exists and is given by $v(x;q_i)$ in \eqref{candv}, for all $x \in [0,\infty)$ and $i=1,2$;
    
\vspace{-2mm}
\item The associated limiting cost criterion $J(\cdot;u_{\ub^*,\beta^*},[\beta^*,\infty))$ (cf.~\eqref{eq:limiting-cost-cr}) exists and is given by $V(\cdot)$ in \eqref{J_mild_def}, which additionally satisfies \eqref{eq_ex_1} and the corresponding BVP \eqref{JBVPa}--\eqref{JBVPc} for this case study, as well as $V(\cdot) \in \mathcal{C}^2(0,\infty)$.
\end{enumerate}
\end{proposition}

\begin{proof}
We prove the parts separately.

\vspace{1mm}
{\it Proof of (i).} 
This is presented in the following five steps, where we fix an arbitrary $q_1$ and we write $\ub^* \equiv \ub^*(q_2;q_1)$ and $\beta^* \equiv \beta^*(q_2;q_1)$ to stress, whenever needed, their dependence on $q_2$ and $q_1$ (cf.~\eqref{ubbeta}).

\vspace{1mm}
{\it Step 1. Proof that $\ub^*(q_2;q_1) <\infty$ for all $q_2 \in (q_1, \infty]$}.
To see this, notice first that it follows by definition \eqref{ubbeta} that $\ub^*(q_2;q_1) <\infty$ for all $q_2 \in (q_1, \infty)$. 
We thus only need to compute the limit  
$$
\lim_{q_2 \to \infty} \ub^*(q_2;q_1) = \frac{2(\gamma_1-1) (q_1-\sigma^2)}{(\gamma_1-2)} \in (0,\infty), 
$$
where we recall that $\gamma_i = \gamma(q_i)$ in \eqref{gammaq} and the dependence on $q_2$ of the constants $a_i$ and $b_i$ defined by \eqref{Aiaibi}, for $i=1,2$.

\vspace{1mm}
{\it Step 2. Proof that there exists $h>0$ such that $0 < \ub^*(q_2;q_1) <\beta^* (q_2;q_1)$ for all $q_2-q_1\geq h$.}
This is a direct consequence of the continuity (cf.~\eqref{ubbeta}) of the maps $q_2 \mapsto \ub^*(q_2;q_1)$ and $q_2 \mapsto \beta^* (q_2;q_1) - \ub^*(q_2;q_1)$, and Step 1, which imply that $\lim_{q_2 \to \infty} \ub^*(q_2;q_1) > 0$ and  $\lim_{q_2 \to \infty} \beta^* (q_2;q_1) - \ub^*(q_2;q_1) = \infty$.

\vspace{1mm}
{\it Step 3. Proof that there exists $h>0$ such that $u_{\ub^*,\beta^*}(x)$ is positive and (strictly) increasing for all $x\in [\ub^*,\beta^*)$ and $q_2-q_1\geq h$.}
We firstly note from the expression \eqref{ci} of $c_1$ and \eqref{u*} of $u_{\ub^*,\beta^*}(\cdot)$ that 
\begin{equation} \label{ub*}
u_{\ub^*,\beta^*}(\ub^*) 
=\frac{\tfrac12 \sigma^2 \ub^* g_1(\ub^*)}{a_1 \ub^* + b_1 - 1}, 
\quad \text{where} \quad 
g_i(x) := (\gamma_i - 2) \Big(\frac{1}{q_i - \sigma^2} - a_i \Big) x - (\gamma_i - 1) b_i 
, \quad i=1,2.
\end{equation}
Assuming that $q_2$ is large enough as in Step 2, such that $0 < \ub^*(q_2;q_1) <\beta^* (q_2;q_1)$,  we can easily see from \eqref{Aiaibi} and \eqref{ubbeta} that 
$(q_2 - q_1) (a_1 \ub^* + b_1 - 1)  
= 2 (\beta^*(q_2;q_1) - \ub^*(q_2;q_1)) > 0$.
Hence, to show that $u_{\ub^*,\beta^*}(\ub^*)>0$, it remains to show that 
$g_1(\ub^*) > 0$. 

By rewriting the expression \eqref{ubbeta} of $\ub^*$ in view of the definition of $g_i(\cdot)$ in \eqref{ub*} for $i=1,2$, we observe that 
$g_1(\ub^*) = - g_2(\ub^*)$. Then, recalling that $\gamma_i = \gamma(q_i)$ in \eqref{gammaq}, the dependence on $q_2$ of the constants $a_i$, $b_i$ and $\beta^*$, defined by \eqref{Aiaibi} and \eqref{ubbeta}, we observe that 
\begin{align*}
g_1(\beta^*) 
&= \big(\gamma(q_1) - 2 \big) \frac{(q_2 + q_1 - 2 \sigma^2) (q_2 + q_1)}{2 (q_2-q_1) (q_1 - \sigma^2)} - \big(\gamma(q_1) - 1 \big) \frac{2 q_2}{q_2-q_1} \\
-g_2(\beta^*) 
&= \big(\gamma(q_2) - 2 \big) \frac{(q_2+q_1-2\sigma^2) (q_2 + q_1)}{2 (q_2-q_1)(q_2 - \sigma^2)} - \big(\gamma(q_2) - 1 \big) \frac{2 q_1}{q_2-q_1},
\end{align*}
and for $q_2$ large enough we have that $g_1(\beta^*) > - g_2(\beta^*) > 0$. 
Finally, we observe that for $q_2$ large enough, we also have that  
\begin{align*} 
x_1 :=\frac{2q_2(q_1-\sigma^2)(\gamma_1-1)}{(q_1+q_2-2\sigma^2)(\gamma_1-2)} 
> x_2 :=\frac{2q_1(q_2-\sigma^2)(\gamma_2-1)}{(q_1+q_2-2\sigma^2)(\gamma_2-2)} 
\quad \text{such that} \quad 
g_1(x_1) = 0 = g_2(x_2),
\end{align*}
thanks to Lemma \ref{nec_inequality}, the continuity of the maps $q_2 \mapsto x_i \equiv x_i(q_2)$, for $i=1,2$, and  
\begin{align*}
\lim_{q_2 \to \infty} \big(x_1(q_2) - x_2(q_2)\big)
= \frac{2 \big(q_1 - \sigma^2(\gamma_1 - 1)\big)}{\gamma_1 - 2} > 0. 
\end{align*}
Therefore, we conclude from the above that there exists $h>0$ such that $g_1(\beta^*) > - g_2(\beta^*) > 0$, $- g_2(x_1) > g_1(x_1) = 0$, and due to the linearity of $g_i(\cdot)$ that $g_1(\ub^*) = - g_2(\ub^*) > 0$, for all $q_2-q_1\geq h$, as required.

In order to prove that $u_{\ub^*,\beta^*}(\cdot)$ is increasing on $(\ub^*,\beta^*)$, we first observe that its denominator $x \mapsto a_1 x + b_1 - 1$ from \eqref{u*} is decreasing and positive on $(\ub^*,\beta^*)$ 
(to see this, recall that $a_1<0$ and that this function takes the value zero for $x=\beta ^*$). 
Hence, it suffices to show that the numerator of $u_{\ub^*,\beta^*}(\cdot)$ is positive and increasing on $(\ub^*,\beta^*)$.  Above we showed that $u_{\ub^*,\beta^*}(\ub^*)>0$, and hence it suffices to show that the derivative of the numerator of \eqref{u*}  is  positive on  $(\ub^*,\beta^*)$. Using \eqref{Aiaibi}, we find that this is equivalent to 
\begin{align*} 
\big(1 + (\sigma^2 - q_1) a_1 \big) x - q_1 b_1 > 0 
&\quad \Leftrightarrow \quad 
\frac{(q_2 + q_1 - 2 \sigma^2) x - 2 q_2 q_1}{q_2 - q_1} > 0 , 
\quad x \in (\ub^*,\beta^*).
\end{align*}
Thanks to Assumption \ref{assum:appli} which guarantees that $q_2 + q_1 - 2 \sigma^2>0$, it thus suffices to show that there exists $h>0$ such that $q_2 \mapsto \big((q_2 + q_1 - 2 \sigma^2)\ub^*(q_2;q_1) - 2 q_2 q_1 \big)/(q_2 - q_1) > 0$ for all $q_2-q_1\geq h$; which follows from the mapping's continuity and the limit in Step 1 which gives
$$
\lim_{q_2 \to \infty} \frac{(q_2 + q_1 - 2 \sigma^2)\ub^*(q_2;q_1) - 2 q_2 q_1}{q_2 - q_1} = \frac{2 \big(q_1 - \sigma^2(\gamma_1 - 1)\big)}{\gamma_1 - 2} > 0,
$$
where the latter inequality follows from $\gamma_1 = \gamma(q_1)$ in \eqref{gammaq} and the result of Lemma \ref{nec_inequality}.

\vspace{1mm}
{\it Step 4. Proof that $c_1(q_2;q_1) <\infty$ for all $q_2 \in (q_1, \infty]$}.
To see this, notice first that it follows by definition \eqref{ci} that $c_1(q_2;q_1) <\infty$ for all $q_2 \in (q_1, \infty)$. 
We thus only need to compute the limit (by substituting $a_1$, $b_1$ from \eqref{Aiaibi}) 
\begin{align*}
\lim_{q_2 \to \infty} q_1 c_1(q_2;q_1) 
&= \lim_{q_2 \to \infty} \bigg\{ 
\tfrac12 \big(\ub^*(q_2;q_1)\big)^2 \Big(1 - \frac{2(\sigma^2 - q_1)}{q_2-q_1} - \sigma^2 (\gamma_1 - 2) \frac{q_2 + q_1 - 2\sigma^2}{(q_2-q_1)(q_1-\sigma^2)} \Big) \\
&\qquad \qquad \;+ \ub^*(q_2;q_1) \big( \tfrac12 \sigma^2 (\gamma_1 - 1) - q_1 \big) \frac{2 q_2}{q_2-q_1} 
\bigg\}\\
&= \tfrac12 \lim_{q_2 \to \infty} \big(\ub^*(q_2;q_1)\big)^2 - 2 q_1 \lim_{q_2 \to \infty} \ub^*(q_2;q_1) 
<\infty,     
\end{align*}
where we use the limit from Step 1 in the second equality.

\vspace{1mm}
{\it Step 5. Proof that there exists $h>0$ such that $s_{\beta^*}(\cdot)$ defined by \eqref{scale_def} with $u_{\beta^*}(x) = u_{\ub^*,\beta^*}(x)$, $\mu(x)=0$ and $\sigma(x) = \sigma x$ for all $x \in (0,\beta^*)$, satisfies the condition \eqref{non-exit-entrance-cond} for all $q_2-q_1\geq h$.}
It follows from \eqref{u*} and \eqref{Aiaibi} that
\begin{equation} \label{bound-u-prel}
\frac{u_{\ub^*,\beta^*}(x)}{\sigma^2 x^2} 
= \frac{\tfrac12 \big(1 + (\sigma^2 - q_1) a_1 \big) x^2 - q_1 b_1 x - q_1 c_1}{\sigma^2 (-a_1) x^2 (\beta^* - x)} {\bf 1}_{\{x \in [\ub^*,\beta^*)\}}, \quad x \in (0,\beta^*) ,
\end{equation}
and we can show using \eqref{Aiaibi} that 
\begin{align} \label{M>0}
\begin{split}
&\frac{\tfrac12 \big(1 + (\sigma^2 - q_1) a_1 \big) x^2 - q_1 b_1 x - q_1 c_1}{\sigma^2 (-a_1) x^2} \geq 1 \\
&\Leftrightarrow \quad M(x;q_2) := \Big(\frac12 - \frac{3\sigma^2 - q_1}{q_2-q_1} \Big) x^2 - \frac{2 q_1 q_2}{q_2-q_1} x - q_1 c_1(q_2;q_1) \geq 0 , \quad x \in \big[\ub^*(q_2;q_1),\beta^*(q_2;q_1) \big).
\end{split}
\end{align} 
In view of the limit in Step 4, we can calculate the limit 
\begin{align*} 
\lim_{q_2 \to \infty}M(x;q_2) := \tfrac12 x^2 - 2 q_1 x - \big((\tfrac12 \lim_{q_2 \to \infty} \big(\ub^*(q_2;q_1)\big)^2 - 2 q_1 \lim_{q_2 \to \infty} \ub^*(q_2;q_1) \big) \geq 0 , \quad x \in \big[\lim_{q_2 \to \infty}\ub^*(q_2;q_1), \infty \big) ,
\end{align*}
whose positivity follows from the fact that $x \mapsto \lim_{q_2 \to \infty}M(x;q_2)$ is increasing thanks to $\gamma_1 = \gamma(q_1)$ in \eqref{gammaq} and the result of Lemma \ref{nec_inequality}, which imply that  
\begin{align*} 
\frac{d}{dx} \Big\{\lim_{q_2 \to \infty}M(x;q_2) \Big\} 
= x - 2 q_1 
\geq \lim_{q_2 \to \infty}\ub^*(q_2;q_1) - 2 q_1
= \frac{2 \big(q_1 - \sigma^2(\gamma_1 - 1)\big)}{\gamma_1 - 2} > 0, 
\quad x \in \big(\lim_{q_2 \to \infty}\ub^*(q_2;q_1), \infty \big) .
\end{align*}
Hence, we can conclude from the continuity of the map $q_2 \mapsto M(x;q_2)$ that that there exists $h>0$ such that $M(x;q_2) \geq 0$ for all $x \in (\ub^*,\beta^*)$ and $q_2-q_1\geq h$, which implies via \eqref{bound-u-prel} and \eqref{M>0} that
\begin{equation} \label{bound-u}
\frac{u_{\ub^*,\beta^*}(x)}{\sigma^2 x^2} 
\geq \frac{1}{\beta^* - x} {\bf 1}_{\{x \in [\ub^*,\beta^*)\}}, \quad x \in (0,\beta^*) .
\end{equation}
Substituting the lower bound obtained in \eqref{bound-u} into the definition \eqref{scale_def} of $s_{\beta^*}$ with $u_{\beta^*}(x) = u_{\ub^*,\beta^*}(x)$, $\mu(x)=0$, $\sigma(x) = \sigma x$, for all $x \in (0,\beta^*)$, 
we conclude that there exists a constant $k_1 = k_1(c)>0$ such that
\begin{align*}
\lim_{x\uparrow \beta^*} s_{\beta^*}(x) 
&= \lim_{x\uparrow \beta^*} \int_{c}^{x} \exp \bigg\{2 \int_{c}^{y} \frac{u_{\ub^*,\beta^*}(z)}{\sigma^2 z^2} dz \bigg\} \, dy \\
&\geq \int_{c \vee \ub^*}^{\beta^*} \exp \bigg\{2 \int_{c}^{y} \frac{1}{\beta^* - z} dz \bigg\} \, dy 
\geq k_1 \int_{c \vee \ub^*}^{\beta^*} e^{-2 \ln(\beta^* - y)} \, dy
= k_1 \int_{c \vee \ub^*}^{\beta^*} \frac{1}{(\beta^* - y)^2}\, dy 
= \infty.
\end{align*}
Then condition \eqref{non-exit-entrance-cond} holds true by \cite[Problem 5.27 (Section 5.5)]{Karatzas2}. 

\vspace{1mm}
Combining Steps 1--5 proves that $0< \ub^* < \beta^*$, and that $(u_{\ub^*,\beta^*},[\beta^*,\infty),\delta)$ is a mild threshold control strategy for each $\delta \in (0,\beta^*)$, which completes the proof of part (i).

\vspace{1mm}
{\it Proof of (ii).} 
Using part (i) and the arguments in Section \ref{sec:mild} showing that the associated SDE \eqref{Xubeta} has a unique strong solution for all $x \in (0,\infty)$, we can conclude from Lemma \ref{useful-lemma} that the associated limiting cost criterion $w(\cdot;q_i,u_{\ub^*,\beta^*},[\beta^*,\infty))$ defined by \eqref{wrqD} exists and is given by $w(x;q_i,u_{\ub^*,\beta^*},[\beta^*,\infty)) = v(x;q_i)$, for all $x \in (0,\infty)$ and $i=1,2$ (cf.~\eqref{w=w}), which takes the form of \eqref{candv} thanks to Proposition \ref{thm:mildvalue}.

\vspace{1mm}
{\it Proof of (iii).}
Now note that the functions in part (ii) straightforwardly satisfy \eqref{hlq3hjl}, thus Lemma \ref{useful-lemma-2} implies that the associated limiting cost criterion $J(\cdot;u_{\ub^*,\beta^*},[\beta^*,\infty))$ defined by \eqref{eq:limiting-cost-cr} exists and is given by $V(\cdot)$ defined by \eqref{J_mild_def} and satisfies \eqref{JBVPa}--\eqref{JBVPd} (specified in view of the present case study), as well as \eqref{eq_ex_1} thanks to Proposition \ref{thm:mildvalue}. 
Thanks to \eqref{JBVPd} we know that $V(\cdot)\in \mathcal{C}^2\big( (0,\infty) \setminus \{\ub^*,\beta^*\} \big) \cap \mathcal{C}^1(0,\infty)$, hence for $V(\cdot)\in\mathcal{C}^2(0,\infty)$, it remains to show that 
$$
V''(\beta^*-) = \tfrac12 (a_1 + a_2) = 0 = V''(\beta^*+)
$$
where the equalities follow by \eqref{J_mild_def}, \eqref{eq_ex_1} and \eqref{candv}, and that
\begin{align*}
V''(\ub^*-) 
&= \frac{q_1 v(\ub^*;q_1) + q_2 v(\ub^*;q_2) - (\ub^*)^2}{\sigma^2 (\ub^*)^2} \\
&= \frac{q_1 v(\ub^*;q_1) + q_2 v(\ub^*;q_2) + 2 u_{\ub^*,\beta^*}(\ub^*) (V'(\ub^*+;q_i) - 1) - (\ub^*)^2}{\sigma^2 (\ub^*)^2}
= V''(\ub^*+)
\end{align*}
where the equalities follow by adding together the ODEs \eqref{XBVPaa} for $i=1,2$, the ODEs \eqref{XBVPaaa} for $i=1,2$, and using \eqref{eq_ex_1}. 
\end{proof}

We are now ready to present our main result, according to which we have an equilibrium strategy in the form of a mild threshold control strategy, when the difference $q_2-q_1$ of the possible discount rates is sufficiently large. 
Its proof is based on an application of our general verification Theorem \ref{main_thm}.

\begin{theorem} \label{ex_mix_theorem}  
For each $q_1$, there exists a constant $h>0$ such that if $q_2-q_1\geq h$, then $(u_{\ub^*,\beta^*},[\beta^*,\infty))$ associated with \eqref{ubbeta}--\eqref{u*} is an equilibrium strategy for the time-inconsistent singular control problem defined by \eqref{TISC}. 
\end{theorem}

\begin{proof} 
Note that the SDE corresponding to the mild threshold control strategy $(u_{\ub^*,\beta^*},[\beta^*,\infty),\delta)$ has a unique strong  solution, for $\delta$ sufficiently small (cf.~Section \ref{sec:mild}).
Recall that the solution $v(\cdot;q_i)$ to the associated BVP from Lemma \ref{useful-lemma}, which takes the form \eqref{XBVPaa}--\eqref{XBVPad}, is constructed in Proposition \ref{thm:mildvalue}. 
Recall also from Proposition \ref{2141gsd} that $w(\cdot;q_i,u_{\ub^*,\beta^*},[\beta^*,\infty)) = v(\cdot;q_i)$ for each $i=1,2$, and  
that $J(\cdot;u_{\ub^*,\beta^*},[\beta^*,\infty)) = V(\cdot)$, where $V(\cdot)$ is defined in \eqref{J_mild_def}. Moreover, $V(\cdot)$ satisfies \eqref{eq_ex_1}, as well as \eqref{JBVPa}--\eqref{JBVPc} (specified in view of the present case study) and $V(\cdot) \in \mathcal{C}^2(0,\infty)$. 
Finally, note that \eqref{hlq3hjl} is immediately verified. 
Hence, all that remains is to verify that conditions \eqref{mainthmcond1}--\eqref{mainthmcond3} of the (verification) Theorem~\ref{main_thm}  hold true (recalling that $l=0<b^*=\beta^*<\infty=r$).

\vspace{1mm}
{\it Verification of condition \eqref{mainthmcond1}}. 
It is easy to see from \eqref{candv}  and \eqref{J_mild_def}, as well as \eqref{eq_ex_1}, that 
\begin{equation} \label{mV''0}
V''(0+) = \frac{1}{2(q_1-\sigma^2)} + \frac{1}{2(q_2-\sigma^2)} > 0 
\quad \text{and} \quad 
V''(\ub^*) = 0.
\end{equation}
By differentiating the ODEs \eqref{XBVPaa} with respect to $x$ for each $i=1,2$ and adding them, we obtain from \eqref{J_mild_def} that 
\begin{align*}
&2x + \sigma^2 x^2 V'''(x) + 2 \sigma^2 x V''(x) - q_1 v'(x;q_1) - q_2 v'(x;q_2) = 0, \quad x \in (0,\ub^*).
\end{align*}
Hence, given that $v'(\ub^*;q_i) = a_i \ub^* + b_i$ thanks to \eqref{XBVPad} and $V''(\ub^*)=0$ thanks to \eqref{mV''0}, we obtain from \eqref{Aiaibi} that  
$\lim_{x \uparrow \ub^*} V'''(x) 
= 0$. 
We also observe directly from \eqref{candv} that $V \in \mathcal{C}^3(0,\ub^*)$ and 
\begin{align*}
V'''(x) = 
\tfrac{1}{2} x^{\gamma_1-3} \left( A_1\gamma_1(\gamma_1-1)(\gamma_1-2) + A_2\gamma_2(\gamma_2-1)(\gamma_2-2) x^{\gamma_2-\gamma_1}\right),  \quad x \in (0,\ub^*),
\end{align*}
which implies that $V'''(\cdot)$ does not have a root in $(0,b^*)$ (to see this use the $\lim_{x \uparrow \ub^*} V'''(x) 
= 0$). Hence, using also \eqref{mV''0} we conclude that 
$V''(x)>0$, for all $x \in (0,\ub^*)$.
Combining this with the observation that $V'(\ub^*) = 1$ from \eqref{eq_ex_1}, we conclude that $V'(x) < 1$ for all $x\in(0,\ub^*)={\rm int}(\mathcal{W}_{u_{\ub^*,\beta^*}})$.

\vspace{1mm}
{\it Verification of condition \eqref{mainthmcond2}}. 
This follows directly by construction of $v(x;q_i)$ leading to \eqref{candv}, such that \eqref{eq_ex_1} holds true, and given that $V(\cdot) \in \mathcal{C}^2(0,\infty)$, we have $V'(x) = 1$ for all $x \in \ol{\mathcal{M}_{u_{\ub^*,\beta^*}}} \setminus \{0,\beta^*\} = [\ub^*,\beta^*)$. 
(Recall that in our notation $\beta^*$ takes the role of $b^*$ in \eqref{mainthmcond2}--\eqref{mainthmcond3}, cf.~Theorem \ref{main_thm}).

\vspace{1mm}
{\it Verification of condition \eqref{mainthmcond3}}.
Recall $\mu(\cdot) \equiv 0$ and $f(x)=\frac12 x^2$ and denote the left-hand side of \eqref{mainthmcond3} by
\begin{align*}
\phi(x)
& :=\tfrac{1}{2} x^2 - \tfrac{1}{2} q_1 v(x;q_1) - \tfrac{1}{2} q_2 v(x;q_2) , 
\quad x\in [\beta^*,\infty) = \cal S.
\end{align*}
Differentiating $\phi(x)$ and using the fact that $v'(x;q_i) = 1$ for $x\in[\beta^*,\infty)$ thanks to \eqref{XBVPab}, we obtain thanks to \eqref{ubbeta} that 
\begin{align*}
\phi'(x) = x - \tfrac12 (q_1+q_2) > \beta^* - \tfrac12 (q_1+q_2) = 0, 
\quad x\in (\beta^*,\infty).
\end{align*}
Combining this with 
\begin{align*}
\phi(\beta^*) 
= \tfrac{1}{2} (\beta^*)^2 - \tfrac{1}{2} q_1 v(\beta^*;q_1) - \tfrac{1}{2} q_2 v(\beta^*;q_2) 
= 0,
\end{align*}
where the latter equality follows from \eqref{candv} with \eqref{Aiaibi}--\eqref{ci} and \eqref{XBVPad}. 
We therefore conclude that $\phi(x) \geq 0$ for all $x \in [b^*,\infty)$.
\end{proof}

\bibliographystyle{abbrv}
\bibliography{NonExpDividends}

\appendix
\section{Appendix}

\begin{lemma} \label{nec_inequality} 
Define $\gamma:(0,\infty) \mapsto \R$ by
\begin{align} \label{gammaq}
\gamma(q) := \frac{1}{2} \bigg( 1 + \sqrt{1+\frac{8q}{\sigma^2}} \bigg).
\end{align} 
Then $q-(\gamma(q)-1)\sigma^2>0$, for all $q \in (\sigma^2,\infty)$.
\end{lemma}
\begin{proof}
Define $g(q) := q-(\gamma(q)-1)\sigma^2$. Note that $g(\sigma^2)=0$ and that
\begin{align*}
    g'(q)=1-\frac{2}{\sqrt{1+\frac{8q}{\sigma^2}}}>0, \enskip q\in[\sigma^2,\infty).
\end{align*}
We conclude that $g(q)>0$ for $q\in (\sigma^2,\infty)$.
\end{proof}

\begin{proof}[\bf Proof of Lemma \ref{scale_properties_lemma}]
Recalling from Assumption \ref{Ass:bounds} that $\mu$ and $\sigma$ are Lipschitz continuous and that $\sigma$ is bounded away from $0$, as well as from the Definition \ref{def:mixed0} that $u_\beta \geq 0$, 
we can see that there exists a constant $K_0>0$ such that
\begin{align*}
    \int_c^{\beta}\frac{2s_\beta(x)}{s_\beta'(x)\sigma^2(x)}dx&\leq K_0\int_c^\beta\frac{\int_c^x\exp \left\{ 2\int_{c}^{y}u_\beta(z) dz \right\} \, dy}{\exp \left\{ 2\int_{c}^{x}u_\beta(z) dz \right\}}dx 
    \leq K_0 \int_c^\beta(x-c)\,dx<\infty ,
\end{align*}
which proves the first statement. 

Using once again that $\mu$ and $\sigma$ are Lipschitz continuous and $u_\beta \geq 0$, we can show that that there exists a constant $K_1>0$, such that
\begin{align*}
\int_c^x s'_\beta(y)\int_c^y \frac{2dz}{s'_\beta(z)\sigma^2(z)} dy 
&\leq K_1 \int_c^x s'_{\beta}(y)\int_{c}^{y}\frac{dz}{\exp \left\{ 2\int_{c}^{z}u_\beta(\zeta) d\zeta \right\}}dy\\
&\leq K_1 (x-c) \int_c^x s'_{\beta}(y)dy 
= K_1 (x-c) s_{\beta}(x).
\end{align*}
Thus taking the limits as $x \to \beta$ on both sides and using \eqref{non-exit-entrance-cond} implies that $\lim_{x\uparrow\beta}s_{\beta}(x)=\infty$, which proves the second claim. 

The third claim that $\lim_{x\uparrow \beta} u_\beta(x) = \infty$ is then an immediate consequence of $\lim_{x\uparrow\beta}s_{\beta}(x)=\infty$, definition \eqref{scale_def}, and the Lipschitz continuity of $\mu$ and $\sigma$. 

Finally, a combination of \eqref{non-exit-entrance-cond} with the first part of \eqref{non-exit-entrance-cond2} yields from \cite[Section II.1.6]{borodin2012handbook} that $\beta$ is an entrance-not-exit boundary for the associated process $X^D$. 
\end{proof}

\begin{proof}[\bf Proof of Proposition \ref{prop_limit_ex}]
We prove this result in the following five steps. 

\vspace{1mm}
{\it Step 1}. 
Define $g(y) := \int_{a}^{y} \exp \left\{\int_{c}^{\kappa} u_\beta(z) dz \right\} d\kappa$. 
Then using that $\mu$ and $\sigma$ are Lipschitz continuous, $\sigma$ is bounded away from $0$ (Assumption \ref{Ass:bounds}), and $u_\beta \geq 0$ (see Definitions \ref{def:mixed0} and \ref{def:mixedm}) implying that $g(\cdot)$ is increasing, 
we can see that there exists a constant $K_0>0$ such that
\begin{align*}
\int_a^x \frac{s_\beta(y)-s_\beta(a)}{\sigma^2(y)s_\beta'(y)}u_\beta(y) dy 
\leq K_0 \int_a^x\frac{g(y)}{g'(y)}u_\beta(y)dy ,
\end{align*}
where $s_\beta$ is the scale function defined by \eqref{scale_def}. 
By straightforward differentiations of $g$, we    notice that $u_\beta(x)g'(x)=g''(x)$, hence by the integration by parts formula the above inequality further yields that 
\begin{align*}
\int_a^x \frac{s_\beta(y)-s_\beta(a)}{\sigma^2(y)s_\beta'(y)}u_\beta(y) dy 
\leq K_0 \int_a^x g(y) \frac{g''(y)}{(g'(y))^2}dy 
= - K_0 \frac{g(x)}{g'(x)} + K_0 (x-a)
\leq K_0 (x-a).
\end{align*}
Taking the limits as $x\to\beta$ on both sides then implies that, for any $a\in (l,\beta)$, we have \begin{align} \label{help_lemma_limit}
\int_a^\beta\frac{s_\beta(y)-s_\beta(a)}{\sigma^2(y)s_\beta'(y)}u_\beta(y)dy<\infty.
\end{align}

\vspace{1mm}
{\it Step 2}. 
We fix a point $a<\beta$ (close to $\beta$) which is such that $u_\beta(\cdot)$ is continuous in $(a,\beta)$ and $(a,\beta) \cap \cal S=\emptyset$ (such a point exists thanks to Definition \ref{def:mixedm}), and we define for $a<b<\beta$ the function
\begin{align} \label{wab}
\ol{w}_{a,b}(x) 
:= \E_x \left[ D^{u_\beta,S,\delta}_{\zeta^D_a\wedge\zeta^D_b} \right] 
= \E_x \bigg[ \int_0^{\zeta^D_a\wedge\zeta^D_b} u_\beta(X_t^{u_\beta,S,\delta}) dt \bigg] , 
\quad x\in(a,b),
\end{align}
where $\zeta^D_{\cdot}$ is defined by \eqref{zeta} and the latter equality follows due to $x \in (a,b) \subseteq \mathcal{M} \cup \mathcal{W}$. 
Using eq.~(3.11) in \cite[Chapter~15.3]{karlin1981second}
we find, for $x\in (a,b)$, that 
\begin{align*}
\ol w_{a,b}(x)
= 2 \bigg(\frac{s_\beta(x)-s_\beta(a)}{s_\beta(b)-s_\beta(a)}\int_x^b\frac{s_\beta(b)-s_\beta(y)}{\sigma^2(y)s_\beta'(y)}u_\beta(y)dy+ \left(1-\frac{s_\beta(x)-s_\beta(a)}{s_\beta(b)-s_\beta(a)}\right)
\int_a^x\frac{s_\beta(y)-s_\beta(a)}{\sigma^2(y)s_\beta'(y)}u_\beta(y)dy\bigg).
\end{align*}
Since $s_\beta(\cdot)$ is increasing, we use the fact  that 
$s_\beta(b)-s_\beta(y) \leq s_\beta(b)-s_\beta(a)$  
and $s_\beta(y)-s_\beta(a) \geq s_\beta(x)-s_\beta(a)$ for all $y>x>a$, 
together with $u_\beta(\cdot) \geq 0$, to obtain 
\begin{align*}
\ol{w}_{a,b}(x) 
\leq 2 \bigg(\int_x^b \frac{s_\beta(y)-s_\beta(a)}{\sigma^2(y)s_\beta'(y)} u_\beta(y) dy 
+ \left(1-\frac{s_\beta(x)-s_\beta(a)}{s_\beta(b)-s_\beta(a)}\right) \int_a^x\frac{s_\beta(y)-s_\beta(a)}{\sigma^2(y)s_\beta'(y)} u_\beta(y) dy \bigg).
\end{align*}
It thus follows from \eqref{help_lemma_limit} and the fact that $s_\beta(\cdot)$ is increasing, that
\begin{align} \label{wabeta}
\begin{split}
\ol{w}_{a,\beta}(x) 
:= \lim_{b\uparrow\beta} \ol{w}_{a,b}(x) 
&\leq 2 \bigg(\int_x^\beta \frac{s_\beta(y)-s_\beta(a)}{\sigma^2(y)s_\beta'(y)} u_\beta(y) dy 
+ \int_a^x \frac{s_\beta(y)-s_\beta(a)}{\sigma^2(y)s_\beta'(y)} u_\beta(y) dy \bigg) \\
&= 2 \int_a^\beta \frac{s_\beta(y)-s_\beta(a)}{\sigma^2(y)s_\beta'(y)} u_\beta(y) dy < \infty 
\quad \Rightarrow \quad 
\lim_{x \uparrow \beta} \ol{w}_{a,\beta}(x) < \infty.
\end{split}
\end{align}

\vspace{1mm}
{\it Step 3}. 
Using the strong Markov property, Lemma \ref{scale_properties_lemma} which implies that $\beta$ is entrance-not-exit (i.e. $\zeta^D_\beta=+\infty$ and thus $\zeta^D_a \wedge \zeta^D_\beta = \zeta^D_a$), Assumption \ref{Ass:f}, \eqref{wabeta}, and $w(x;q,u_\beta,\mathcal{S},\delta)<\infty$ for $x\in (l,\beta)$, 
we conclude from \eqref{wrd} that 
\begin{align*}
\limsup_{x\uparrow \beta} w(x;q,u_\beta,\mathcal{S},\delta) 
&\leq \limsup_{x\uparrow \beta} \E_x\bigg[ \int_0^{\zeta^D_a} e^{-qt} ( f + u_\beta)(X^{u_\beta,\mathcal{S},\delta}_t) dt 
+ e^{-r \zeta^D_a} w(a;q,u_\beta,\mathcal{S},\delta) \bigg] \\
&\leq
\frac{f(\beta)}{q}+
\limsup_{x\uparrow \beta} \ol{w}_{a,\beta}(x) + 
w(a;q,u_\beta,\mathcal{S},\delta) 
<\infty.
\end{align*}
Hence, the limit $\lim_{x\uparrow \beta}w(x;q,u_\beta,\mathcal{S},\delta)$ is finite, if it exists. 

\vspace{1mm}
{\it Step 4}.
To prove the existence of the limit as required at the end of Step 3, let $x\in ( a \vee (\beta-\epsilon),\beta)$ and use the strong Markov property, and that $\beta$ is entrance-not-exit (thanks to Lemma \ref{scale_properties_lemma}) to obtain 
\begin{align*}
w(x;q,u_\beta,\mathcal{S},\delta)
= \E_x \bigg[ \int_0^{\tau^D_{\beta-\epsilon}} e^{-qt} (f + u_\beta)(X^{u_\beta,\mathcal{S},\delta}_t) dt \bigg] 
+ w(\beta-\epsilon;q,u_\beta,\mathcal{S},\delta) \E_x\left[ e^{-q \tau^D_{\beta-\epsilon}} \right].
\end{align*}
Using that $u_\beta + f$ is increasing on $(\beta-\epsilon,\beta)$, we can easily verify that the first expectation is increasing in $x$ (e.g.~by comparison principle). Moreover, the second expectation is bounded and decreasing in $x$. We conclude that the limit $\lim_{x\uparrow \beta}w(x;q,u_\beta,\mathcal{S},\delta)$ exists, and by Step 3 it is thus also finite. 

\vspace{1mm}
{\it Step 5}. Using the definition of a generalised threshold control strategy we find 
\begin{align*}
w(x;q,u_\beta,\mathcal{S},\delta)=x-(\beta-\delta) + w(\beta-\delta;q,u_\beta,\mathcal{S},\delta), \enskip x \geq \beta.
\end{align*}
Hence,
\begin{align*}
\lim_{\delta\downarrow 0} w(x;q,u_\beta,\mathcal{S},\delta)=x-\beta + \lim_{\delta\downarrow 0} w(\beta-\delta;q,u_\beta,\mathcal{S},\delta),
\end{align*}
which exists and is finite thanks to Step 4, and completes the proof.
\end{proof}

\end{document}